\documentclass{amsart}

\usepackage{mathrsfs}
\usepackage{amscd}
\usepackage{amsmath}
\usepackage{amssymb}
\usepackage{amsthm}
\usepackage{epsf}
\usepackage{latexsym}
\usepackage{verbatim}
\usepackage[all, cmtip]{xy}
\usepackage{tikz}
\usetikzlibrary{positioning}
\usetikzlibrary{matrix}
\usepackage{float}
\usepackage{hyperref}
\usepackage{comment}

\tikzstyle{bsq}=[rectangle, draw, thick, minimum width=.5cm, minimum height=.5cm]
\tikzstyle{bver}=[rectangle, draw, thick, minimum width=1cm, minimum height=2cm]
\tikzstyle{bhor}=[rectangle, draw, thick, minimum width=2cm, minimum height=1cm]

\usepackage[left=3.2cm, right=3.2cm]{geometry}

\newtheorem{theorem}{Theorem}[section]
\newtheorem{definition}[theorem]{Definition}
\newtheorem{lemma}[theorem]{Lemma}

\newtheorem{corollary}[theorem]{Corollary}
\newtheorem{proposition}[theorem]{Proposition}

\newtheorem{varexample}[theorem]{Example}

\newtheorem*{SMRC}{Strong Maximal Rank Conjecture}
\newtheorem*{ShapeLemma}{Shape Lemma for Minima}

\theoremstyle{definition}
\newtheorem{remark}[theorem]{Remark}

\newcommand{\PP}{\mathbb{P}}

\newcommand{\RR}{\mathbb{R}}

\newcommand{\M}{\overline{M}}

\newcommand{\cL}{\mathcal{L}}

\newcommand{\cA}{\mathcal{A}}
\newcommand{\cB}{\mathcal{B}}

\newcommand{\ord}{\operatorname{ord}}
\newcommand{\Trop}{\operatorname{Trop}}
\newcommand{\trop}{\operatorname{trop}}
\newcommand{\ddiv}{\operatorname{div}}

\newcommand{\PL}{\operatorname{PL}}

\newcommand{\val}{\operatorname{val}}

\newcommand{\Sym}{\operatorname{Sym}}

\newcommand{\an}{\mathrm{an}}

\newenvironment{example}{\begin{varexample}
\begin{normalfont}}{\end{normalfont}
\end{varexample}}

\begin{document}
\title{On the Strong Maximal Rank Conjecture in genus 22 and 23}
\author{David Jensen}
\author{Sam Payne}
\date{}
\bibliographystyle{alpha}

\vspace{-10 pt}

\begin{abstract}
We develop new methods to study tropicalizations of linear series and show linear independence of sections.  Using these methods, we prove two new cases of the strong maximal rank conjecture for linear series of degree 25 and 26 on curves of genus 22 and 23, respectively.
\end{abstract}

\maketitle

\setcounter{tocdepth}{1}
\tableofcontents

\vspace{-15 pt}

\section{Introduction}

We use tropical independence on chains of loops to prove the following:

\begin{theorem}
\label{Thm:MainThm}
Let $\rho = 1$ or $2$ and let $X$ be a general curve of genus $21+\rho$.  Then the map
\[
\mu_2 : \Sym^2 H^0(X,D_X) \to H^0(X,2D_X)
\]
is injective for \emph{all} divisors $D_X$ of degree $24 + \rho$ and rank $6$ on $X$.
\end{theorem}


\noindent Along the way, we develop new techniques for understanding the tropicalization of a not necessarily complete linear series and an effective criterion for verifying tropical independence.




The condition on the curve is Zariski open, so Theorem~\ref{Thm:MainThm} is equivalent to the existence of a Brill-Noether general curve $X$ of genus $21 + \rho$ such that the image of every map $X \rightarrow \PP^6$ of degree $24 + \rho$ is not contained in a quadric.  The analogous statement for $\rho = 0$ follows from the maximal rank theorem, but the cases $\rho =1$ and $2$ are new.    The main additional difficulty is the need to deal with \emph{all} linear series of the given degree and rank,
which form a family of dimension $\rho$. 

As discussed in Section~\ref{sec:mrc}, this theorem confirms two cases of the strong maximal rank conjecture.  These particular cases are steps forward in a program initiated and developed by Farkas to show that $\M_{22}$ and $\M_{23}$ are of general type.  As a first step, Farkas computed the classes of one virtual divisor on each of these moduli spaces.  These are the virtual fundamental classes of loci of curves with a map to $\PP^6$ of degree $25$ and $26$, respectively, such that the image is contained in a quadric.  
In characteristic zero, $\M_{22}$ and $\M_{23}$ are of general type provided that these classes are represented by effective divisors \cite{Farkas09c, Farkas18}.  Our contribution is to show that each locus in question is not the whole moduli space, i.e., we prove the existence of a single curve in each genus that admits no such map.  To complete Farkas's program, it remains to show that the relevant loci of linear series giving rise to maps to $\PP^6$ with image contained in a quadric is generically finite over each divisorial component of its image in $\M_{22}$ and $\M_{23}$. 
 Note that, while the program envisioned by Farkas requires characteristic zero, our tropical proof of Theorem~\ref{Thm:MainThm} is characteristic free and works over any algebraically closed field.

\medskip

We follow a strategy based on tropical independence, as in our proof of the maximal rank conjecture for quadrics \cite{MRC}.  We let $\Gamma$ be a chain of loops with edge lengths satisfying conditions specified in Section~\ref{sec:chainsofloops} and consider a curve $X$ over a nonarchimedean field whose skeleton is $\Gamma$.  We then verify injectivity of $\mu_2$ by proving the existence of sufficiently many sections in the image whose tropicalizations are tropically independent on $\Gamma$.  The latter statement is new also in the case $\rho = 0$.

\begin{theorem} \label{thm:independence}
Let $\rho = 0$, $1,$ or $2$ and let $X$ be a curve of genus $21 + \rho$ over a nonarchimedean field whose skeleton is a chain of loops $\Gamma$ satisfying the edge length conditions specified in Section~\ref{sec:chainsofloops}.  Let $D_X$ be a divisor of degree $24 + \rho$ and rank $6$ on $X$.  Then there are 28 tropically independent functions on $\Gamma$ of the form $\trop(f) + \trop(g)$ with $f$ and $g$ in $\cL(D_X)$.
\end{theorem}

\noindent  Here $28 = \binom{8}{2}$ is the dimension of $\Sym^2 H^0(X, D_X)$, and $\trop(f) + \trop(g)$ is the tropicalization of $\mu_2(f \otimes g)$.  Rational functions are linearly independent if their tropicalizations are tropically independent \cite[Lemma~3.2]{tropicalGP}, so Theorem~\ref{Thm:MainThm} follows as an immediate consequence.

\begin{remark}
The locus of curves $X$ that satisfy the conclusion of Theorem~\ref{Thm:MainThm} is Zariski open and defined over the integers.  Therefore, by standard arguments in algebraic geometry (as discussed, e.g., in \cite[Section~3]{tropicalBN}), to prove that this locus contains all sufficiently general points defined over a given algebraically closed field, it suffices to prove nonemptiness over some other field of the same characteristic.  In particular, even though we work over a specially chosen nonarchimedean field (of arbitrary characteristic) to prove Theorem~\ref{thm:independence}, we can immediately deduce Theorem~\ref{Thm:MainThm} over any algebraically closed field including, e.g., over the complex numbers or the algebraic closure of a finite field.\end{remark}

\noindent We briefly discuss how the arguments here differ from those in our proof of the maximal rank conjecture for quadrics.  For any divisor $D = \Trop(D_X)$ of rank $6$ on $\Gamma$ whose class is vertex-avoiding in the sense of \cite{LiftingDivisors}, we have a canonical and well-studied collection of functions $\{ \varphi_0, \ldots, \varphi_6 \}$ in $R(D)$ that is the tropicalization of a basis for $\cL(D_X)$.  In \cite{MRC}, we fix one particular  $D$, assume that there is a tropical dependence among the pairwise sums $\varphi_i + \varphi_j$, determine the degree of the divisor in $R(D)$ given by this dependence, and derive a contradiction.  There are several difficulties in extending this approach to \emph{all} divisors of degree $d$ and rank $r$.  One is sheer combinatorial complexity.  Our arguments in \cite{MRC} are specific to the combinatorial type of $D$.  When $\Gamma$ has genus 23, the number of combinatorial types of vertex avoiding divisor classes in $W^6_{26} (\Gamma)$ is
\[
\frac{23!}{9!8!7!} = 350,574,510.
\]
This difficulty is overcome primarily through a new method for proving tropical independence using special tropical linear combinations, which we call \emph{maximally independent combinations}, introduced in Section~\ref{sec:configurations}.  We present an algorithm for constructing maximally independent combinations for vertex avoiding divisor classes in Section~\ref{Sec:VertexAvoiding}.  Such classes are an open dense subset of $W^6_{26}(\Gamma)$.

In the non vertex-avoiding cases, we face the additional problem of understanding which functions in $R(D)$ are tropicalizations of functions in $\cL (D_X)$, and finding a suitable substitute for the distinguished functions $\varphi_i$.  For an arbitrary divisor $D$, this seems to be an intractable problem.  However, when $\rho$ is small, we find that in most cases it is enough to understand the tropicalizations of certain pencils in $\cL(D_X)$.  These, in turn, behave similarly to tropicalizations of pencils on $\PP^1$, which we analyze in Section~\ref{Sec:Pencil}.  The possibilities for the tropicalizations of $\cL(D_X)$ are then divided into cases, according to the combinatorial properties of these pencils.  We then construct maximally independent combinations case-by-case, in Sections~\ref{Sec:Construction}-\ref{Sec:Generic}, using a generalization of the algorithm that works for vertex avoiding divisors.  Only one subcase, treated in Section~\ref{sec:4d}, does not reduce to an analysis of pencils; the arguments in this subcase are nevertheless of a similar flavor, with just a few more combinatorial possibilities to consider.

\medskip

We note that the maximally independent combinations, which we use to prove tropical independence, have the following natural interpretation in terms of semi\-stable models.  Let $f_0, \ldots, f_r$ be sections of a line bundle $\mathcal{O}(D_X)$ on a curve $X$ over a discretely valued field with a regular semistable model $\mathcal X$.  Then there is a maximally independent combination of $\{ \trop(f_0), \ldots, \trop(f_r) \}$ on the dual graph of the special fiber of $\mathcal X$ if and only if there is a line bundle $\mathcal{L}$ on $\mathcal{X}$ extending $\mathcal{O}(D_X)$, scalars $a_0, \ldots, a_r$, and components $X_0, \ldots, X_r$ of the special fiber of $\mathcal{X}$ such that $a_i f_i$ extends to a regular section of $\mathcal{L}$, and $a_i f_i |_{X_j} \neq 0$  if and only if $i = j$.  Since each section $a_i f_i$ is nonvanishing on a component of the special fiber where all others vanish, this collection must be linearly independent on the special fiber, and hence also on the general fiber.  See Proposition~\ref{prop:ag-interpretation}.

\medskip

The notion that degeneration methods including tropical geometry and limit linear series could be used to prove these cases of the strong maximal rank conjecture 
has been circulating among experts for several years.  Indeed, this was among the central topics discussed at both a BIRS conference on specialization of linear series for tropical and algebraic curves in 2014 as well as an AIM workshop on degenerations in algebraic geometry in 2015.  We completed our proof in late 2017 and first announced the result publicly at GAGS in February 2018.  Jensen later presented this work at a Tufts conference on the maximal rank conjecture in March 2018, with audience members including Osserman and Teixidor.  In May 2018, Osserman informed us that, in collaboration with Liu, Teixidor, and Zhang, he had found an alternate proof in characteristic zero using limit linear series, to appear in \cite{LiuOssermanTeixidorZhang18}.

\medskip

\noindent{\textbf{Acknowledgments.}}  The work of DJ is partially supported by NSF DMS--1601896.  DJ would also like to thank Yale University for hosting him during the Summer and Fall of 2017, during which time the majority of the work on this paper was completed.  The work of SP is partially supported by NSF DMS--1702428.

\section{Preliminaries}
\label{Sec:Prelim}

\subsection{Maximal rank theorems and conjectures}  \label{sec:mrc} One natural invariant of a linear series $V \subset H^0(X,L)$ on a curve $X$ is the Hilbert function $h_V$.  Recall that $h_V (m)$ is defined to be the rank of the linear map
\[
\mu_m \colon \Sym^m V \rightarrow H^0 (X,L^{\otimes m}) .
\]
The Hilbert function of a general linear series on a general curve is predicted by a well-known conjecture known as the maximal rank conjecture, now a theorem of Larson \cite{Larson17}, which says that the multiplication map $\mu_m : \Sym^m H^0(X, D_X) \rightarrow H^0(X, mD_X)$ has maximal rank (i.e., is either injective or surjective) when $X$ is a general curve of genus $g$, and $D_X$ is a general divisor of degree $d$ and rank $r$ \cite{Harris82}.  

For $X$ general, and in the range of cases where the Brill-Noether number $\rho(g,r,d)$ is at most $r -2$, the strong maximal rank conjecture makes a further prediction for the dimension of the locus of divisor classes of degree $d$ and rank $r$ for which $\mu_m$ fails to have maximal rank.

\begin{SMRC} \cite{AproduFarkas11}
Fix positive integers $g$, $r$, and $d$ such that $g-d+r \geq 0$ and $0 \leq \rho (g,r,d) \leq r-2$.  For a general curve $X$ of genus $g$, the determinantal variety
\[
\Sigma^r_{d,m} (X) = \{ D \in W^r_d (X) \ \vert \ \mu_m \text{ does not have maximal rank} \}
\]
is of expected dimension
\[
\rho (g,r,d)-1 - \Big \lvert md-g+1-{{r+m}\choose{m}} \Big \rvert.
\]
\end{SMRC}

\noindent Our second main result Theorem~\ref{Thm:MainThm} verifies the strong maximal rank conjecture for $(g,r,d) = (22,6,25)$ and $(23,6,26)$.

\subsection{Tropical and nonarchimedean geometry}
We briefly recall basic facts to be used throughout the paper.  Let $X$ be a curve of positive genus over an algebraically closed nonarchimedean field $K$ with valuation ring $R$ and residue field $\kappa$.  For simplicity, we assume that $K$ is spherically complete with value group $\RR$, so every point in the nonarchimedean analytification $X^\mathrm{an}$ has type 1 or 2.  Here, a type 1 point is simply a $K$-rational point, and a type 2 point $v$ corresponds to a valuation $\val_v$ on the function field $K(X)$ whose associated residue field is a transcendence degree 1 extension of the residue field $\kappa$ of $K$.  See \cite{BPR-section5, BPR16} for details on the structure theory of curves over nonarchimedean fields, relations to tropical geometry, and proofs of the basic properties of analytification and tropicalization that we omit.

\subsubsection{Skeletons} The \emph{minimal skeleton} of $X^\an$ is the set of points with no neighborhood isomorphic to a ball, and carries canonically the structure of a finite metric graph.  More generally, a \emph{skeleton}  for $X^\an$ is the underlying set of a finite connected subgraph of $X^\an$ that contains this minimal skeleton.  Any skeleton $\Gamma$ is contained in the set of type 2 points, and any decomposition of a skeleton $\Gamma$ into vertices and edges determines a semistable model of $X$ over $R$.  The vertices correspond to the irreducible components of the special fiber, and the irreducible component $X_v$ corresponding to $v$ has function field $\kappa(X_v)$, the transcendence degree 1 extension of $\kappa$ given by the residue field of $K(X)$ with respect to $\val_v$.  The edges correspond to the nodes of the special fiber, with the length of each edge given by the thickness of the corresponding node.

\subsubsection{Tropicalizations and reductions of rational functions} \label{sec:slopesandreductions} Let $\Gamma$ be a skeleton for $X^\an$.   Since $\Gamma$ is contained in the set of type 2 points, for each nonzero rational function $f \in K(X)^*$ we get a real-valued function
\[
\trop(f):\Gamma \rightarrow \RR, \ \ v \mapsto \val_v(f).
\]
This function is piecewise linear with integer slopes, and the slopes along edges incident to $v$ are related to the divisors of the reduction of $f$ at $v$, via the slope formula, a nonarchimedean analogue of the Poincar\'e-Lelong formula, as follows.

Given a nonzero rational function $f \in K(X)$ and a type 2 point $v \in X^\an$, choose $c \in K^*$ whose valuation is equal to $\val_v(f)$.  Then $f/c$ has valuation zero, and the \emph{reduction} of $f$ at $v$, denoted $f_v$, is defined to be its image in $\kappa (X_v)^*/\kappa^*$.  This does not depend on the choice of $c$, so $f_v$ is well-defined.  Divisors of rational functions are invariant under multiplication by nonzero scalars, and we denote the divisor on $X_v$ of any representative of $f_v$ in $\kappa(X_v)^*$ by $\ddiv (f_v)$.  Each germ of an edge of $\Gamma$ incident to $v$ corresponds to a point of $X_v$ (a node in the special fiber of a semistable model with skeleton $\Gamma$, in which $X_v$ appears as a component).  The slope formula then says that the outgoing slope of $\trop(f)$ along this germ of an edge is equal to the order of vanishing of $f_v$ at that point \cite[Theorem~5.15(c)]{BPR-section5}.

\subsubsection{Complete linear series on graphs} Let $\PL(\Gamma)$ be the set of piecewise linear functions on $\Gamma$ with integer slopes.  Throughout, we will use both the additive group structure on $\PL(\Gamma)$, and the tropical module structure given by addition of real scalars and pointwise minimum.

Given $v \in \Gamma$ and $\varphi \in \PL (\Gamma)$, the order of $\varphi$ at $v$, denoted $\ord_v (\varphi)$, is the sum of the incoming slopes of $\varphi$ at $v$.  The principal divisor associated to $\varphi$ is then $\ddiv (\varphi) = \sum_{v \in \Gamma} \ord_v (\varphi) v$.  The \emph{complete linear series} of a divisor $D$ on $\Gamma$ is
\[
R(D) = \{ \varphi \in \PL (\Gamma) \, \vert \, \ddiv (\varphi) + D \geq 0 \}.
\]
Note that $R(D) \subset \PL(\Gamma)$ is a tropical submodule, i.e., it is closed under scalar addition and pointwise minimum.

By the Poincar\'{e}-Lelong formula, if $D_X$ is any divisor on $X$ tropicalizing to $D$ and $f$ is a section of $\mathcal{O}(D_X)$, then $\trop(f) \in R(D)$.  We refer the reader to \cite{BakerNorine07, Baker08} for further background on the divisor theory of graphs and metric graphs, and specialization from curves to graphs.

\subsubsection{Tropical independence}  A key idea for many of our arguments is that of tropical independence, defined in \cite{tropicalGP}.  We recall this definition here.

\begin{definition}
We say that $\{\psi_0 , \ldots \psi_r\} \subset \PL (\Gamma)$ is \emph{tropically dependent} if there are coefficients $c_0 , \ldots , c_r$ in $\RR$ such that the minimum of the functions $\psi_i + c_i$ is achieved at least twice at every point of $\Gamma$. If no such coefficients exist, we say that $\{\psi_0 , \ldots \psi_r\}$ is  \emph{tropically independent}.
\end{definition}

\noindent If a set of nonzero functions $\{f_0, \ldots , f_r\}$ is linearly dependent over $K$, then the set of tropicalizations $\{ \trop(f_0), \ldots , \trop (f_r) \}$ is tropically dependent \cite[Lemma~3.2]{tropicalGP}.  Therefore, tropical independence of the tropicalizations is a sufficient condition for linear independence of rational functions.

\section{Tropicalizations of linear series}
\label{Sec:Slopes}

In this section, we discuss properties of tropicalizations of not necessarily complete linear series, when the skeleton is an arbitrary tropical curve.  In the remainder of the paper, we will apply these results in the special case where the skeleton is a chain of loops, but the results of this section will be useful more generally.  We begin by discussing tropicalizations and reductions of linear series, in the spirit of \cite{AminiBaker15}.

\subsection{Tropicalizations and reductions of linear series}
Let $X$ be a curve over $K$ with a divisor $D_X$ and a skeleton $\Gamma \subset X^{\an}$.  Let $V \subseteq \cL (D_X)$ be a linear series of rank $r$, and let $\Sigma = \trop (V)$, the set of tropicalizations of nonzero rational functions in $V$.  Note that $\Sigma \subset \PL(\Gamma)$ is a tropical submodule.

A \emph{tangent vector} in $\Gamma$ is a germ of a directed edge.  Given a function $\psi \in \PL(\Gamma)$,  we write $s_\zeta(\psi)$ for the slope of $\psi$ along a tangent vector $\zeta$.

\begin{lemma}
\label{Lem:NumberOfSlopes}
For each tangent vector $\zeta$ in $\Gamma$ there are exactly $r+1$ different slopes $s_\zeta (\psi)$, as $\psi$ ranges over $\Sigma$.
\end{lemma}

\begin{proof}
Suppose $\zeta$ is based at the point $v \in \Gamma$. By the nonarchimedean Poincar\'{e}-Lelong formula, the slope $s_\zeta (\trop (f))$ is equal to the order of vanishing of $f_v$ at the point corresponding to the incoming tangent direction.  Since the reductions form a vector space of dimension $r+1$ over $\kappa$ \cite[Lemma~4.3]{AminiBaker15}, they have exactly $r+1$ different orders of vanishing.
\end{proof}

When $\Sigma$ is fixed and no confusion is possible, we write
\[
s_\zeta(\Sigma) = (s_\zeta [0], \ldots, s_\zeta [r])
\]
for the vector of slopes $s_\zeta(\psi)$, for $\psi \in \Sigma$, ordered so that  $s_\zeta [0] <  \cdots < s_\zeta[r]$.

\begin{remark}
If $D = \Trop(D_X)$ then the tropical complete linear series $R(D)$ often has far more than $r(D) + 1$ slopes along some tangent vectors.  In such cases, Lemma~\ref{Lem:NumberOfSlopes} shows that the tropicalization of the complete algebraic linear series $\cL (D_X)$ is properly contained in  $R(D)$.
\end{remark}

Given any two tangent vectors, there is a function in $\Sigma$ with complementary lower and upper bounds on its slopes in these directions, as follows.

\begin{lemma}  \label{Lemma:Existence}
For any pair of tangent vectors $\zeta$ and $\zeta'$, and for any $i$, there is a function $\psi \in \Sigma$ such that
\[
 s_\zeta (\psi) \leq s_\zeta [i] \mbox{ \ and \ } s_{\zeta'} (\psi) \geq s_{\zeta'} [i].
\]
\end{lemma}

\begin{proof}
Let $f_0 , \ldots , f_i \in V$ be functions satisfying $s_\zeta (\trop (f_i)) = s_\zeta [i]$.  Because the functions $\trop (f_i)$ have distinct slopes, the reductions of $f_i$ have distinct orders of vanishing at the point corresponding to $\zeta$, and are therefore linearly independent over $\kappa$.  It follows that the functions $f_i$ are themselves linearly independent over $K$ and span a linear subseries $V'$ of rank $i$.  The slopes $s_{\zeta'} (\trop (f))$ for $f \in V'$ therefore take on $i+1$ distinct values in $\{ s_{\zeta'}[0], \ldots, s_{\zeta'}[r] \}$, and some $f$ in $V'$ satisfies $s_{\zeta'} (\trop(f)) \geq s_{\zeta'} [i]$, as required.  \end{proof}

While Lemma~\ref{Lem:NumberOfSlopes} can be roughly interpreted as providing an upper bound for how large $\Sigma$ can be, Lemma~\ref{Lemma:Existence} provides a lower bound.  We use both of these lemmas extensively throughout the paper.

\medskip

The following proposition provides another natural lower bound for how small $\Sigma$ can be.  This result is not used in the proofs of our main theorems, but the statement is so natural, especially in view of \cite{BakerNorine07}, that it must be seen as one of the fundamental properties of the tropicalization of a linear series.

\begin{proposition}
For any effective divisor $E$ on $\Gamma$ of degree $r$, there is some $\psi \in \Sigma$ such that $\ddiv (\psi) + D - E$ is effective.
\end{proposition}

\begin{proof}
We follow the standard argument for showing that the rank of the tropicalization of a divisor is greater than or equal to its rank on the algebraic curve.  Let $E_X$ be a divisor of degree $r$ on $X$ that specializes to $E$.  Since $V$ has rank $r$, there is a function $f \in V$ such that $\ddiv (f) + D_X - E_X$ is effective.  Setting $\psi = \trop (f)$ yields the result.
\end{proof}

\subsection{Maximally independent combinations}  \label{sec:configurations}
Let $\{ \psi_i \ | \ i \in I\}$ be a finite collection of piecewise linear functions.  A \emph{tropical linear combination} is an expression
\[
\theta = \min \{ \psi_i + c_i \},
\]
for some choice of real coefficients $\{c_i\}$.  Note that different choices of coefficients may yield the same pointwise minimum, but we consider the coefficients $\{ c_i \}$ to be part of the data in a tropical linear combination.  This means that the tropical linear combinations of  $\{ \psi_i \}$  are naturally identified with $\RR^I$.

Given a tropical linear combination $\theta = \min \{ \psi_i + c_i \}$, we say that $\psi_i$ \emph{achieves the minimum} at $v \in \Gamma$ if $\theta(v) = \psi_i(v) + c_i$, and \emph{achieves the minimum uniquely} if, moreover, $ \theta(v)\neq \psi_j(v) + c_j $ for $j \neq i$.  We say that \emph{the minimum is achieved at least twice} at $v \in \Gamma$ if there are at least two distinct indices $i \neq j$ such that $\theta(v)$ is equal to both $\psi_i(v) + c_i$ and $\psi_j(v) + c_j$.  A tropical linear combination is a \emph{tropical dependence} if the minimum is achieved at least twice at every point.  Equivalently, $\theta = \min \{ \psi_i + c_i \}$ is a tropical dependence if  $\theta = \min_{j \neq i} \{ \psi_j + c_j \}$, for all $i$.

We will be most interested in tropical linear combinations that are as far from tropical dependence as possible, in the following sense.

\begin{definition}
A tropical linear combination $\theta = \min \{ \psi_i + c_i \}$ is \emph{maximally independent} if each $\psi_i$ achieves the minimum uniquely at some point $v \in \Gamma$.
\end{definition}

\noindent Equivalently, $\theta = \min \{ \psi_i + c_i \}$ is maximally independent if $\theta \neq \min_{j \neq i} \{ \psi_j + c_j \}$, for all $i \in I$.  Throughout, we will use the following proposition as a criterion for tropical independence.

\begin{proposition}
\label{Prop:Strategy}
Let $\{ \psi_i \, | \, i \in I \}$ be a finite subset of $\PL(\Gamma)$.  If there is a maximally independent tropical linear combination $\theta = \min \{ \psi_i + c_i \}$ then $\{ \psi_i \}$ is tropically independent.
\end{proposition}

\begin{proof}
Suppose that $\{ \psi_i \}$ is tropically dependent, and choose real coefficients $c'_i$ such that the minimum of  $\{\psi_i + c'_i\}$ occurs at least occurs at least twice at every point $v \in \Gamma$.  Now, consider an arbitrary tropical linear combination $\theta = \min \{ \psi_i + c_i \}$.  Choose $j \in I$ so that $c_j - c'_j$ is maximal.  Then $\psi_j + c_j \geq \min_{j \neq i} \{ \psi_i + c_i \}$ at every point $v \in \Gamma$, and hence $\theta = \min \{ \psi_i + c_i  \}$ is not maximally independent.
\end{proof}

\subsection{Main strategy}
Our strategy for proving Theorem~\ref{thm:independence} is based on maximally independent combinations and Proposition~\ref{Prop:Strategy}, as follows.  Let $D_X$ be a divisor of degree $24 + \rho$ and rank 6 on a curve $X$ with skeleton $\Gamma$, and let $\Sigma = \trop(\cL(D_X))$.  Note that $\Sym^2 (\cL(D_X))$ has dimension $8 \choose 2$ = 28, and any function of the form $\psi_i + \psi_j$, with both $\psi_i$ and $\psi_j$ in $\Sigma$, is in the tropicalization of $\mathrm{Im}(\mu_2) \subset \cL(2D_X)$.

Therefore, to show that $\mu_2$ is injective, it suffices to give a maximally independent combination of 28 functions $\psi_i + \psi_j$, with both $\psi_i$ and $\psi_j$ in $\Sigma$.  This strategy has a distinct advantage over the approach used in \cite{MRC, MRC2}.  Rather than ruling out the existence of a tropical dependence by considering all possible tropical linear combinations of a given set of functions and arguing by contradiction, we algorithmically construct a single maximally independent combination, and apply Proposition~\ref{Prop:Strategy} to conclude that this set is tropically independent.

Although not logically necessary for the proofs of our theorems, we include the following interpretation of maximally independent combinations in the language of algebraic geometry.

\begin{proposition}  \label{prop:ag-interpretation}
Let $X$ be a curve over a discretely valued field with a regular semistable model $\mathcal X$ whose skeleton is $\Gamma$.  Let $f_0, \ldots, f_r$ be sections of $\mathcal{O}(D_X)$.  Then the following are equivalent:
\begin{enumerate}
\item  There are integers $c_0, \ldots, c_r$ such that $\theta = \min \{ \trop(f_i) + c_i \}$ is maximally independent.
\item  There are scalars $a_0, \ldots, a_r$, irreducible components $X_0, \ldots, X_r$ in the special fiber of $\mathcal X$, and a line bundle $\mathcal{L}$ on $\mathcal X$ extending $\mathcal{O}(D_X)$, such that $a_i f_i$ extends to a regular section of $\mathcal{L}$ and vanishes on $X_j$ if and only if $i \neq j$.
\end{enumerate}
\end{proposition}

\begin{proof}
Suppose there are integers $c_0, \ldots, c_r$ such that $\theta = \min \{ \trop(f_i) + c_i \}$ is maximally independent.  Choose vertices $v_0, \ldots, v_r$ such that $\trop(f_i)$ achieves the minimum uniquely at $v_i$, and label the remaining vertices $v_{r+1}, \ldots, v_s$.  Let $X_i$ be the irreducible component of the special fiber of $\mathcal X$ corresponding to $v_i$.  Consider the line bundle $$\mathcal{L} = \mathcal{O}_{\mathcal{X}}\big(\overline D_X - \theta(v_0) X_0 - \cdots - \theta(v_s) X_s\big),$$ where $\overline D_X$ is the closure of $D_X$ in the total space of $\mathcal X$.  Choose scalars $a_0, \ldots, a_r$ such that $\val(a_i) = \theta(v_i)$.  Then the order of vanishing of $a_i f_i$ along $X_j$ is $\trop(f_i) + c_i - \theta$ evaluated at $v_j$.  In particular, $a_i f_i$ is a regular section of $\mathcal L$, does not vanish on $X_i$, and vanishes on $X_j$ for $0 \leq j \leq r$ and $j \neq i$.

Conversely, given scalars $a_0, \ldots, a_r$, irreducible components $X_0, \ldots, X_r$, and an extension $\mathcal{L}$ of $\mathcal{O}(D_X)$ satisfying (2), set $c_i = \val(a_i)$.  Note that $\mathcal{L}$ must be of the form $\mathcal{O}_{\mathcal{X}} \big( \overline D_X - b_0 X_0 - \cdots - b_s X_s \big)$ for some integers $b_0, \ldots b_s$, and the order of vanishing of $a_j f_j$ along $X_i$ is $\trop(f_j) + c_j - b_i$, evaluated at $v_i$.  By assumption, $a_j f_j$ vanishes on $X_i$, for $j \neq i$, and $a_i f_i$ does not.  We conclude that $\trop(f_i) + c_i$ is strictly less than $\trop(f_j) + c_j$ at $v_i$, and hence $\theta = \min \{ \trop(f_i) + c_i \}$ is maximally independent.
\end{proof}

\begin{remark}
Proposition~\ref{prop:ag-interpretation} suggests some resemblance between our approach to proving linear independence of sections via maximally independent combinations and the technique used to prove cases of the maximal rank conjecture via limit linear series and linked Grassmannians on chains of elliptic curves in \cite{LiuOssermanTeixidorZhang}.  Osserman has also developed a notion of limit linear series for curves of pseudocompact type \cite{Osserman14}, a class of curves that includes the semistable reduction of the curve $X$ we study here, and relations to the Amini-Baker notion of limit linear series in tropical and nonarchimedean geometry are spelled out in \cite{Osserman17}.
\end{remark}

\subsection{The tropicalization of a pencil}  \label{Sec:Pencil}

The problem of understanding tropicalizations of arbitrary linear series seems hopelessly complicated.  In most cases that we can analyze, either the behavior is sufficiently similar to the vertex avoiding case, or the essential difference can be confined to a pencil.

Here we analyze the tropicalization of a pencil on $\PP^1$.  Although this is the simplest nontrivial example, it captures the features that will be essential for our purposes.  To understand the relation between this example and our later computations, imagine zooming out from the chain of loops, letting the bridges get long and the loops shrink to points.  Then the graph looks like an interval.  From this perspective, the pencils that we isolate and study in Section~\ref{Sec:Generic} behave exactly like the following example.

\begin{example}
\label{Ex:Interval}
Let $\Gamma$ be an interval with left endpoint $w$, viewed as a skeleton of $X = \PP^1$.  Let $D_X$ be a divisor specializing to $D = 2w$, let $V \subset \cL (D_X)$ be a rank 1 linear subseries, and let $\Sigma = \trop (V)$.

For each point $v$ in $\Gamma$, we consider the possibilities for the vector of slopes along the rightward pointing tangent vector $\zeta$ based at $v$.  Since each point $v$ has a unique rightward pointing tangent vector, for the purpose of this example we denote the vector of rightward slopes at $v$ by $s_v = (s_{\zeta}[0], s_\zeta[1])$.

Any function $\psi \in R(D)$ has rightward slope between 0 and 2 at every point, so $s_v$ is either $(1,2)$, $(0,2)$, or $(0,1)$.  We divide the interval $\Gamma$ into regions, according to these three possibilities.  Because the slopes of functions in $R(D)$ do not increase from left to right, the region where $s_v$ is equal to $(1,2)$ must be to the left of the region where it is equal to $(0,2)$, which in turn must be to the left of the region where it is equal to $(0,1)$.

First, consider the case where there is a non-empty region where $s_v$ is equal to $(0,2)$.  We can then identify functions in $\Sigma$ with specified slopes at every point of $\Gamma$, as follows.  Choose a function $\psi_1$ with maximal rightward slope at the right endpoint.  Because the slopes of functions in $R(D)$ cannot increase from left to right, we see that $s_v(\psi_1)$ must be positive at all rightward tangent vectors.  In particular $\psi_1$ must have rightward slope 2 in the region where $s_v$ is equal to $(0,2)$, and hence must also have rightward slope 2 to the left of this region.  Hence the rightward slope of $\psi_1$ is maximal among all functions in $\Sigma$, at all points of $\Gamma$.  By a similar argument, if we choose a function $\psi_0$ with minimal rightward slope at the left endpoint $w$ then it has minimal slope at every point of $\Gamma$.  The functions $\psi_0$ and $\psi_1$, with the three regions according to the vector $s_v$, are illustrated in Figure~\ref{Fig:EasyCase}.

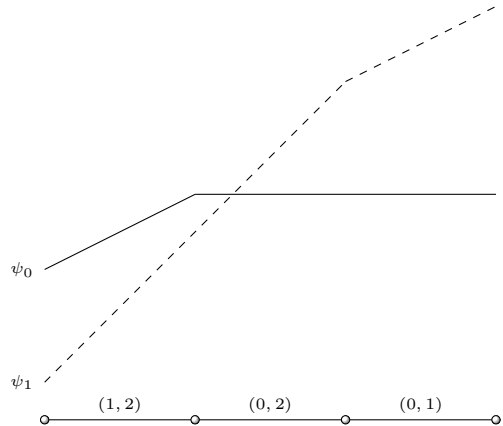
\begin{figure}[H]
\begin{center}
\begin{tikzpicture}

\draw (-0.3,2) node {{\tiny $\psi_1$}};
\draw[dashed] (0,2)--(4,6);
\draw[dashed] (4,6)--(6,7);

\draw (-0.3,3.5) node {{\tiny $\psi_0$}};
\draw (0,3.5)--(2,4.5);
\draw (2,4.5)--(6,4.5);

\draw (0,1.5)--(6,1.5);
\draw [ball color=white] (0,1.5) circle (0.55mm);
\draw [ball color=white] (6,1.5) circle (0.55mm);
\draw [ball color=white] (2,1.5) circle (0.55mm);
\draw [ball color=white] (4,1.5) circle (0.55mm);
\draw (1,1.7) node {{\tiny $(1,2)$}};
\draw (3,1.7) node {{\tiny $(0,2)$}};
\draw (5,1.7) node {{\tiny $(0,1)$}};
\end{tikzpicture}

\end{center}

\caption{The functions $\psi_0$ and $\psi_1$, when the $(0,2)$ region is nonempty.}
\label{Fig:EasyCase}
\end{figure}

A similar argument shows that, for any function $\psi \in \Sigma$, there is a region on the left where the slope of $\psi$ equals that of $\psi_1$ and a region to the right where it agrees with $\psi_0$.  Hence every function in $\Sigma$ is a tropical linear combination of $\psi_0$ and $\psi_1$.  So $\Sigma$ is completely determined by the regions where $s_v$ is $(1,2)$, $(0,2)$, and $(0,1)$, provided that the $(0,2)$ region is not empty.

\medskip

When there is no point $v$ with $s_v = (0,2)$, the possibilities for $\Sigma$ are more complicated.  In this case, there is a distinguished point $x \in \Gamma$ such that $s_v$ is equal to $(1,2)$ for all points $v$ to the left of $x$, and equal to $(0,1)$ for all points to the right of $x$.  Now, consider a function $\psi_A$ with minimal slope $1$ at the left endpoint.  We then see that $\psi_A$ has slope $1$ at all points to the left of $x$, because slopes of functions in $\Sigma$ do not increase,  but the slope of $\psi_A$ could be either $0$ or $1$ at any given point to the right of $x$.  More precisely, $\psi_A$ has slope 1 for some distance $t$ to the right of $x$, at which point its slope decreases to 0, and the slope is 0 the rest of the way.

Similarly, there is a function $\psi_B$ with maximal slope 1 at the right endpoint.  Then $\psi_B$ has slope 1 at all points to the right of $x$, and for some distance $t'$ to the left of $x$.  Then at all points of distance greater than $t'$ to the left of $x$, the slope of $\psi_B$ is 2.

Furthermore, by taking a tropical linear combination of two functions that agree at $x$, one with outgoing slope 0 and the other with incoming slope 2, we get a third function $\psi_C \in \Sigma$ with rightward slope $2$ everywhere to the left of $x$ and $0$ everywhere to the right of $x$.  The functions $\psi_A$, $\psi_B$, and $\psi_C$, with the distances $t$ and $t'$,  are illustrated schematically in Figure~\ref{Fig:HardCase}.

\begin{figure}[H]
\begin{center}
\begin{tikzpicture}

\draw (-0.3,4) node {{\tiny $\psi_A$}};
\draw (0,4)--(4,6);
\draw (4,6)--(6,6);

\draw (-0.3,3.35) node {{\tiny $\psi_B$}};
\draw[dashed] (0,3.5)--(1.5,5);
\draw[dashed] (1.5,5)--(6,7.25);

\draw (-0.3,2.7) node {{\tiny $\psi_C$}};
\draw[dotted] (0,2.9)--(3,5.9);
\draw[dotted] (3,5.9)--(6,5.9);

\draw (0,2.3)--(6,2.3);
\draw [ball color=white] (0,2.3) circle (0.55mm);
\draw [ball color=white] (6,2.3) circle (0.55mm);
\draw [ball color=white] (3,2.3) circle (0.55mm);
\draw [ball color=black] (1.4,2.3) circle (0.55mm);
\draw [ball color=black] (4,2.3) circle (0.55mm);
\draw (3,2.5) node {{\tiny $x$}};
\draw [<->] (2.95,2.1) -- (1.45,2.1);
\draw [<->] (3.05,2.1) -- (3.95,2.1);
\draw (2.25,1.9) node {{\tiny $t'$}};
\draw (3.5,1.9) node {{\tiny $t$}};
\end{tikzpicture}

\end{center}

\caption{The functions $\psi_A$, $\psi_B$, and $\psi_C$ when the $(0,2)$ region is empty.}
\label{Fig:HardCase}
\end{figure}
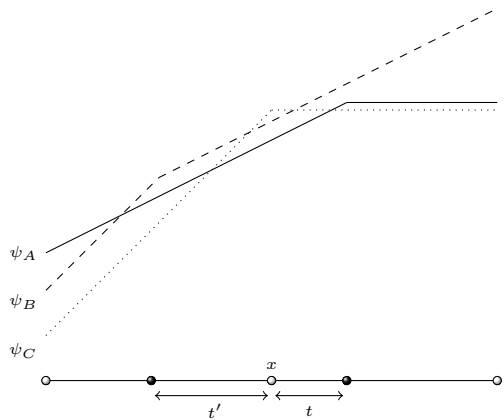

\noindent In this case, we may again define $\PL$ functions $\psi_0$ and  $\psi_1$ with rightward slopes $s_\zeta[0]$ and $s_\zeta[1]$, respectively, at every rightward tangent vector $\zeta$.  These functions are still contained in the complete tropical linear series $R(D)$, but may not be contained in $\Sigma$, the tropicalization of our algebraic pencil.  (With our notation above, if $t = 0$ then $\psi_0$ is in $\Sigma$, and if $t' = 0$, then $\psi_1$ is in $\Sigma$.)

We may also define a function $\psi^{\infty} \in R(D)$ with rightward slope $1$ at every point.  Again, the function $\psi^{\infty}$ is not necessarily contained in $\Sigma$.  (If $t$ or $t'$ is sufficiently large, equal to the distance from $x$ to the right or left endpoint of $\Gamma$, respectively, then $\psi^{\infty}$ is in $\Sigma$.)

It is important to notice, however, that each of $\psi_A$, $\psi_B,$ and $\psi_C$ is a tropical linear combination of $\psi_0$, $\psi_1,$ and $\psi^{\infty}$.

Most importantly, we claim that the distances $t$ and $t'$ \emph{must be equal}.  To see this, note that the functions $\psi_A$, $\psi_B$, and $\psi_C$ are tropicalizations of functions in a pencil.  Therefore the rational functions in $\cL(D_X)$ tropicalizing to these three functions are linearly dependent, and hence  $\psi_A$, $\psi_B$, and $\psi_C$ \emph{must be tropically dependent}.  On each region where all three functions are linear, there are exactly two with the same slope, and this determines the combinatorial type of the tropical dependence, i.e., which functions achieve the minimum on which regions, as shown in Figure~\ref{Fig:IntervalDependence}.

\begin{figure}[H]
\begin{center}
\begin{tikzpicture}

\draw (0,0)--(6,0);
\draw [ball color=white] (0,0) circle (0.55mm);
\draw [ball color=white] (6,0) circle (0.55mm);

\draw (-0.3,2.15) node {{\tiny $\psi_A$}};
\draw (0,2.15)--(4,4.15);
\draw (4,4.15)--(6,4.15);

\draw (-0.3,0.95) node {{\tiny $\psi_B$}};
\draw[dashed] (0,1)--(2.05,3.05);
\draw[dashed] (2.16,3.16)--(6,5.08);

\draw (-0.3,1.15) node {{\tiny $\psi_C$}};
\draw[dotted] (0,1.1)--(3,4.1);
\draw[dotted] (3,4.1)--(6,4.1);

\draw [ball color=black] (2,0) circle (0.55mm);
\draw [ball color=white] (3,0) circle (0.55mm);
\draw [ball color=black] (4,0) circle (0.55mm);
\draw (1,0.3) node {{\tiny $BC$}};
\draw (3,0.3) node {{\tiny $AB$}};
\draw (5,0.3) node {{\tiny $AC$}};
\draw (3,-0.2) node {{\tiny $x$}};

\end{tikzpicture}
\caption{The tropical dependence that shows $t=t'$.}
\label{Fig:IntervalDependence}
\end{center}
\end{figure}
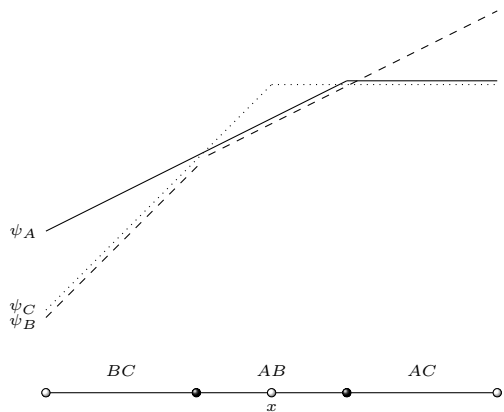

\noindent Note that all three functions achieve the minimum at two points:  the point at distance $t$ to the right of $x$, and the point at distance $t'$ to the left of $x$.   Comparing the slopes of $\psi_C$ to those of $\psi_A$ shows that $\psi_C - \psi_A$ is equal to $t'$ at $x$, and also equal to $t$.  This proves that $t = t'$, as claimed.
\end{example}

\section{Chains of loops}  \label{sec:chainsofloops}

We now focus attention, for the remainder of the paper, on the case where $\Gamma$ is a chain of loops with admissible edge lengths.

Let $\Gamma$ be a chain of loops with bridges.  It has $2g+2$ vertices, one on the lefthand side of each bridge, which we label $w_0, \ldots , w_g$, and one on the righthand side of each bridge, which we label $v_1, \ldots, v_{g+1}$.  There are two edges connecting the vertices $v_k$ and $w_k$, the top and bottom edges of the $k$th loop, whose lengths are denoted $\ell_k$ and $m_k$, respectively, as shown in Figure~\ref{Fig:TheGraph}.

\begin{center}
\begin{figure}[H]
\scalebox{.9}{
\begin{tikzpicture}
\draw (-1.75,0) node {\footnotesize $w_0$};
\draw (-1.5,0)--(-0.5,0);
\draw (0,0) circle (0.5);
\draw (-0.25,0) node {\footnotesize $v_1$};
\draw (0.25,0) node {\footnotesize $w_1$};
\draw (0.5,0)--(1.5,0);
\draw (2,0) circle (0.5);
\draw (1.75,0) node {\footnotesize $v_2$};
\draw (2.5,0)--(3.5,0);
\draw (4,0) circle (0.5);
\draw (4.5,0)--(5.5,0);
\draw (6,0) circle (0.5);
\draw (6.5,0)--(7.5,0);
\draw (8,0) circle (0.5);
\draw (7.75,0) node {\footnotesize $v_g$};
\draw (8.25,0) node {\footnotesize $w_g$};
\draw (8.5,0)--(9.5,0);
\draw (9.85,0) node {\footnotesize $v_{g+1}$};
\draw [<->] (4.65,0.15)--(5.35,0.15);
\draw [<->] (4.6,.25) arc[radius = 0.65, start angle=20, end angle=160];
\draw [<->] (4.61, -.15) arc[radius = 0.63, start angle=-9, end angle=-173];
\draw (5,0.4) node {\footnotesize$n_k$};
\draw (4,1) node {\footnotesize$\ell_k$};
\draw (4,-1) node {\footnotesize$m_k$};
\end{tikzpicture}
}
\caption{The chain of loops $\Gamma$.}
\label{Fig:TheGraph}
\end{figure}
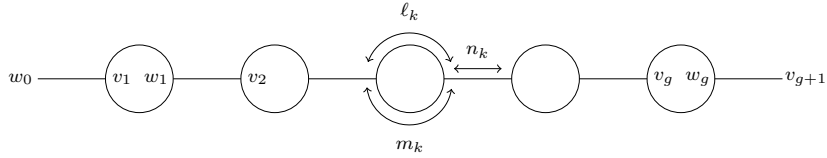
\end{center}

\noindent We denote by $\gamma_k$ the $k$th loop, which is formed by the two edges connecting $v_k$ and $w_k$, of length $\ell_k$ and $m_k$, for $1 \leq k \leq g$.  And we denote by $\beta_k$ the $k$th bridge, which connects $w_k$ and $v_{k+1}$, of length $n_k$, for $0 \leq k \leq g$.

Throughout, we assume that $\Gamma$ has admissible edge lengths in the following sense.

\begin{definition} \label{Def:Admissible}
The graph $\Gamma$ has \emph{admissible edge lengths} if
\[
\ell_{k+1} \ll m_k \ll \ell_k \ll n_k \ll n_{k-1} \mbox{ for all $k$}.
\]
\end{definition}

\begin{remark}
These conditions on edge lengths are more restrictive than those in \cite{MRC} and \cite{tropicalBN}.  Our arguments here, e.g., in Lemma~\ref{Lem:VANotToRight}, require not only that the bridges are much longer than the loops, but also that each loop is much larger than the loops that come after it.  For illustrative purposes, we will generally draw the loops and bridges as if they are the same size.
\end{remark}

\subsection{Special divisors on a chain of loops}  \label{sec:classification}
By the Riemann-Roch Theorem, every divisor class of degree $d$ on $\Gamma$ has rank at least $d - g$ \cite{BakerNorine07}.  The special divisor classes on $\Gamma$, i.e., the classes of degree $d$ and rank strictly greater than $d-g$, are classified in \cite{tropicalBN}.  We briefly recall the structure of this classification and refer the reader to the original paper for further details.

The Brill-Noether locus $W^r_d (\Gamma)$ parametrizing divisor classes of degree $d$ and rank $r$ is a union of $\rho$-dimensional tori.  These tori are in bijection with standard Young tableaux on a rectangle of size $(r+1) \times (g-d+r)$, with entries from $\{ 1, \ldots , g \}$.

An open dense subset of each torus consists of vertex avoiding divisor classes.  We refer the reader to \cite[Definition~2.3]{LiftingDivisors} for a definition.  Before proceeding to the general case, in Section~\ref{Sec:VertexAvoiding}, we prove Theorem~\ref{Thm:MainThm} in the case where the divisor class $D$ is vertex avoiding.  In this section, we review the combinatorics of vertex avoiding divisors.  For analogues of these combinatorial results in the case of not necessarily complete tropicalizations of linear series, we refer the reader to Section~\ref{Sec:Loops}.

Given a vertex avoiding divisor class of rank $r$ on $\Gamma$, there is a unique effective divisor $D_i$ in this class such that $\deg_{w_0}(D_i) = r-i$ and $\deg_{v_{g+1}}(D_i) = i$.  We define the functions $\psi_i$ such that $\ddiv(\psi_i) = D_i - D_0$.  Note that $\psi_i$ is uniquely determined up to an additive constant.  Because the divisors $D_i$ are unique, if $D_X$ is a divisor of rank $r$ that specializes to the given divisor class, then $D_i \in \Trop (\vert D_X \vert)$.  We may therefore choose $D_X$ so that it specializes to $D_0$, and we have $\psi_i \in \trop (\cL (D_X))$ for all $0 \leq i \leq r$.

\begin{remark}
Our notation here differs slightly from that used in \cite{MRC}.  In the earlier paper, the divisor denoted here by $D_i$ was denoted $D_{r-i}$.  We find the current choice more natural, as the function $\psi_i$ has slope $i$ along the bridge $\beta_0$.
\end{remark}

Although we have worked extensively with the functions $\psi_i$ in the past, here we work with a slightly different collection of functions, obtained by adding a fixed PL function $\varphi$ to each $\psi_i$.  The reason for this choice is explained in Remark~\ref{Rem:Breaks}.  The function $\varphi$ is chosen as follows.

Every divisor on $\Gamma$ is equivalent to a unique \emph{break divisor} $D$, with $d-g$ chips at $w_0$, and precisely one chip on each loop $\gamma_k$; see, for instance, \cite{ABKS}.  We then choose $\varphi$ so that
\[
\ddiv(\varphi) = D_0 - D,
\]
and define
\[
\varphi_i = \psi_i + \varphi.
\]
With this notation, we have $\ddiv(\varphi_i) = D_i - D$.  In other words, the functions $\psi_i$ move $D_0$ to $D_i$, whereas the functions $\varphi_i$ move $D$ to $D_i$.  So we are now using the break divisor $D$ as the starting point for our constructions.  Note that adding or subtracting $\varphi$ gives a natural bijection between tropical dependences (resp. maximally independent combinations) of $\{ \psi_i \}$ and those of $\{ \varphi_i \}$.  Similarly, adding or subtracting $2 \varphi$ gives a natural bijection between tropical dependences (resp. maximally independent combinations) of $\{ \psi_i + \psi_j \}$ and those of $\{ \varphi_i + \varphi_j \}$.

Unlike the divisors $D_i$, the break divisor $D$ may not be the specialization of any divisor in $\vert D_X \vert$.  However, we find that the combinatorial advantages of the functions $\varphi_i$, as explained in Remark~\ref{Rem:Breaks}, outweigh this minor inconvenience.  These advantages were also discovered and used earlier, by Pflueger, in his work classifying special divisors on arbitrary chains of loops, with non-generic edge lengths \cite{Pflueger17a}.

\subsection{Slopes along bridges}  The slopes of $\PL$ functions along the bridges  of $\Gamma$, and especially the incoming and outgoing slopes at each loop, play a special role in controlling which functions can achieve the minimum on which regions of the graph in a given tropical linear combination.  We use the following notation for these slopes.

Given $\psi \in \PL (\Gamma)$, let $s_k(\psi)$ be the incoming slope from the left at $v_k$, and let $s'_k(\psi)$ be the outgoing slope of $\psi$ to the right at $w_k$.

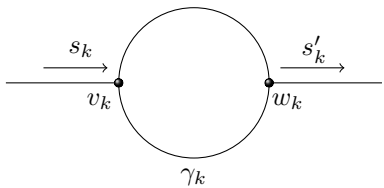
\begin{figure}[h!]
\begin{tikzpicture}
\begin{scope}[grow=right, baseline]
\draw (-1,0) circle (1);
\draw (-3.5,0)--(-2,0);
\draw (0,0)--(1.5,0);
\draw [ball color=black] (-2,0.) circle (0.55mm);
\draw [ball color=black] (0,0) circle (0.55mm);
\draw (-1,-1.25) node {{$\gamma_k$}};
\draw (-2.25,-.25) node {{$v_k$}};
\draw (.25,-.25) node {{$w_k$}};
\draw (-2.5,.45) node {{$s_k$}};
\draw [->] (-3,.2)--(-2.15,0.2);
\draw (.6,.45) node {{$s'_k$}};
\draw [->] (.15,.2)--(1,0.2);
\end{scope}
\end{tikzpicture}
\caption{The Slopes $s_k$ and $s'_k$.}
\label{Fig:Slopes}
\end{figure}

\noindent Note that $s'_k(\psi)$ and $s_{k+1}(\psi)$ are the rightward slopes of $\psi$ at the beginning and end, respectively, of the bridge $\beta_k$.  The functions $\varphi_i$ have constant slopes along each bridge, so $s'_k(\varphi_i) = s_{k+1}(\varphi_i)$.

\subsection{Slopes of the maximally independent combination}  \label{sec:slopes}

From this point onward, we assume that $r=6$, $g=21+\rho$, and $d=24+\rho$.  Our goal is to construct a maximally independent combination of functions in the tropicalization of the image of $\mu_2$.  In the vertex avoiding case, we follow the usual notational convention, writing
\[
\varphi_{ij} = \varphi_i + \varphi_j,
\]
for the pairwise sums of the distinguished functions $\varphi_0, \ldots, \varphi_6$, and construct a maximally independent combination
\[
\theta = \min_{ij}  \{ \varphi_{ij}+ c_{ij} \}.
\]

\begin{remark}
In the general case, we will give a similar formula for the slopes of an auxiliary $\PL$-function, which we call the \emph{master template}.  We construct the master template as a tropical linear combination of pairwise sums of \emph{building blocks} analogous to $\psi_0$, $\psi_1$, and $\psi^\infty$ in Example~\ref{Ex:Interval}.  These building blocks will not necessarily be in the tropicalization $\Sigma$ of our linear series, but $\Sigma$ is contained in the tropical convex hull of the building blocks, and we will use the master combination as a key step toward building the required maximally independent combination of pairwise sums of functions in $\Sigma$.  See Section~\ref{Sec:Construction}.
\end{remark}

Recall that each vertex avoiding class is contained in the torus of special divisor classes corresponding to a unique standard Young tableau of shape $(r+1) \times (g-d+r)$ with entries from $\{1, \ldots, g \}$.  In terms of this tableau, the slope of $\varphi_i$ along the bridge $\beta_k$ is
\[
s_k (\varphi_i) = i - (g-d+r) + \# \left\{ \text{entries} \leq k \text{ in column } r+1-i \right\} .
\]

Our algorithm for constructing the maximally independent tropical linear combination $\theta$ is easier to describe if we specify the slope of $\theta$ on each bridge in advance.  We do so as follows.

\begin{definition}\label{def:z}
We define $z$ to be the 6th smallest entry appearing in the union of the first two rows of the tableau, and choose $z'$ so that $z'+2$ is the 10th smallest entry appearing in the union of the second and third row.
\end{definition}

The incoming slopes of $\theta$ at $v_k$, the leftmost point on $\gamma_k$, will be:
\begin{displaymath}
s_k (\theta) = \left\{ \begin{array}{ll}
4 & \textrm{if $k \leq z$,} \\
3 & \textrm{if $z < k \leq z'$,}\\
2 & \textrm{if $z' < k \leq g$.}
\end{array} \right.
\end{displaymath}
The precise choice of $z$ and $z'$ will be important later; see, e.g.,  Lemmas~\ref{Lem:VAAtMostThree} and \ref{lem:nonew}, and Corollary~\ref{cor:counting}.  For now, what is important to see is that the graph $\Gamma$ is divided into three \emph{blocks}, the first from $\gamma_1$ to $\gamma_z$, the second from $\gamma_{z+1}$ to $\gamma_{z'}$ and the third from $\gamma_{z'+1}$ to $\gamma_g$.  Within each block, the slope of $\theta$ will be nearly constant on each bridge, equal to 4, 3, or 2, on bridges within the first, second, and third blocks, repectively.  By nearly constant, we mean that the slope of $\theta$ might be different (1 or 2 higher) for a short distance at the beginning of the bridge, but since the bridges are very long, the average slope over each bridge within a block will be very close to 4, 3, or 2, according to the block.  On the bridges between the blocks, the slope decreases by 1 at the midpoint of the bridge.

Constructing $\theta$ in this way, with a predetermined average slope on each bridge, constant over long blocks, gives us precise control over which functions $\varphi_{ij}$ are candidates to achieve the minimum (permissible functions, as defined in the following section) on each loop.

\subsection{Permissible functions}
\label{Sec:Permissible}
In \cite{MRC}, we introduced the notion of permissible functions.  These are functions that satisfy a natural necessary condition to achieve the minimum at some point of a given loop, provided that the divisor associated to the minimum has a particular specified degree distribution.  Slopes along bridges encode the same information as degree distributions, as explained, e.g., in Definition~3.1 and the proof of Proposition~3.3 in \cite{MRC2}.

We now restate the characterization of permissibility for functions with constant slopes along bridges, such as $\varphi_{ij}$, assuming that $\theta$ has slopes along bridges as specified above.

\begin{remark}
In the vertex avoiding case, we will use permissibility only for the functions $\varphi_{ij}$.  Nevertheless, we discuss permissibility for arbitrary functions with constant slopes along bridges, since this more general notion will be applied to building blocks in the general case.
\end{remark}

\begin{definition}
Let $\psi \in \PL (\Gamma)$ be a function with constant slope along each bridge.  We say that $\psi$ is \emph{permissible} on $\gamma_k$ if
\begin{enumerate}
\item  $s_{j} (\psi) \leq s_{j}(\theta)$ for all $j\leq k$,
\item $s_{k+1}(\psi) \geq s_{k}(\theta)$, and
\item  if $s_{\ell} (\psi) < s_{\ell}(\theta)$ for some $\ell > k$, then there exists $k'$, with $k < k' < \ell$, such that $s_{k'} (\psi) > s_{k'} (\theta)$.
\end{enumerate}
\end{definition}

\noindent To understand the motivation for the definition, keep in mind that $\theta$ has nearly constant slope on each bridge (with the exception of bridges between blocks, where the slope decreases by 1 at the midpoint).  Even if $\gamma_k$ is the last loop of a block, $\theta$ has slope $s_k(\theta)$ on the second half of the bridge $\beta_{k-1}$ and average slope very close to $s_k(\theta)$ on the first half of $\beta_k$.  Also, the bridges adjacent to a loop are much longer than the edges in the loop, and both bridges and loops get much smaller as we move from left to right across the graph.  This definition is therefore pointing to the simple fact that any $\psi$ with constant slopes along bridges that achieves the minimum on $\gamma_k$ must have smaller than or equal slope, when compared with the minimum $\theta$, on every bridge to the left and greater than or equal slope on the first half of the bridge immediately to the right.  Moreover, if it has smaller slope on a bridge further to the right, then it must have had strictly larger slope on some bridge in between.

In the vertex avoiding case, the condition for $\varphi_{ij}$ to be permissible on $\gamma_k$ simplifies as follows.

\begin{lemma}
\label{Lem:VAPermissible}
In the vertex avoiding case, $\varphi_{ij}$ is permissible on $\gamma_k$ if and only if
$$s_{k} (\varphi_{ij}) \leq s_{k}(\theta) \leq s_{k+1} (\varphi_{ij}).$$
\end{lemma}

\begin{proof}
In the vertex avoiding case, the slopes $s_k(\varphi_{ij})$ are nondecreasing in $k$, while $s_k(\theta)$ is nonincreasing.
\end{proof}

\noindent We also note the following, which holds for any function with constant slopes along bridges, not just those of the form $\varphi_{ij}$.

\begin{lemma}
\label{Lem:EverythingIsPermissible}
For any function $\psi \in \PL (\Gamma)$ with constant slope along each bridge, one of the following is true:
\begin{enumerate}
\item  $s_1 (\psi) > s_1(\theta)$;
\item  $s_{g+1} (\psi) < s_{g+1}(\theta)$;
\item  there is a $k$ such that $\psi$ is permissible on $\gamma_k$.
\end{enumerate}
\end{lemma}

\begin{proof}
Suppose that $s_1 (\psi) \leq s_1(\theta)$ and $s_{g+1} (\psi) \geq s_{g+1}(\theta)$.  If $s_k (\psi) \leq s_k(\theta)$ for all $k$, then $\psi$ is permissible on $\gamma_g$.  Otherwise, consider the smallest value of $k$ such that $s_k (\psi) > s_k(\theta)$.  Then $\psi$ is permissible on $\gamma_{k-1}$.
\end{proof}

\begin{remark} \label{rem:consecutive}
The set of loops on which a given function $\psi$ with constant slope along each bridge is permissible consists of \emph{consecutive} loops.  The last loop where a function is permissible is $\gamma_k$, where $k$ is the smallest value such that $s_{k+1} (\psi) > s_{k+1}(\theta)$.  The first loop where a function is permissible is $\gamma_{\ell}$, where $\ell$ is the largest value such that $\ell \leq k$ and $s_{\ell} (\psi) < s_{\ell}(\theta)$.
\end{remark}

\section{The vertex avoiding case}
\label{Sec:VertexAvoiding}

In this section, we prove Theorem~\ref{thm:independence} in the vertex avoiding case using maximally independent combinations, as follows.

\begin{theorem} \label{thm:vertexavoiding}
Let $D$ be a break divisor of degree $24 + \rho$ and rank $6$ on a chain of $21 + \rho$ loops whose class is vertex avoiding.  Then there is a maximally independent combination $\theta = \min \{ \varphi_{ij} + c_{ij} \}$.  In particular, $\{ \varphi_{ij} \ | \ 0  \leq i, j \leq 6 \}$ is tropically independent.
\end{theorem}

\begin{remark}
In the general case, which is detailed in Sections~\ref{Sec:Construction}-\ref{Sec:Generic}, we will follow a similar approach to prove tropical independence for certain collections of pairwise sums of functions in $\trop(\cL(D_X))$, but with additional arguments to deal with the different possibilities for what these functions might be.
\end{remark}

We proceed from left to right across the graph and give an algorithm for constructing this maximally independent combination $\theta$ as we go, with slopes $s_k(\theta)$ as specified in Section~\ref{sec:slopes}.  Recall that we have divided $\Gamma$ into three \emph{blocks}, containing loops $\gamma_1$ through $\gamma_z$, $\gamma_{z+1}$ through $\gamma_{z'}$, and $\gamma_{z'+1}$ through $\gamma_g$.  The slope of $\theta$ is nearly constant on each bridge within a given block, equal to 4 on the first block, 3 on the second block, and 2 on the third block.  The slope of $\theta$ changes from $4$ to $3$ and from $3$ to $2$ at the midpoints of the bridges $\beta_z$ and $\beta_{z'}$ between the blocks.

The values of $z$ and $z'$ and the slopes of $\theta$ are chosen so that, on each block, the number of functions $\varphi_{ij}$ that are permissible on some loop is one more than the number of lingering loops (Corollary~\ref{cor:counting}).  Within each block, we assign one of these permissible functions to achieve the minimum uniquely on each nonlingering loop, and the remaining permissible function achieves the minimum uniquely on the bridge following the block.  Since there are 21 nonlingering loops and three blocks, this gives us a maximally independent configuration of 24 functions.  The remaining 4 functions, with slopes too high or too low to be permissible on any block, achieve the minimum uniquely on the bridges to the left of the first loop or to the right of the last loop, respectively.  See Example~\ref{ex:randomtableau} for a schematic illustration of the output of this algorithm, for one randomly chosen tableau.

\begin{remark}
Note that the functions $\{ \varphi_{ij} \}$ may admit many different maximally independent combinations with different combinatorial properties.  There is no obvious reason to prefer one such combination over another.  We present one particular algorithm for constructing a maximally independent combination that works uniformly for all vertex avoiding divisors and generalizes naturally to the non-vertex avoiding case.
\end{remark}

\subsection{Counting permissible functions}
\label{sec:counting}

Recall that functions with constant slopes along bridges, such as $\varphi_{ij}$, may achieve the minimum only on loops where they are permissible, as discussed in Section~\ref{Sec:Permissible}.   Our algorithm is organized around keeping track of which functions $\varphi_{ij}$ are permissible at each step, as we move from left to right across the graph.

The set of indices $k$ such that $\varphi_{ij}$ is permissible on $\gamma_k$ is the set of integer points in an interval, as noted in Remark~\ref{rem:consecutive}, so we pay special attention to the first and last loops on which a function is permissible.  If $\gamma_k$ is the first loop or on which $\varphi_{ij}$ is permissible, then we say that $\varphi_{ij}$ is a \emph{new permissible function} on $\gamma_k$.  Similarly, if $\gamma_k$ is the last loop on which $\varphi_{ij}$ is permissible, then we say that $\varphi_{ij}$ is a \emph{departing permissible function} on $\gamma_k$.

The behavior of permissible functions is slightly different on the first and last loops of each of the three blocks where $s_k(\theta)$ is constant.  Although these are special cases, they end up being relatively simple.  The main part of our argument in this section is an algorithm for constructing a maximally independent combination of permissible functions on each block. We prepare for the proof of Theorem~\ref{thm:vertexavoiding} with the following lemmas controlling the new and departing permissible functions on loops within a given block.

\begin{lemma}
\label{lem:onenew}
If $\gamma_k$ is not the first loop of a block, then there is at most one new permissible function $\varphi_{ij}$  on $\gamma_k$.  Furthermore, if $\gamma_k$ is lingering then there are none.
\end{lemma}

\noindent Note that the first loops of the blocks are $\gamma_1$, $\gamma_{z+1}$, and $\gamma_{z'+1}$, so the conclusion of the lemma holds for $k \not \in \{ 1, z+1, z'+1 \}$.

\begin{proof}
Recall that, by Lemma~\ref{Lem:VAPermissible}, $\varphi_{ij}$ is permissible on $\gamma_k$ if and only if $s_{k} (\varphi_{ij}) \leq s_{k}(\theta) \leq s_{k +1}(\varphi_{ij})$.  Suppose $\gamma_k$ is not the first loop of a block.  Then $s_{k-1}(\theta) = s_{k}(\theta)$.  If $\varphi_{ij}$ is a new permissible function, we must have $s_k(\varphi_{ij}) < s_k(\theta) \leq s_{k+1}(\varphi_{ij})$.  Hence the outgoing slope of $\varphi_i$ or $\varphi_j$ must be strictly greater than the incoming slope.  If $\gamma_k$ is lingering then there is no such function, and hence there is no new permissible function.

Otherwise, assume $\gamma_k$ is nonlingering.  Then there is exactly one index $i$ such that $s_{k+1}(\varphi_i) > s_k(\varphi_i)$, and the increase in slope is exactly 1.  Note that the slopes of all other $\varphi_j$ are unchanged, and different from both $s_{k}(\varphi_i)$ and $s_{k+1}(\varphi_i)$.  It follows that there is at most one $j$ (possibly equal to $i$) such that $s_{k}(\varphi_{ij}) < s_k (\theta)$ and $s_{k+1}(\varphi_{ij}) \geq s_k (\theta)$, and hence there is at most one new permissible function $\varphi_{ij}$.
\end{proof}

\begin{lemma}
\label{lem:onedeparts}
If $\gamma_k$ is not the last loop of a block, then there is at most one departing permissible function $\varphi_{ij}$  on $\gamma_k$.  Furthermore, if $\gamma_k$ is lingering then there are none.
\end{lemma}

\noindent Note that the last loops of the blocks are $\gamma_z$, $\gamma_{z'}$, and $\gamma_{g}$, so the conclusion of the lemma holds for $k \not \in \{ z, z', g \}$.

\begin{proof}
Suppose $\gamma_k$ is not the last loop of a block.  Then $s_{k}(\theta) = s_{k+1}(\theta)$.  If $\varphi_{ij}$ is a departing permissible function, we must have $s_{k}(\varphi_{ij}) \leq s_k(\theta) < s_{k+1}(\varphi_{ij})$.  Hence the slope of $\varphi_i$ or $\varphi_j$ must increase from $\beta_{k-1}$ to $\beta_k$.  If $\gamma_k$ is lingering then there is no such function, and hence there is no departing permissible function.  The rest of the proof is similar to that of the previous lemma.
\end{proof}

\begin{lemma}
\label{Lem:VAAtMostThree}
For any loop $\gamma_k$, there are at most 3 non-departing permissible functions on $\gamma_k$.  Moreover, there are at most 3 permissible functions on each of the loops $\gamma_1$ and $\gamma_{z'+1}$.
\end{lemma}

\begin{proof}
By definition, if $\varphi_{ij}$ is a non-departing permissible function on $\gamma_k$, then $s_{k+1} (\varphi_{ij}) = s_{k+1} (\theta)$.  Now, if $i \neq i'$ and $s_{k+1} (\varphi_i) + s_{k+1} (\varphi_j) = s_{k+1} (\varphi_{i'}) + s_{k+1} (\varphi_{j'})$, then $j \neq j'$.  It follows that there are most $\lceil \frac{r+1}{2} \rceil$ non-departing permissible functions $\varphi_{ij}$, with equality only in the case where $s_{k+1} (\varphi_3) = \frac{s_{k+1}(\theta)}{2}$.  Note, however, that $s_{k}(\theta)$ is even only if ${k} \leq z$ or $k > z'$.  By the definition of $z$, the center column of the corresponding tableau contains at most 1 symbol less than or equal to $z$, so if $k+1 \leq z$ we have
\[
s_{k+1} (\varphi_3) \leq 1 < 2 = \frac{s_{k+1}(\theta)}{2} .
\]
A similar argument applies in the case that $k > z'$.  So we see that the number of permissible functions is at most $\lceil \frac{6+1}{2} \rceil -1 = 3$.

The fact that there are at most 3 permissible functions on $\gamma_1$ follows directly from the fact that $s_0 (\varphi_i) = i-3$ for all $i$.  The fact that there are at most 3 permissible functions on $\gamma_{z'+1}$ follows by enumerating the possibilities for the vector $(s_{z'+2} (\varphi_0), \ldots , s_{z'+2} (\varphi_6))$ as in Lemma~\ref{lem:nonew} below.
\end{proof}

\begin{remark}
Lemma~\ref{Lem:VAAtMostThree} and its generalization Lemma~\ref{Lem:AtMostThree} are key places where we use the assumption that $r=6$.  Extending our method to prove further cases of the strong maximal rank conjecture for larger $r$ would require new ideas at these steps.
\end{remark}

Let $b$ be the 7th smallest entry appearing in the first two rows of the tableau and let $b'$ be the 8th smallest symbol appearing in the union of the first and third row.  We note that $z < b < b' \leq z'$.  The first two inequalities are straightforward.  To see the last inequality, recall from Definition~\ref{def:z} that $z' + 2$ is the 10th smallest entry that appears in the union of the second and third row.  Therefore, the 9th smallest symbol appearing in the union of the second and third row must be strictly between $b'$ and $z'+2$.  From these inequalities, it follows that $\gamma_b$ and $\gamma_{b'}$ are in the second block.

\begin{lemma}  \label{lem:nonew}
If $\gamma_b$ is not the first loop in the second block, then the non-lingering loops with no new permissible functions are exactly $\gamma_z$, $\gamma_b$, $\gamma_{b'}$, and $\gamma_{z'+2}$.  Otherwise, the non-lingering loops with no new permissible functions are exactly $\gamma_z$, $\gamma_{b'}$, and $\gamma_{z'+2}$, and there are only 3 permissible functions on $\gamma_b$.
\end{lemma}

\begin{proof}
We begin by showing that there are no new permissible functions on $\gamma_z$.    Suppose $\varphi_{ij}$ is a new permissible function on $\gamma_z$.   This exactly means that $s_{z} (\varphi_{ij}) < 4 \leq s_{z+1} (\varphi_{ij})$.  We will show that this is impossible.

Recall (from Definition~\ref{def:z}) that $z$ is the 6th smallest entry appearing in the first two rows of the tableau.  There are 4 possibilities for the location of these entries, corresponding to the partitions of 6 with no part larger than 2.  These give rise to the following four ranges of possibilities for the vector $(s_{z+1} (\varphi_0), \ldots , s_{z+1} (\varphi_6))$.  Here, an entry $\geq n$ for $s_{z+1}(\varphi_i)$ indicates that $s_{z+1} (\varphi_i) \geq n$ and also that $s_{z+1} (\varphi_i) = s_{z}(\varphi_i)$.
\[
(-3,-1,0,1,2,3,4)
\]
\[
(-3,-2,0,1,2,3,5)
\]
\[
(-3,-2,-1,1,2,4,\geq 5)
\]
\[
(-3,-2,-1,0,3,\geq 4, \geq 5)
\]
Recall that there is exactly one value of $i$ such that $s_{z+1} (\varphi_i) > s_{z} (\varphi_i)$.  Specifically, if $z$ appears in column $r+1-i$ of the tableau, then $s_{z+1} (\varphi_i) = s_{z} (\varphi_i)+1 > s_{z+1} (\varphi_{i-1})+1$.  Considering each of the 4 vectors above, we see that for each $i$ satisfying $s_{z+1} (\varphi_i) > s_{z+1} (\varphi_{i-1})+1$, there is no $j$ such that $s_{z+1} (\varphi_{ij}) =4$.  The result follows.

The proofs that $\gamma_{b'}$ and $\gamma_{z'}$ have no new permissible functions are similar, as is the proof that $\gamma_b$ has no new permissible functions if $b$ is not the first loop in the second block, i.e., if $b \neq z+1$.  If $b = z+1$, then a similar argument shows that there is no permissible function $\varphi_{ij}$ with $s_{b} (\varphi_{ij}) < 3$.  By an argument similar to the proof of Lemma~\ref{Lem:VAAtMostThree}, we then see that there are only 3 permissible functions on $\gamma_b$.

It remains to show that these are the only loops with no new permissible functions.  We note that if $\gamma_k$ is not the first loop in a block, then there is at most one new permissible function on $\gamma_k$.  Furthermore, there are at most 3 permissible functions on $\gamma_1$, at most 4 permissible functions on $\gamma_{z+1}$ (by Lemma~\ref{Lem:VAAtMostThree}), and at most 3 permissible functions on $\gamma_{z'+1}$.  This means that there must be at least $28-(2 \cdot 3 + 2 \cdot 4) = 14$ non-lingering loops apart from these three on which there is a new permissible function.  But if $b \neq z+1$, then the number of non-lingering loops other than $\gamma_1$, $\gamma_z$, $\gamma_{z+1}$, $\gamma_b$, $\gamma_{b'}$, $\gamma_{z'+1}$, and $\gamma_{z'+2}$ is $21-7=14$.  It follows that, on every one of these loops, there is a new permissible function.  If $b=z+1$, then there are 15 non-lingering loops, and at least $28-(3 \cdot 3 + 4) = 15$ non-lingering loops apart from those listed on which there is a new permissible function.  Hence every other non-lingering loop has a new permissible function, as required.
\end{proof}

\begin{corollary}
\label{cor:counting}
On each of the three blocks, the number of permissible functions is 1 more than the number of non-lingering loops.
\end{corollary}

The following proposition will be most useful when working with three non-departing permissible functions on a loop.  Its proof relies on the following lemma about the divisor of a piecewise linear function on $\Gamma$ obtained as the minimum of several functions in $R(D)$, from our previous work on tropical independence.

\begin{ShapeLemma} \cite[Lemma~3.4]{tropicalGP}
Let $D$ be a divisor on a metric graph $\Gamma$, with $\psi_0, \ldots, \psi_r$ piecewise linear functions in $R(D)$, and let
\[
\theta = \min \{ \psi_0, \ldots, \psi_r \}.
\]
Let $\Gamma_j \subset \Gamma$ be the closed set where $\theta$ is equal to $\psi_j$.  Then $\ddiv( \theta ) +D$ contains a point $v \in \Gamma_j$ if and only if $v$ is in either
\begin{enumerate}
\item  the divisor $\ddiv( \psi_j ) + D$, or
\item  the boundary of $\Gamma_j$.
\end{enumerate}
\end{ShapeLemma}

\begin{proposition}  \label{prop:threeshape}
Consider a set of at most three non-departing permissible functions from the set $\{ \varphi_{ij} \}$ on a loop $\gamma_k$ and assume that all of the functions take the same value at $w_k$.  Then there is a point of $\gamma_k$ at which one of these functions is strictly less than the others.
\end{proposition}

\begin{proof}
We will consider the case where there are exactly three functions $\varphi$, $\varphi'$ and $\varphi''$.  The cases where there are one or two functions follows from a similar, but simpler, argument.  Since $\varphi$ is permissible on $\gamma_k$, we have $s_{k} (\varphi) \leq s_{k}(\theta)$, and since $\varphi$ is non-departing, we have $s_{k+1} (\varphi) = s_{k+1} (\theta)$.  The same holds for $\varphi'$ and $\varphi''$.

Let $\vartheta$ be the pointwise minimum of the three functions.  Since the slope of $\vartheta$ along any tangent direction agrees with that of one of the three, the incoming slope from the left at $v_k$ is at most the outgoing slope to the right at $w_k$, which is equal to $s_{k+1}(\theta)$.

It follows that the restriction $(D + \ddiv \vartheta)|_{\gamma_k}$ has degree at most 2.  Hence $\gamma_k \smallsetminus \mathrm{Supp} (D + \ddiv \vartheta)$ consists of at most two connected components.  By the Shape Lemma for Minima, the boundary points of a region where a function $\varphi$ achieves the minimum are contained in the support.  Therefore, the region where any one of the functions $\varphi$ achieves the minimum is either one of these connected components, or the union of both.

Since all three functions agree at $w_k$, and no two functions agree on the whole loop $\gamma_k$, we can narrow down the combinatorial possibilities as follows:  either all three functions agree on one region which contains $w_k$ and one of the three achieves the minimum uniquely on the other region, or two different pairs of functions agree on the two different regions, and $w_k$ is in the boundary of both.  These two possibilities are illustrated in Figure~\ref{Fig:ThreeFunctions}.

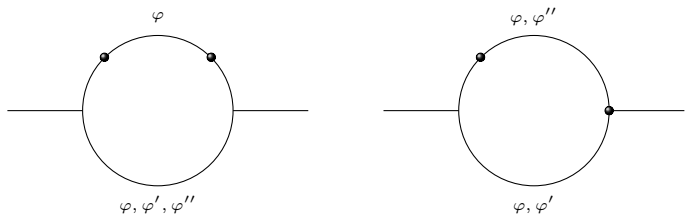
\begin{figure}[H]
\begin{tikzpicture}
\matrix[column sep=0.5cm] {

\begin{scope}[grow=right, baseline]
\draw (-1,0) circle (1);
\draw (-3,0)--(-2,0);
\draw (0,0)--(1,0);
\draw [ball color=black] (-0.29,0.71) circle (0.55mm);
\draw [ball color=black] (-1.71,0.71) circle (0.55mm);
\draw (-1,-1.25) node {{\tiny $\varphi, \varphi', \varphi''$}};
\draw (-1,1.25) node {{\tiny $\varphi$}};

\draw (4,0) circle (1);
\draw (2,0)--(3,0);
\draw (5,0)--(6,0);
\draw [ball color=black] (5,0) circle (0.55mm);
\draw [ball color=black] (3.29,0.71) circle (0.55mm);
\draw (4,-1.25) node {{\tiny $\varphi, \varphi'$}};
\draw (4,1.25) node {{\tiny $\varphi, \varphi''$}};

\end{scope}

\\};
\end{tikzpicture}

\caption{Two possibilities for where 3 functions may achieve the minimum.}
\label{Fig:ThreeFunctions}
\end{figure}
We now rule out the possibility illustrated on the right in Figure~\ref{Fig:ThreeFunctions}, where $w_k$ is in the boundary of both regions.  Note that one function, which we may assume to be $\varphi$, achieves the minimum on all of $\gamma_k$.  Furthermore, all three functions have the same slope along the outgoing bridge $\beta_k$, so $\varphi$ also achieves the minimum on the bridge.  Therefore $\vartheta$ is equal to $\varphi$ in a neighborhood of $w_k$.  However, $D+ \ddiv \vartheta$ contains $w_k$ and $D+ \ddiv \varphi$ does not, a contradiction.

We conclude that the minimum is achieved as depicted on the left in Figure~\ref{Fig:ThreeFunctions}, with all three functions achieving the minimum on a region that includes $w_k$, and one function achieving the minimum uniquely on the other region.  This proves the lemma.
\end{proof}

\subsection{Algorithm for constructing a maximally independent combination}
We now sketch the overall procedure that we will use to build a maximally independent combination $\theta = \min_{ij} \{ \varphi_{ij} + c_{ij} \}$ with slopes as specified in Section~\ref{sec:slopes}.  In this algorithm, we move from left to right across each of the three blocks where $s_k(\theta)$ is constant, assigning one function $\varphi_{ij}$ to each non-lingering loop and adjusting the coefficients $c_{ij}$ so that each function achieves the minimum on the loop to which it is assigned. At the end of each block, we start the next block by choosing coefficients such that $\theta$ bends at the midpoint of the bridge between blocks; the main interesting part in the algorithm is what happens within each block.

We now list a few of the key properties of the algorithm:
\renewcommand{\theenumi}{\roman{enumi}}
\begin{enumerate}
\item Once a function has been assigned to a loop, it always achieves the minimum uniquely at some point on that bridge or loop  (Lemma~\ref{Lem:VAConfig}).
\item A function never achieves the minimum on any loop to the right of the loop to which it is assigned  (Lemma~\ref{Lem:VANotToRight}).
\item Coefficients are initialized to $\infty$, and functions are assigned a finite coefficient at the first loop or bridge where they become permissible.
\item After the initial assignment of a finite coefficient, subsequent adjustments to this coefficient are smaller and smaller perturbations.  This is related to the fact that the edges get shorter and shorter as we move from left to right across the graph.  See the inequalities on edge lengths in Definition~\ref{Def:Admissible}.
\item Coefficients are only adjusted upward.  This ensures that once a function is assigned and achieves the minimum uniquely on a loop, it always achieves the minimum uniquely on that loop.
\item Exactly one function is assigned to each of the 21 non-lingering loops, and the remaining seven functions achieve the minimum on either the leftmost bridge, the rightmost bridge, or the two bridges between blocks, as described above.
\end{enumerate}

\noindent The algorithm terminates when we reach the rightmost bridge, at which point each of the 28 functions $\varphi_{ij} + c_{ij}$ achieves the minimum uniquely at some point on the graph. We now sketch the main steps.

\medskip

\noindent \textbf{Start at the first bridge.}
Start at the leftmost bridge $\beta_0$ and initialize $c_{66} = 0$.  Initialize $c_{56}$ so that $\varphi_{56} + c_{56}$ equals $\varphi_{66}$ at a point one third of the way from $w_0$ to $v_1$ on the first bridge $\beta_0$.   Initialize $c_{55}$ and $c_{46}$ so that $\psi_{55} + c_{55}$ and $\psi_{46} + c_{46}$ agree with $\varphi_{56} + c_{56}$ at a point two thirds of the way from $w_0$ to $v_1$.  Initialize all other coefficients $c_{ij}$ to $\infty$.  Note that $\varphi_{66}$ and $\varphi_{56}$ achieve the minimum uniquely on the first and second half of $\beta_0$, respectively.  Proceed to the first loop.

\medskip

\noindent \textbf{Loop subroutine.}
Each time we arrive at a loop $\gamma_k$, apply the following steps.

\medskip

\noindent \textbf{Loop subroutine, step 1:  Re-initialize unassigned coefficients.}
Suppose $\gamma_k$ is non-lingering.  Note that there are at least two unassigned permissible functions, by Lemma~\ref{lem:nonew}.  Find the unassigned permissible function $\varphi_{ij}$ that maximizes $\varphi_{ij}(w_k) + c_{ij}$.   Initialize the coefficients of the new permissible functions (if any) and adjust the coefficients of the other unassigned permissible functions upward so that they all agree with $\varphi_{ij}$ at $w_k$.    (The algorithm will be constructed so that the unassigned permissible functions are strictly less than all other functions, at every point in $\gamma_k$, even after this upward adjustment.)

\medskip

\noindent \textbf{Loop subroutine, step 2:  Assign departing functions.}
If there is a departing function, assign it to the loop.  (There is at most one, by Lemma~\ref{lem:onedeparts}.) Adjust the coefficients of the other permissible functions upward so that all of the functions agree at a point on the following bridge a short distance to the right of $w_k$, but far enough so that the departing function achieves the minimum uniquely on the whole loop.  This is possible because the bridge is much longer than the edges in the loop.  Proceed to the next loop.

\medskip

\noindent \textbf{Loop subroutine, step 3: Skip lingering loops.}
If $\gamma_k$ is a lingering loop, do nothing and proceed to the next loop.

\medskip

\noindent \textbf{Loop subroutine, step 4: Otherwise, use Proposition~\ref{prop:threeshape}.}
By Lemma~\ref{Lem:VAAtMostThree}, there are at most 3 non-departing functions.  By Proposition~\ref{prop:threeshape}, there is one that achieves the minimum uniquely at some point of $\gamma_k$.  We assign this function to the loop and adjust its coefficient upward slightly, enough so that it will never achieve the minimum on any loops to the right, but not so much that it does not achieve the minimum uniquely on this loop.  (Specifically, we increase the coefficient of this function by $\frac{1}{3}m_k$; see Lemma~\ref{Lem:VANotToRight}, below.)  Proceed to the next loop.

\medskip

\noindent \textbf{Proceeding to the next loop.}
If the next loop is contained in the same block, then move right to the next loop and apply the loop subroutine.  Otherwise, the current loop is the last loop in its block.  In this case, proceed to the next block.

\medskip

\noindent \textbf{Proceeding to the next block.}
After applying the loop subroutine to the last loop in a block, there is exactly one unassigned permissible function.  This follows from Corollary~\ref{cor:counting}.  As we shall see, the unassigned permissible function already achieves the minimum uniquely on the outgoing bridge, without any further adjustments of coefficients.

If we are at the last loop $\gamma_g$, then proceed to the last bridge.  Otherwise, there are several new permissible functions on the first loop of the next block, as detailed in Lemma~\ref{Lem:VAAtMostThree}, above.  Initialize the coefficient of each new permissible function so that it is equal to $\theta$ at the midpoint of the bridge between the blocks, and then apply the loop subroutine.

\medskip

\noindent \textbf{The last bridge.}
Initialize the coefficient $c_{01}$ so that it equals $\theta$ at the midpoint of the last bridge.  Initialize $c_{00}$ so that it equals $\theta$ halfway between the midpoint and the rightmost endpoint of the last bridge.  Note that both of these functions now achieve the minimum uniquely at some point on the second half of the final bridge.  Output the coefficients $\{ c_{ij} \}$.

\bigskip

\subsection{Verifying the algorithm}
We now prove that the output
\[
\theta =  \min_{ij} \{ \varphi_{ij} + c_{ij} \}
\]
is maximally independent.

\begin{lemma}
\label{Lem:VANotToRight}
Suppose that $\varphi_{ij}$ is assigned to the loop $\gamma_k$.  Then $\varphi_{ij}$ does not achieve the minimum at any point to the right of $v_{k+1}$.
\end{lemma}

\begin{proof}
If, when running the algorithm, there is an unassigned departing function $\varphi_{ij}$ on $\gamma_k$, then $\varphi_{ij}$ is not permissible on loops $\gamma_{k'}$ for $k'>k$.  Thus, it cannot achieve the minimum on any of the loops to the right of $\gamma_k$.

Otherwise, note that since $\varphi_{ij}$ is permissible, the difference $\varphi_{ij} (v) - \theta (v)$ for any point $v \in \gamma_{k'}$ with $k'>k$ is at least
\[
\frac{1}{3}m_k - d \sum_{t = k+1}^{k'} m_t .
\]
By our assumptions on edge lengths, this expression is always positive.
\end{proof}

We now show that each function achieves the minimum uniquely at some point of the loop to which it is assigned.  The analogous statement about functions assigned to bridges will follow once we see that there is only one function assigned to each bridge between blocks, which we will see in the proof of Theorem~\ref{thm:vertexavoiding}.

\begin{lemma}
\label{Lem:VAConfig}
Suppose that $\varphi_{ij}$ is assigned to the loop $\gamma_k$.  Then there is a point $v \in \gamma_k$ where $\varphi_{ij}$ achieves the minimum, and no other function $\varphi_{i'j'}$ achieves the minimum at $v$.
\end{lemma}

\begin{proof}
If there is an unassigned departing function $\varphi_{ij}$ on $\gamma_k$, then by construction this is the only function that achieves the minimum at $w_k$.  Otherwise, any two permissible functions differ by an integer multiple of $\frac{1}{2}m_k$ at points whose distance from $w_k$ is a half-integer multiple of $m_k$.  By construction, there is such a point where the assigned function $\varphi_{ij}$ achieves the minimum and no other permissible function does.  Thus, by increasing the coefficient by $\frac{1}{3}m_k$, the assigned function still achieves the minimum at this point.
\end{proof}

\begin{proof}[Proof of Theorem~\ref{thm:vertexavoiding}]
By construction, the functions $\varphi_{66}$ and $\varphi_{56}$ achieve the minimum on the bridge $\beta_0$, and the functions $\varphi_{00}$ and $\varphi_{01}$ achieve the minimum on the bridge $\beta_g$.

We first show that every non-lingering loop has an assigned function.  To see this, suppose that there is a non-lingering loop with no assigned function, and let $\gamma_k$ be the first such loop.  If $\varphi_{ij}$ is a function that was permissible on an earlier loop but not permissible on $\gamma_k$, then there is a $k'<k$ such that $\varphi_{ij}$ is a departing permissible function on $\gamma_{k'}$.  By construction, this function must be assigned to loop $\gamma_{k'}$, or an earlier loop.  It follows that the number of such functions is at most the number of non-lingering loops $\gamma_{k'}$ with $k'<k$.  Now, for every other function $\varphi_{ij}$ that is permissible on the block containing $\gamma_k$, there must exist a non-lingering loop $\gamma_{k'}$, with $k'>k$, such that $\varphi_{ij}$ is a new permissible function on $\gamma_{k'}$.  Since there is at most one new permissible function per loop, we see that the number of permissible functions on the block containing $\gamma_k$ is fewer than the number of non-lingering loops.  By Corollary~\ref{cor:counting}, however, this is impossible.

Indeed, by Corollary~\ref{cor:counting}, the number of permissible functions on each of the 3 blocks is exactly one more than the number of non-lingering loops in that block.  Every non-lingering loop has an assigned function, and the remaining function achieves the minimum at the bridge immediately to the right of the block.  The result then follows from Lemma~\ref{Lem:VAConfig}.
\end{proof}

\begin{example}  \label{ex:randomtableau}
We now illustrate the maximally independent tropical linear combination of $ \{  \varphi_{ij} \}$ for the randomly generated tableau pictured in Figure~\ref{Fig:RandomTableau}.

\begin{center}
\begin{figure}[H]
\begin{tikzpicture}
\matrix[column sep=0.7cm, row sep = 0.7cm] {
\begin{scope}[node distance=0 cm,outer sep = 0pt]
	      \node[bsq] (11) at (2.5,0) {01};
	      \node[bsq] (21) [below = of 11] {02};
	      \node[bsq] (31) [below = of 21] {04};
	      \node[bsq] (12) [right = of 11] {03};
	      \node[bsq] (22) [below = of 12] {05};
	      \node[bsq] (32) [below = of 22] {08};
	      \node[bsq] (13) [right = of 12] {06};
	      \node[bsq] (23) [below = of 13] {07};
	      \node[bsq] (33) [below = of 23] {11};
	      \node[bsq] (14) [right = of 13] {09};
	      \node[bsq] (24) [below = of 14] {12};
	      \node[bsq] (34) [below = of 24] {14};
	      \node[bsq] (15) [right = of 14] {10};
	      \node[bsq] (25) [below = of 15] {16};
	      \node[bsq] (35) [below = of 25] {17};
	      \node[bsq] (16) [right = of 15] {13};
	      \node[bsq] (26) [below = of 16] {18};
	      \node[bsq] (36) [below = of 26] {20};
	      \node[bsq] (17) [right = of 16] {15};
	      \node[bsq] (27) [below = of 17] {19};
	      \node[bsq] (37) [below = of 27] {21};
\end{scope}
\\};
\end{tikzpicture}
\caption{A randomly generated $3 \times 7$ tableau.}
\label{Fig:RandomTableau}
\end{figure}
\end{center}

\smallskip

\noindent In this example, $z=7$ and $z'=15$.  The three rows of Figure~\ref{Fig:Config} are the blocks, with $7$, $8$, and $6$ loops.  The 45 black circles indicate the support of the divisor $D' = 2D + \ddiv \theta$, and each point in the support appears with coefficient 1, aside from points on the bridges $\beta_3$, $\beta_6$, and $\beta_{12}$ a short distance to the right of the previous loops, which appear with coefficient 2, as marked.  Each of the 28 functions $\varphi_{ij}$ achieves the minimum uniquely on one of the connected components of the complement of $\mathrm{Supp}(D')$.  Since $\rho=0$ in this example, there are no lingering loops.  Each loop is labeled by the indices $ij$ of the fuction $\varphi_{ij}$ assigned to that loop.

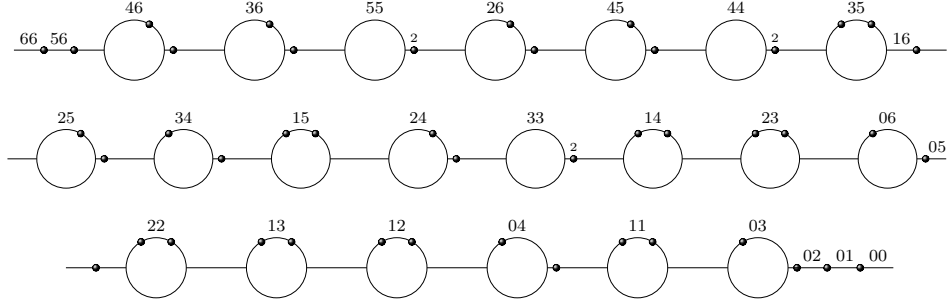
\begin{figure}[H]

\begin{center}
\scalebox{.8}{
\begin{tikzpicture}
\draw (-1.5,12)--(0,12);
\draw (-1.25,12.2) node {\footnotesize $66$};
\draw [ball color=black] (-1,12) circle (0.55mm);
\draw (-0.75,12.2) node {\footnotesize $56$};
\draw [ball color=black] (-0.5,12) circle (0.55mm);

\draw (0.5,12) circle (0.5);
\draw (0.5,12.7) node {\footnotesize $46$};
\draw [ball color=black] (0.75,12.43) circle (0.55mm);

\draw (1,12)--(2,12);
\draw [ball color=black] (1.15,12) circle (0.55mm);
\draw (2.5,12) circle (0.5);
\draw (2.5,12.7) node {\footnotesize $36$};
\draw [ball color=black] (2.75,12.43) circle (0.55mm);

\draw (3,12)--(4,12);
\draw [ball color=black] (3.15,12) circle (0.55mm);
\draw (4.5,12) circle (0.5);
\draw (4.5,12.7) node {\footnotesize $55$};

\draw (5,12)--(6,12);
\draw [ball color=black] (5.15,12) circle (0.55mm);
\draw (5.15,12.2) node {\tiny $2$};
\draw (6.5,12) circle (0.5);
\draw (6.5,12.7) node {\footnotesize $26$};
\draw [ball color=black] (6.75,12.43) circle (0.55mm);

\draw (7,12)--(8,12);
\draw [ball color=black] (7.15,12) circle (0.55mm);
\draw (8.5,12) circle (0.5);
\draw (8.5,12.7) node {\footnotesize $45$};
\draw [ball color=black] (8.75,12.43) circle (0.55mm);
\draw (9,12)--(10,12);

\draw [ball color=black] (9.15,12) circle (0.55mm);
\draw (10.5,12) circle (0.5);
\draw (10.5,12.7) node {\footnotesize $44$};

\draw (11,12)--(12,12);
\draw [ball color=black] (11.15,12) circle (0.55mm);
\draw (11.15,12.2) node {\tiny $2$};
\draw (12.5,12) circle (0.5);
\draw (12.5,12.7) node {\footnotesize $35$};
\draw (13.25,12.2) node {\footnotesize $16$};
\draw [ball color=black] (12.75,12.43) circle (0.55mm);
\draw [ball color=black] (12.25,12.43) circle (0.55mm);
\draw (13,12)--(14,12);
\draw [ball color=black] (13.5,12) circle (0.55mm);
\end{tikzpicture}
}

\bigskip

\scalebox{.78}{
\begin{tikzpicture}
\draw (-6.5,6)--(-6,6);
\draw (-5.5,6) circle (0.5);
\draw (-5.5,6.7) node {\footnotesize $25$};
\draw [ball color=black] (-5.25,6.43) circle (0.55mm);

\draw (-5,6)--(-4,6);
\draw [ball color=black] (-4.85,6) circle (0.55mm);
\draw (-3.5,6) circle (0.5);
\draw (-3.5,6.7) node {\footnotesize $34$};
\draw [ball color=black] (-3.75,6.43) circle (0.55mm);

\draw (-3,6)--(-2,6);
\draw [ball color=black] (-2.85,6) circle (0.55mm);
\draw (-1.5,6) circle (0.5);
\draw (-1.5,6.7) node {\footnotesize $15$};
\draw [ball color=black] (-1.75,6.43) circle (0.55mm);
\draw [ball color=black] (-1.25,6.43) circle (0.55mm);

\draw (-1,6)--(0,6);
\draw (0.5,6) circle (0.5);
\draw (0.5,6.7) node {\footnotesize $24$};
\draw [ball color=black] (0.75,6.43) circle (0.55mm);

\draw (1,6)--(2,6);
\draw [ball color=black] (1.15,6) circle (0.55mm);
\draw (2.5,6) circle (0.5);
\draw (2.5,6.7) node {\footnotesize $33$};

\draw (3,6)--(4,6);
\draw [ball color=black] (3.15,6) circle (0.55mm);
\draw (4.5,6) circle (0.5);
\draw (4.5,6.7) node {\footnotesize $14$};
\draw (3.15,6.2) node {\tiny $2$};
\draw [ball color=black] (4.75,6.43) circle (0.55mm);
\draw [ball color=black] (4.25,6.43) circle (0.55mm);

\draw (5,6)--(6,6);
\draw (6.5,6) circle (0.5);
\draw (6.5,6.7) node {\footnotesize $23$};
\draw [ball color=black] (6.25,6.43) circle (0.55mm);
\draw [ball color=black] (6.75,6.43) circle (0.55mm);

\draw (7,6)--(8,6);
\draw (8.5,6) circle (0.5);
\draw (8.5,6.7) node {\footnotesize $06$};
\draw [ball color=black] (8.25,6.43) circle (0.55mm);
\draw (9,6)--(9.5,6);
\draw [ball color=black] (9.15,6) circle (0.55mm);
\draw (9.35,6.2) node {\footnotesize $05$};
\end{tikzpicture}
}

\bigskip

\scalebox{.8}{
\begin{tikzpicture}
\draw (-2,3)--(-1,3);
\draw [ball color=black] (-1.5,3) circle (0.55mm);
\draw (-0.5,3) circle (0.5);
\draw (-0.5,3.7) node {\footnotesize $22$};
\draw [ball color=black] (-0.75,3.43) circle (0.55mm);
\draw [ball color=black] (-0.25,3.43) circle (0.55mm);

\draw (0,3)--(1,3);
\draw (1.5,3) circle (0.5);
\draw (1.5,3.7) node {\footnotesize $13$};
\draw [ball color=black] (1.25,3.43) circle (0.55mm);
\draw [ball color=black] (1.75,3.43) circle (0.55mm);

\draw (2,3)--(3,3);
\draw (3.5,3) circle (0.5);
\draw (3.5,3.7) node {\footnotesize $12$};
\draw [ball color=black] (3.25,3.43) circle (0.55mm);
\draw [ball color=black] (3.75,3.43) circle (0.55mm);

\draw (4,3)--(5,3);
\draw (5.5,3) circle (0.5);
\draw (5.5,3.7) node {\footnotesize $04$};
\draw [ball color=black] (5.25,3.43) circle (0.55mm);

\draw (6,3)--(7,3);
\draw [ball color=black] (6.15,3) circle (0.55mm);
\draw (7.5,3) circle (0.5);
\draw (7.5,3.7) node {\footnotesize $11$};
\draw [ball color=black] (7.25,3.43) circle (0.55mm);
\draw [ball color=black] (7.75,3.43) circle (0.55mm);

\draw (8,3)--(9,3);
\draw (9.5,3) circle (0.5);
\draw (9.5,3.7) node {\footnotesize $03$};
\draw [ball color=black] (9.25,3.43) circle (0.55mm);
\draw [ball color=black] (10.15,3) circle (0.55mm);

\draw (10,3)--(11.75,3);
\draw [ball color=black] (10.65,3) circle (0.55mm);
\draw [ball color=black] (11.2,3) circle (0.55mm);
\draw (10.4,3.2) node {\footnotesize $02$};
\draw (10.95,3.2) node {\footnotesize $01$};
\draw (11.5,3.2) node {\footnotesize $00$};
\end{tikzpicture}
}
\end{center}
\caption{Configuration of the functions $\varphi_{ij}$ for the tableau depicted in Figure~\ref{Fig:RandomTableau}.}
\label{Fig:Config}
\end{figure}

\noindent On the loops without unassigned departing permissible functions, the assigned permissible function achieves the minimum on only part of the loop.  These are the loops $\gamma_k$ for $k \in \{ 7, 10, 13, 14, 16, 17, 18, 20 \}$. On the other loops, the departing permissible function achieves the minimum on the whole loop.

The unassigned permissible functions at the end of the first, second, and third blocks are $\varphi_{16}$, $\varphi_{05}$, and $\varphi_{02}$, respectively.  Note that each of these functions achieves the minimum uniquely at the beginning of the bridge to the right of the corresponding block.

Note that the restriction of $D'$ to $\gamma_k \cup \beta_k$ has degree 3 for $k \in \{7, 15, 21\}$ and degree $2$ otherwise.  Since the restriction of $2D$ to $\gamma_k \cup \beta_k$ is $2$ for all $1 \leq k \leq 21$, this reflects the fact that $s_k(\theta)$ is constant on each block and decreases by 1 when passing from one block to the next.
\end{example}

\section{Tropicalization of linear series on chains of loops}
\label{Sec:Loops}

In this section, we make first steps toward classifying linear series of degree $d$ and rank $r$ on a chain of $g$ loops with admissible edge lengths, extending the combinatorial classification of divisor classes of degree $d$ and rank $r$ in \cite{tropicalBN}.

We maintain the notation from Sections~\ref{Sec:Slopes}-\ref{sec:chainsofloops}.  In particular, $X$ is a curve of genus $g$ whose skeleton $\Gamma$ is a chain of $g$ loops with admissible edge lengths, and $D_X$ is a divisor of degree $d$ on $D_X$, with a linear series $V \subset \cL(D_X)$ of rank $r$.  Replacing $D_X$ by a linearly equivalent (and not necessarily effective) divisor, we can and do assume $D = \Trop(D_X)$ is a break divisor.  This means that the restriction of $D$ to each loop is effective of degree 1, and we write $x_k$ for the distance from $v_k$, in the counterclockwise direction, to the point of $D$ on $\gamma_k$.

\subsection{Slope vectors and ramification}
By Lemma~\ref{Lem:NumberOfSlopes}, for each tangent vector in $\Gamma$, there are exactly $r+1$ distinct slopes of functions in $\Sigma = \trop(V)$.  The incoming and outgoing (rightward) slopes at each loop play a special role in our analysis.  Recall that we write $s_k(\varphi)$ for the incoming slope of $\varphi$ at $v_{k}$.  We will write $s'_k(\varphi)$ for the outgoing slope of $\varphi$ at $w_k$.

As in Section~\ref{Sec:Slopes}, we write $s_k(\Sigma) = (s_k[0], \ldots, s_k[r])$ for the $r+1$ slopes that occur as $s_k (\varphi)$ for $\varphi \in \Sigma$, written in increasing order.  Similarly, we write $s'_k(\Sigma) = (s'_k[0], \ldots, s'_k[r])$ for the $r+1$ possible slopes $s'_k(\varphi)$ for $\varphi \in \Sigma$.

We keep track of how these vectors of slopes change as one proceeds from left to right across the graph, as follows.

\begin{proposition} \label{Prop:LingeringLatticePath}
The difference between the vectors of incoming and outgoing slopes at $\gamma_k$ are bounded as follows:
\begin{displaymath}
s'_k[i] - s_{k}[i] \leq \left\{ \begin{array}{ll}
1 & \textrm{if $x_k \equiv (s_{k}[i]+1)m_k \ (\mathrm{mod} \  \ell_k + m_k)$} \\
  & \textrm{and $s_k [i+1] \neq s_k[i]+1$},\\
0 & \textrm{otherwise.}
\end{array} \right.
\end{displaymath}
Furthermore, the difference between incoming and outgoing slopes at $\beta_k$ are bounded by $s_{k+1}[i] - s'_k[i] \leq 0$ for all $i$ and $k$.
\end{proposition}

\begin{proof}
To prove the inequality for loops, first note that, by Lemma~\ref{Lemma:Existence}, there is a function $\varphi_i \in \Sigma$ such that $s_{k} (\varphi_i) \leq s_{k} [i]$ and $s'_k (\varphi_i) \geq s'_k [i]$.   It now suffices to show that $s'_k(\varphi_i) - s_{k-1}(\varphi_i)$ is bounded by 1 if $x_k = (s_{k-1}(i) +1)m_k$ mod $(\ell_k + m_k)$, and by $0$ otherwise.  The fact that these bounds hold for any function in $R(D)$ is the essential content of \cite[Example 2.1]{tropicalBN}.

The proof of the inequality for bridges is even simpler.  There is a function $\varphi_i \in \Sigma$ such that $s_{k+1}(\varphi_i) \geq s_{k+1} [i]$ and $s'_k (\varphi_i) \leq s'_k [i]$.  Since the support of $D$ is disjoint from the interior of the bridge $\beta_k$, the slope of $\varphi_i$ can only decrease along the bridge, and hence $s_{k+1} (\varphi_i) \leq s'_k(\varphi_i)$, as required.
\end{proof}

\begin{remark}
\label{Rem:Breaks}
In \cite{tropicalBN} and several subsequent papers, the divisor was chosen to be $w_0$-reduced, rather than a break divisor.  This convention had the mildly inelegant consequence that the inequalities analogous to those in Proposition~\ref{Prop:LingeringLatticePath} differ depending on whether or not $D$ has a point on $\gamma_k$.  By working with break divisors, which have exactly one point on every loop, we avoid this inelegance.
\end{remark}

By Proposition~\ref{Prop:LingeringLatticePath}, for a fixed $k$ we have
\[
s'_k [i] - s_{k} [i] \leq 1 ,
\]
with equality for at most one value of $i$.  We define the \emph{multiplicity} of the $k$th loop to be the total amount by which the coordinatewise difference $s'_k(\Sigma) - s_k(\Sigma)$ deviates from this bound.  More precisely, the multiplicity of the $k$th loop is
\[
\mu (\gamma_k) = 1 + \sum_{i=0}^r \left( s_{k} [i] - s'_k [i] \right) .
\]
By Proposition~\ref{Prop:LingeringLatticePath}, the multiplicity of each loop is nonnegative.  Similarly, we define the \emph{multiplicity} of the $k$th bridge to be
\[
\mu (\beta_k) = \sum_{i=0}^r \left( s'_k [i] - s_{k+1} [i] \right).
\]
A given loop $\gamma_k$ may have positive multiplicity in several ways.  We provide the following definition.

\begin{definition}
We say that $\gamma_k$ is a \emph{decreasing loop} if there is a value $h$ such that $s_{k} [h] > s'_k [h]$.
Similarly, we say that $\beta_k$ is a \emph{decreasing bridge} if there is a value $h$ such that $s'_k [h] > s_{k+1}[h]$.
\end{definition}

Note that all bridges with positive multiplicity are decreasing bridges, but not all loops with positive multiplicity are decreasing loops.  For example, in the case where the divisor $D$ is vertex avoiding, the loops with positive multiplicity are precisely the lingering loops, and these satisfy $s_{k}[i] = s'_k [i]$ for all $i$.

The multiplicities of loops and bridges record where, and by how much, the rightward slopes of $\Sigma$ fail to increase or stay the same, as expected.  We also keep track of the extent to which the rightward slopes of $\Sigma$ may be lower than expected at the leftmost point $w_0$, or higher than expected at the rightmost point $v_{g+1}$.

For a vertex avoiding divisor of degree $d$ and rank $r$, the rightward slopes at $w_0$ are $(d-g-r, d-g-r+1, \ldots, d-g)$.  We define the \emph{ramification weight at} $w_0$ to be $$\alpha (v_0) = \sum_{i=0}^r (d-g -r + i - s'_0 [i]).$$
Similarly, for a vertex avoiding divisor, the rightward slopes at $v_{g+1}$ are $(0, \ldots, r)$, and we define the \emph{ramification weight at} $v_{g+1}$ to be $$\alpha (v_{g+1}) = \sum_{i=0}^r (s_{g+1} [i] - i).$$    Our choice of terminology reflects the fact that these are ramification weights of the reductions of $V$ to $\kappa (X_{w_0})$ and $\kappa (X_{v_{g+1}})$, respectively.

The main theorem of \cite{tropicalBN} says that the space of divisor classes of degree $d$ and rank $r$ on $\Gamma$ has dimension $\rho = g - (r+1)(g-d+r)$.  We have the following analogue for tropicalizations of linear series.

\begin{theorem}
\label{Thm:BNThm}
The sum of the multiplicities of all loops and bridges plus the ramification weights at $w_0$ and $v_{g+1}$ is the Brill-Noether number $\rho = g-(r+1)(g-d+r)$.
\end{theorem}

\begin{proof}
Starting from the definitions of the multiplicities of the loops and bridges and then collecting and canceling terms, we have
\[
\sum_{k=1}^g \mu (\gamma_k) + \sum_{k=0}^{g} \mu (\beta_k) =
g + \sum_{i=0}^r s'_0 [i] - s_{g+1} [i] .
\]
Moreover,
\[
\alpha (w_0) = (r+1)(d-g) - {{r+1}\choose{2}} - \sum_{i=0}^r (s'_0 [i] + i) .
\]
and
\[
\alpha (v_{g+1}) = \sum_{i=0}^r (s_{g+1}[i]) - {{r+1}\choose{2}}.
\]
Adding these together and again collecting and canceling terms gives
\[
g - (r+1)(g-d) - 2{{r+1}\choose{2}} = g-(r+1)(g-d+r),
\]
as required.
\end{proof}

\noindent  Note that the multiplicity of each loop and each bridge is nonnegative, as are the ramification weights.  In the vertex avoiding case, the lingering loops have multiplicity one, and the multiplicities of bridges and ramification weights are all zero.  In general, the distribution of multiplicities and ramification weights is a useful indication of how and where the tropicalization of a given linear series differs essentially from the vertex avoiding case.

\subsection{Switching loops and bridges}

We begin with the following definition.

\begin{definition}
A loop $\gamma_k$ is a \emph{switching loop} if there is some $\varphi \in \Sigma$ such that
\[
s_{k} (\varphi) \leq s_{k} [h] \mbox{ and } s'_k (\varphi) \geq s'_k [h+1] .
\]
Similarly, $\beta_k$ is a \emph{switching bridge} if there is some $\varphi \in \Sigma$ such that
\[
s'_k (\varphi) \leq s'_k [h] \mbox{ and } s_{k+1} (\varphi) \geq s_{k+1} [h+1] .
\]
\end{definition}

\noindent The terminology is chosen to emphasize that, on a switching loop, it is possible for a function $\varphi$ with slope $s_{k} [h]$ to the left to ``switch'' to having slope $s'_k [h+1]$ to the right.  Similarly, on a switching bridge, it is possible for a function to have slope $s'_k [h]$ at the beginning of the bridge and $s_{k+1} [h+1]$ at the end of the bridge.

\subsubsection{Classifying switching loops}
Figure~\ref{Fig:SwitchingLoops} schematically depicts a classification of switching loops of multiplicity at most 2.  The dots on the left of each picture represent the values of $s_k [h]$ and $s_k [h+1]$, while the dots on the right represent the values of $s'_k [h]$ and $s'_k [h+1]$.  For instance, in the second picture, $s_{k} [h+1] = s'_k [h+1]$ and $s_k [h] = s'_k [h] + 1$.  The arrow from $s_k[h]$ to $s'_k[h+1]$ indicates that there is a function $\varphi \in R(D)$ with
\[
s_k (\varphi) = s_k [h] \mbox{ and } s'_k (\varphi) = s'_k [h + 1].
\]
When the arrow has positive slope, this implies that $x_k = (s_k [h] +1)m_k$.  The first picture depicts a switching loop of multiplicity 1, and the other 3 depict switching loops of multiplicity 2.

\begin{figure}[h!]
\begin{tikzpicture}

\draw [ball color=black] (0,0) circle (0.55mm);
\draw [ball color=black] (0,0.5) circle (0.55mm);
\draw [ball color=black] (1,0) circle (0.55mm);
\draw [ball color=black] (1,0.5) circle (0.55mm);
\draw [->] (0.25,0)--(0.75,0.5);

\draw [ball color=black] (3,0.5) circle (0.55mm);
\draw [ball color=black] (3,1) circle (0.55mm);
\draw [ball color=black] (4,0) circle (0.55mm);
\draw [ball color=black] (4,1) circle (0.55mm);
\draw [->] (3.25,0.5)--(3.75,1);

\draw [ball color=black] (6,0) circle (0.55mm);
\draw [ball color=black] (6,1) circle (0.55mm);
\draw [ball color=black] (7,0) circle (0.55mm);
\draw [ball color=black] (7,0.5) circle (0.55mm);
\draw [->] (6.25,0)--(6.75,0.5);

\draw [ball color=black] (9,0.5) circle (0.55mm);
\draw [ball color=black] (9,1) circle (0.55mm);
\draw [ball color=black] (10,0) circle (0.55mm);
\draw [ball color=black] (10,0.5) circle (0.55mm);
\draw [->] (9.25,.5)--(9.75,0.5);

\end{tikzpicture}
\caption{Switching Loops of Multiplicity at Most Two.}
\label{Fig:SwitchingLoops}
\end{figure}
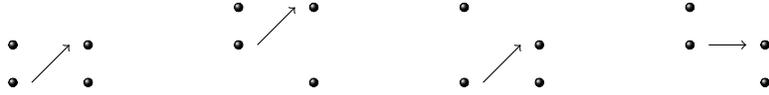

\noindent In this classification we see that, if $\mu (\gamma_k) \leq 2$, then there is at most one value of $h$ such that there is a function $\varphi \in R(D)$ with
\[
s_{k} (\varphi) \leq s_{k} [h] \mbox{ and } s'_k (\varphi) \geq s'_k [h+1] .
\]
We will say that such a loop \emph{switches slope $h$}.

\subsubsection{Classification of switching bridges}
We can similarly classify switching bridges of multiplicity at most 2.  There is only one possibility, depicted in Figure~\ref{Fig:SwitchingBridge}, and it has multiplicity 2.

\begin{figure}[h!]
\begin{tikzpicture}

\draw [ball color=black] (9,0.5) circle (0.55mm);
\draw [ball color=black] (9,1) circle (0.55mm);
\draw [ball color=black] (10,0) circle (0.55mm);
\draw [ball color=black] (10,0.5) circle (0.55mm);
\draw [->] (9.25,.5)--(9.75,0.5);

\end{tikzpicture}
\caption{A Switching Bridge of Multiplicity Two.}
\label{Fig:SwitchingBridge}
\end{figure}
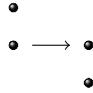

\noindent  The figure indicates that $s'_k[h] = s_{k+1}[h] + 1$ and $s'_k[h+1] = s'_{k+1} [h+1] + 1$.  The arrow indicates that there is a function $\varphi \in R(D)$ with
\[
s'_k (\varphi) = s'_k [h] \mbox{ and } s_{k+1} (\varphi) = s_{k+1} [h + 1].
\]
We will say that such a bridge \emph{switches slope $h$}.  The behavior of a switching bridge of multiplicity 2 is very closely analogous to the behavior of the pencil in Example~\ref{Ex:Interval}.  We will use the following lemma in Section~\ref{Sec:SwitchBridgeCase}.

\begin{lemma}
\label{Lem:BridgeSwitch}
Let $\beta_k$ be a switching bridge of multiplicity $2$ that switches slope $h$.  Then there is a point $x \in \beta_k$ such that, if we write $s_x, s'_x$ for the incoming and outgoing slopes, respectively, we have
\[
s_x [i] = s'_k [i] \mbox{ and } s'_x [i] = s_{k+1} [i] \mbox{ for all } i.
\]
\end{lemma}

\begin{proof}
By assumption, there is a function $\varphi \in \Sigma$ that has constant slope $s'_k[h] = s_{k+1} [h+1]$ along the entire bridge $\beta_k$.  Thus, as in Example~\ref{Ex:Interval}, there cannot be a rightward tangent vector $\zeta$ in $\beta_k$ where
\[
s_{\zeta} [h+1] = s'_k [h+1] \mbox{ and } s_{\zeta} [h] = s_{k+1} [h],
\]
and there must exist a point $x \in \beta_k$ such that $s_x [h] = s'_k [h]$ and $s'_x [h+1] = s_{k+1} [h+1]$.  Since we must have
\[
s_x [h] < s_x [h+1] \leq s'_k [h+1] = s'_k [h] + 1,
\]
we see that $s_x [h+1] = s'_k [h+1]$.  Similarly, we see that $s'_x [h] = s_{k+1} [h]$.
\end{proof}

\subsubsection{Existence of distinguished functions} In the vertex avoiding case, there are distinguished functions $\varphi_i \in R(D)$, for all $0 \leq i \leq r$, unique up to adding a constant, such that $s_k(\varphi) = s_k[i]$ and $s'_k(\varphi_i) = s'_k[i]$ for all $k$, and these functions are always in $\Sigma$.  We now identify the subset of indices in $\{ 0, \ldots, r\}$ for which such distinguished functions exist, in the general case.

\begin{lemma}
\label{Lem:GenericFns}
For each $0 \leq i \leq r$, either there is a function $\varphi_i \in \Sigma$ such that
\[
s_k (\varphi_i) = s_k [i] \mbox{ and } s'_k (\varphi_i) = s'_k [i] \mbox{ for all } k,
\]
or there is one of the following:
\begin{enumerate}
\item  a loop that switches slope $i$
\item  a loop that switches slope $i-1$
\item  a bridge that switches slope $i$, or
\item  a bridge that switches slope $i-1$.
\end{enumerate}
\end{lemma}

\begin{proof}
By Lemma~\ref{Lemma:Existence}, there is a function $\varphi_i \in \Sigma$ with $s'_0 (\varphi_i) \leq s'_0 [i]$ and $s_{g+1} (\varphi_i) \geq s_{g+1} [i]$.  By induction, if there are no loops or bridges that switch slope $i$, we see that $s_k (\varphi_i) \leq s_k [i]$ and $s'_k (\varphi_i) \leq s'_k [i]$ for all $k$.  Similarly, if there are no loops or bridges that switch slope $i-1$, we see that $s_k (\varphi_i) \geq s_k [i]$ and $s'_k (\varphi_i) \geq s'_k [i]$ for all $k$.
\end{proof}

\begin{corollary}
\label{Cor:GenericFns}
If there are no switching loops or switching bridges, then for all $i$ there is a function $\varphi_i \in \Sigma$ such that
\[
s_k (\varphi_i) = s_k [i] \mbox{ and } s'_k (\varphi_i) = s'_k [i] \mbox{ for all } k.
\]
\end{corollary}

\noindent There do exist cases that are not vertex avoiding that nevertheless have no switching loops or bridges.  In these cases, we will produce a maximally independent combination of the distinguished functions $\varphi_i$ given by Corollary~\ref{Cor:GenericFns} in a manner similar to that of Section~\ref{Sec:VertexAvoiding}.  In cases with switching loops and bridges, the analysis becomes more difficult, and requires careful consideration of the possibilities for $\Sigma$, identifying functions analogous to $\psi_A$, $\psi_B$ and $\psi_C$ in Example~\ref{Ex:Interval} which can be used as substitutes in the construction of a maximally independent combination.  These cases occupy most of Section~\ref{Sec:Generic}.

\section{Building blocks and the master template}
\label{Sec:Construction}

Throughout the rest of the paper, we assume that $\rho \leq 2$.
We now begin our construction of a maximally independent combination in the case where the divisor $D$ is not vertex avoiding.  Our goal is to produce a set of ${{r+2}\choose{2}}$ functions of the form $\varphi + \varphi'$ with $\varphi, \varphi' \in \Sigma$, and a maximally independent combination of these functions.  Because it is difficult to identify a suitable collection of functions in $\Sigma$ when $D$ is not vertex avoiding, we first produce an intermediary tropical linear combination $\theta$ of pairwise sums of certain simpler functions in $R(D)$ that may or may not themselves be contained in $\Sigma$.  We refer to these functions in $R(D)$ as \emph{building blocks}, and we call the tropical linear combination $\theta$ that we construct from pairwise sums of building blocks the \emph{master template}.

\subsection{Building Blocks}
\label{Sec:Extremals}

Roughly speaking, the building blocks are the functions in $R(D)$ that have constant slopes along bridges and behave as much as possible like the functions $\varphi_i$ in the vertex avoiding case, while respecting the constraints on functions in $\Sigma$ imposed by the slope vectors.  The simplest are those functions $\varphi$ whose incoming and outgoing slopes at each loop $\gamma_k$ satisfy $s_k[\varphi] = s_k[i]$ and $s'_k[\varphi] = s_k[i]$, for some fixed $i$.  Our definition is motivated by the extremals of \cite{HMY12} and by Example~\ref{Ex:Interval}.

Before giving the general definition, we introduce the useful auxiliary notion of \emph{incoming and outgoing slope indices}.  These indices account for how the incoming and outgoing slopes of a function $\varphi \in R(D)$ relate to the slope vectors $s_k$ and $s'_k$ determined by $\Sigma$, adjusted for any bends at $v_k$ and $w_k$.  Let $d_{v_k}(\varphi) = \deg_{v_k}(D + \ddiv(\varphi))$ and $d_{w_k}(\varphi) = \deg_{w_k}(D + \ddiv(\varphi))$.  Then the \emph{incoming slope index} $\tau_k(\varphi)$ is
\[
\tau_k(\varphi) := \min \{ i \, \vert \, s_k[i] \geq s_k(\varphi) - d_{v_k}(\varphi) \}.
\]
Similarly, the \emph{outgoing slope index} $\tau'_k(\varphi)$ is
\[
\tau'_k(\varphi) := \max \{ i \, \vert \, s'_k[i] \leq s'_k(\varphi) + d_{w_k}(\varphi) \}.
\]
Intuitively, one may think that a function with incoming (resp. outgoing) slope index $i$ behaves most like a typical function in $\Sigma$ with incoming slope $s_k[i]$ (resp. a typical function in $\Sigma$ with outgoing slope $s'_k[i]$), near the left hand side (resp. right hand side) of the loop $\gamma_k$.

\begin{remark}
\label{Rem:Equivalent}
The integer $s_k (\varphi) - d_{v_k}(\varphi)$ is equal to the sum of the slopes of $\varphi$ along the two rightward pointing tangent vectors based at $v_k$.  Because of this, if the restriction of two functions $\varphi$ and $\phi$ to $\gamma_k$ differ by a constant, then by definition we have $\tau_k (\varphi) = \tau_k (\phi)$.  Similarly, if the restriction of $\varphi$ and $\phi$ to $\gamma_k$ differ by a constant, then $\tau'_k (\varphi) = \tau'_k (\phi)$.
\end{remark}

Functions whose slope indices decrease when moving from left to right across the graph can be expressed as tropical linear combinations of functions with constant slopes along bridges whose slope indices do not decrease.  Moreover, since we restrict our attention to the case $\rho \leq 2$, the classification of switching loops and bridges discussed in the previous section ensures that the slope indices of functions in $\Sigma$ never increase by more than 1 when crossing any loop or bridge, and only when that loop or bridge switches the relevant slope.  For this reason, we only consider building blocks whose slope indices satisfy this condition.

\begin{definition} \label{def:buildingblock}
A building block is a function $\varphi \in R(D)$ with constant slope along each bridge, whose slope index sequence $\tau'_0(\varphi), \tau_1(\varphi), \tau'_1(\varphi), \ldots, \tau'_g, \tau_{g+1}(\varphi)$ is nondecreasing and satisfies
\begin{enumerate}
\item  $\tau'_k \leq \tau_k + 1$, with $\tau'_k = \tau_k$ if $\gamma_k$ does not switch slope $\tau_k$;
\item  $\tau_{k+1} \leq \tau'_{k} + 1$, with $\tau_{k+1} = \tau'_k$ if $\beta_k$ does not switch slope $\tau'_k$.
\end{enumerate}
\end{definition}

\noindent  In the vertex avoiding case, the building blocks are precisely the distinguished functions $\varphi_i$, which have $\tau_k(\varphi_i) = \tau'_k(\varphi_i) = i$ for all $k$.   In general, for any $0 \leq i \leq r$, there is a (not necessarily unique) building block $\phi$ with constant slope index sequence $\tau_k(\phi) = \tau'_k (\phi) = i$.  If there are no switching loops or bridges, then the slope index sequence of any building block is constant.

\begin{remark}
If $\varphi$ is a building block and $s_k [\tau_k(\varphi)] \leq s'_k [\tau'_k(\varphi)]$ for all $k$, then $\ddiv (\varphi) + D$ does not contain a smooth cut set, and it follows that $\varphi \in R(D)$ is an \emph{extremal}, as defined in \cite{HMY12}.  On the other hand, if $s_k [\tau_k(\varphi)] > s'_k [\tau'_k(\varphi)]$ for some $k$, then $\varphi$ is not necessarily an extremal.  In such cases, $\gamma_k$ must be a decreasing loop.  In any case, every function in $R(D)$ can be written as a tropical linear combination of extremals, and hence our main constructions could be rephrased in terms of extremals.  We find it simpler to work with building blocks, as defined above, since they are more closely tailored to the properties of $\Sigma$.
\end{remark}

\begin{example}
We may think of Example~\ref{Ex:Interval} as a chain of one loop, where the loop is located at the point $x$ and its edges have length zero.  This loop is a switching loop, and there are 3 possibilities for the slope index sequence of a building block, namely
\[
(\tau'_0,\tau_1,\tau'_1,\tau_2) = (0,0,0,0), (0,0,1,1), (1,1,1,1).
\]
There is a unique building block with each of these slope index sequences.  These building blocks are $\psi_0 , \psi_1,$ and $\psi^{\infty}$, respectively.  Note that, since $\Sigma$ is the tropicalization of a pencil and these three functions are tropically independent, it is impossible for all of them to be in $\Sigma$.  Nevertheless, every function in $\Sigma$ can be written as a tropical linear combination of these three functions.
\end{example}

\begin{example}
Suppose that $\varphi$ is a building block with $s'_{k} (\varphi) = s_k(\varphi)-1$, i.e., the outgoing slope at $\gamma_k$ is one less than the incoming slope.  There are three ways that this could happen.  Either the bridge $\beta_{k-1}$ is a decreasing bridge, the bridge $\beta_k$ is a decreasing bridge, or the loop $\gamma_k$ is a decreasing loop; these three possibilities are illustrated in Figure~\ref{Fig:Decrease}.  If $\beta_{k-1}$ is a decreasing bridge, then $\ddiv (\varphi) + D$ contains $v_k$, and the restriction of $\varphi$ to the loop $\gamma_k$ is then uniquely determined.  Similarly, if $\beta_k$ is a decreasing bridge, then $\ddiv (\varphi) + D$ contains $w_k$, and the restriction of $\varphi$ to the loop $\gamma_k$ is again uniquely determined.  If, however, $\gamma_k$ is a decreasing loop, then there may be  infinitely many possibilities for the restriction of $\varphi$ to the loop $\gamma_k$.  In this case, it is possible for $D+\ddiv(\varphi)$ to contain a smooth cut set.

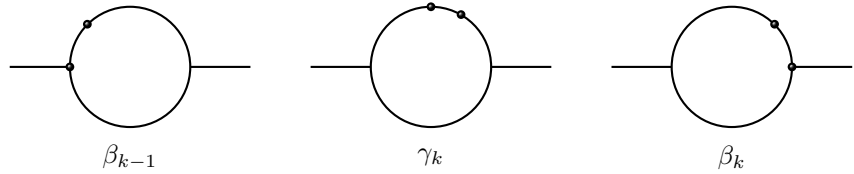
\begin{figure}[H]
\begin{tikzpicture}[thick, scale=0.8]

\begin{scope}[grow=right, baseline]
\draw (-1,0) circle (1);
\draw (-3,0)--(-2,0);
\draw (0,0)--(1,0);
\draw [ball color=black] (-1.71,0.71) circle (0.55mm);
\draw [ball color=black] (-2,0) circle (0.55mm);
\draw (-1,-1.5) node {$\beta_{k-1}$};

\draw (4,0) circle (1);
\draw (2,0)--(3,0);
\draw (5,0)--(6,0);
\draw [ball color=black] (4,1) circle (0.55mm);
\draw [ball color=black] (4.5,0.87) circle (0.55mm);
\draw (4,-1.5) node {$\gamma_k$};

\draw (9,0) circle (1);
\draw (7,0)--(8,0);
\draw (10,0)--(11,0);
\draw [ball color=black] (10,0) circle (0.55mm);
\draw [ball color=black] (9.71,0.71) circle (0.55mm);
\draw (9,-1.5) node {$\beta_k$};

\end{scope}
\end{tikzpicture}

\caption{Three different building blocks with $s'_k = s_k -1$, labeled by the bridge or loop that is decreasing.}
\label{Fig:Decrease}
\end{figure}
\end{example}

\noindent In our constructions, we will always start from a finite set of building blocks, chosen so that no two have both the same slope index sequences and also the same slopes along bridges.

\subsection{Equivalence on loops}

The main difference between our construction in this section and that of Section~\ref{Sec:VertexAvoiding} is that now we may assign more than one function to a given loop, as long as all of the functions that we assign agree on $\gamma_k$, in the following sense.

\begin{definition}
We say that functions \emph{agree} on a subgraph of $\Gamma$ if their restrictions to that subgraph differ by an additive constant.
\end{definition}

We will most often consider agreement on one loop at a time, but in a few key places, such as Definition~\ref{def:properties}, we also consider agreement on larger subgraphs.  We make a few preliminary observations about sufficient conditions for two functions to agree on a loop, starting with the following analogue of Lemmas~\ref{lem:onenew} and~\ref{lem:onedeparts}.

\begin{lemma}
\label{Lem:OneNewFunction}
If $\gamma_k$ is not the first loop in a block, then any two new permissible functions agree on $\gamma_k$.  Similarly, if $\gamma_k$ is not the last loop in a block, then any two departing permissible functions agree on $\gamma_k$.
\end{lemma}

\begin{proof}
Recall that permissible functions have constant slope on bridges.  If $\psi \in R(2D)$ is a new permissible function and $\gamma_k$ is not the first loop in a block, then $s_{k+1} (\psi) \geq s_{k+1} (\theta) > s_{k}(\psi)$.  If we write $s=s_{k+1} (\theta)$, then the restriction $\psi \vert_{\gamma_k}$ is in the linear series $R(E)$, where
\[
E = 2D \vert_{\gamma_k} + (s-1) v_k - s w_k .
\]
The divisor $E$ has degree 1.  On a loop, every divisor of degree 1 is equivalent to a unique effective divisor, so this determines $\psi \vert_{\gamma_k}$ up to an additive constant.

If $\psi$ is departing and $\gamma_k$ is not the last loop in a block, then $s_{k+1} (\psi) > s_{k+1} (\theta) \geq s_{k}(\psi)$, and the rest of the proof is similar.
\end{proof}

We note the following corollary.

\begin{corollary} \label{cor:equivtonew}
Let $\varphi$ and $\varphi'$ be building blocks.  Assume that $\varphi + \varphi'$ is permissible on $\gamma_k$, which is not the first loop in a block, and that $s_k (\varphi) > s_k [\tau_k (\varphi)]$.  Then there are new permissible functions on $\gamma_k$, and they all agree with $\varphi+\varphi'$.
\end{corollary}

\begin{proof}
Since $s_k (\varphi)$ is strictly greater than $s_k [\tau_k (\varphi)]$, the degree of $D+\ddiv(\varphi)$ at $v_k$ is strictly positive.
Consider a function $\phi \in \PL(\Gamma)$ that has the same slope index sequence as $\varphi$, and the same restriction to every loop, but with
\[
s_k (\phi) = s_k (\varphi) -1.
\]
By construction, $\deg_{v_k} (D + \ddiv(\phi)) = \deg_{v_k} (D + \ddiv(\varphi)) -1 \geq 0$.  It follows that $\phi \in R(D)$ is a building block.  Since
\[
s_k (\phi + \varphi') < s_k (\varphi + \varphi') \leq s_k (\theta),
\]
we see that $\phi + \varphi'$ is new on $\gamma_k$.  It follows that $\varphi + \varphi'$ agrees with a new permissible function on $\gamma_k$.  By Lemma~\ref{Lem:OneNewFunction}, $\varphi + \varphi'$ agrees with every new permissible function on $\gamma_k$.
\end{proof}

With this in mind, we make the following definition.

\begin{definition}
We say that a sum of two building blocks $\varphi + \varphi'$ is \emph{shiny} on $\gamma_k$ if it is permissible on $\gamma_k$, and
\[
s_k [\tau_k (\varphi)] + s_k [\tau_k (\varphi')] < s_k (\theta).
\]
\end{definition}

By Corollary~\ref{cor:equivtonew}, if $\gamma_k$ is not the first loop in a block, then any new function on $\gamma_k$ is shiny, as the terminology suggests.
However, a shiny function is not necessarily new.  Moreover, if $\gamma_k$ is the first loop in a block, then there may be new functions that do not agree with each other on $\gamma_k$, but all shiny functions do agree.   The following proposition examines the structure of shiny functions in a little more detail.

\begin{proposition}
\label{Prop:Shiny}
If $\psi = \varphi + \varphi'$ is shiny on $\gamma_k$, then the restriction of either $D+\ddiv (\varphi)$ or $D+\ddiv (\varphi')$ to $\gamma_k \smallsetminus \{ v_k \}$ has degree 0.  Moreover, either
\[
s_{k+1} (\varphi) > s_k [\tau_k (\varphi)] \mbox{ or } s_{k+1} (\varphi') > s_k [\tau_k (\varphi')].
\]
\end{proposition}

\begin{proof}
If $s_k (\varphi) = s_k [\tau_k (\varphi)]$ and $s_k (\varphi') = s_k [\tau_k (\varphi')]$, then as in the proof of Lemma~\ref{Lem:OneNewFunction}, we see that the restriction of $2D+\ddiv(\psi)$ to $\gamma_k$ has degree at most 1.  Otherwise, by Corollary~\ref{cor:equivtonew}, $\psi$ agrees with a new permissible function, and $2D+\ddiv(\psi)$ contains $v_k$.  Since $\psi$ is permissible on $\gamma_k$, the restriction of $2D+\ddiv(\psi)$ to $\gamma_k$ has degree at most 2, hence the restriction of $2D+\ddiv(\psi)$ to $\gamma_k \smallsetminus \{ v_k \}$ has degree at most 1.  It follows that the restriction of either $D+\ddiv(\varphi)$ or $D+\ddiv(\varphi')$ to $\gamma_k \smallsetminus \{ v_k \}$ must therefore have degree 0.  Without loss of generality we may assume that the restriction of $D + \ddiv(\varphi)$ to $\gamma_k \smallsetminus \{ v_k \}$ has degree 0.  Since $\deg_{w_k} (D+\ddiv(\varphi)) = 0$, we see that $s_{k+1} (\varphi) = s'_k [\tau'_k (\varphi)]$.

If $D+\ddiv(\varphi)$ does not contain $v_k$, then
\[
s_{k+1} (\varphi) > s_k (\varphi) = s_k [\tau_k (\varphi)].
\]
On the other hand, if $D+\ddiv(\varphi)$ contains $v_k$, then
\[
s_{k+1} (\varphi) = s_k (\varphi) > s_k [\tau_k (\varphi)].
\]
In either case, we see that $s_{k+1} (\varphi) > s_k [\tau_k (\varphi)]$, as required.
\end{proof}

\begin{remark}
\label{Rem:Shiny}
In the argument above, we assume that the restriction of $D + \ddiv(\varphi)$ to $\gamma_k \smallsetminus \{ v_k \}$ has degree 0, and show that $ s_{k+1}(\varphi) > s_k[\tau_k (\varphi)]$.  The converse is also true; if a building block $\varphi$ satisfies $s_{k+1} (\varphi) > s_k [\tau_k (\varphi)]$, then the restriction of $D+\ddiv(\varphi)$ to $\gamma_k \smallsetminus \{ v_k \}$ has degree 0.
\end{remark}

\begin{lemma}
\label{Lem:ShinyEq}
Let $\psi = \varphi + \varphi'$ and $\psi' = \phi + \phi'$ be shiny functions on $\gamma_k$.  Then, after possibly reordering $\phi$ and $\phi'$, we have
\[
\tau_k (\varphi) = \tau_k (\phi) \mbox{ and } \tau_k (\varphi') = \tau_k (\phi').
\]
\end{lemma}

\begin{proof}
By Proposition~\ref{Prop:Shiny}, we may assume after possibly reordering that the restrictions of both $D+\ddiv(\varphi)$ and $D+\ddiv(\phi)$ to $\gamma_k \smallsetminus \{ v_k \}$ have degree zero.  It follows that $\varphi$ and $\phi$ agree on $\gamma_k$.  By Lemma~\ref{Lem:OneNewFunction}, $\varphi + \varphi'$ agrees with $\phi + \phi'$ on $\gamma_k$, hence $\varphi'$ agrees with to $\phi'$ on $\gamma_k$ as well.  The equality of slope indices follows from Remark~\ref{Rem:Equivalent}.
\end{proof}

In Section~\ref{sec:master} below, we state Theorem~\ref{Thm:ExtremalConfig}, which gives the essential properties of the master template, constructed as a tropical linear combination of a collection $\cB$ of pairwise sums of building blocks.  We should stress that the hypotheses of this theorem are as important as the conclusions; we need several technical conditions on the collection $\cB$ of pairwise sums of building blocks in order to successfully run the algorithm to construct the master template that appears in the next section.  In Section~\ref{Sec:Generic}, we will consider several cases depending on the properties of $\Sigma$, and in each case we will choose a set $\cB$ and show that it satisfies these properties.

We begin with a technical property on the collection of building blocks to be used as summands.

\begin{definition} \label{def:property*}
Let $\cA$ be a subset of the building blocks.  We say that $\cA$ satisfies property $(\ast)$ if any two functions $\varphi , \varphi' \in \cA$ with $\tau'_k (\varphi) = \tau'_k (\varphi')$ agree on $\gamma_k$, and no two functions in $\cA$ differ by a constant.
\end{definition}

\noindent Note that there are only finitely many possibilities for the slope of a building block on each bridge.  It follows that any collection of building blocks that satisfies $(\ast)$ is necessarily finite.

Before stating the other technical properties, we note that this definition has the following important consequences.

\begin{lemma}
\label{Lem:Rho2}
Let $\cA$ be a subset of the building blocks satisfying property $(\ast)$.  Let $\varphi, \varphi' \in \cA$, and suppose that $s_{k+1} (\varphi) = s_{k+1} (\varphi')$ and $\mu (\beta_k) \leq 1$.  Then $\varphi$ and $\varphi'$ agree on $\gamma_k$.
\end{lemma}

\begin{proof}
If $\varphi$ and $\varphi'$ agree on $\gamma_k$, then $\tau'_k(\varphi) \neq \tau'_k(\varphi')$.  Therefore, there must be $j \neq j'$ such that
\[
s'_k [j] \geq s_{k+1} (\varphi) \geq s_{k+1} [j] \mbox{ and } s'_k [j'] \geq s_{k+1} (\varphi') \geq s_{k+1} [j'].
\]
Assume without loss of generality that $j'>j$.  Then
\[
s'_k [j'] > s'_k [j] \geq s_{k+1} [j'] > s_{k +1} [j],
\]
which implies that the bridge $\beta_k$ has multiplicity at least 2.
\end{proof}

\begin{lemma}
\label{Lem:SameSlope}
Let $\cA$ be a subset of the building blocks satisfying property $(\ast)$, and let $\varphi, \varphi' \in \cA$ satisfy
\[
\tau_k(\varphi) = \tau_k(\varphi') \mbox{ and } \tau_{k+1}(\varphi) = \tau_{k+1}(\varphi').
\]
Then
\[
\tau'_k(\varphi) = \tau'_k(\varphi')
\]
and $\varphi$ agrees with $\varphi'$ on $\gamma_k$.
\end{lemma}

\begin{proof}
Because a switching loop has multiplicity at least 1 and a switching bridge has multiplicity at least 2, we cannot have both that $\gamma_k$ is a switching loop and $\beta_k$ is a switching bridge.  It follows that either
\[
\tau'_k(\varphi) = \tau_k(\varphi) \mbox{ and } \tau'_k(\varphi') = \tau_k(\varphi')
\]
or
\[
\tau'_k(\varphi) = \tau_{k+1}(\varphi) \mbox{ and } \tau'_k(\varphi') = \tau_{k+1}(\varphi').
\]
Thus,
\[
\tau'_k(\varphi) = \tau'_k(\varphi').
\]
By property $(\ast)$, it follows that $\varphi$ agrees with $\varphi'$ on $\gamma_k$.
\end{proof}

\subsection{Properties $(\ast\ast)$ and $(\dagger\dagger)$}

We now introduce two key technical properties for collections of pairwise sums of building blocks.  These conditions will be essential for our construction of the master template.

\begin{definition}  \label{def:properties}
Let $\cA$ be a subset of the building blocks, and let
\[
\cB \subset \{ \varphi + \varphi' \, \vert \, \varphi, \varphi' \in \cA \}.
\]
We consider the following two properties:
\begin{itemize}
\item[($\ast\ast$)]  Whenever there is $\varphi + \varphi' \in \cB$ such that $2D+\ddiv(\varphi + \varphi')$ contains $w_k$, and either  $\gamma_k$ switches slope $\tau_k (\varphi)$ or $s_{k+1} (\varphi) < s'_k [\tau'_k (\varphi)]$, then there is some $\psi \in \cB$ that agrees with $\varphi + \varphi'$ on $\Gamma_{[1,k]}$ such that $s_{k+1} (\psi) > s_{k+1} (\varphi + \varphi')$.
\item[($\dagger\dagger$)]  Whenever the are permissible functions in $\cB$ that agree on $\gamma_k$ with different slopes on $\beta_k$, and either $\gamma_k$ is a switching loop or $\beta_k$ is a switching bridge, then no permissible function is shiny on $\gamma_k$.
\end{itemize}
\end{definition}
Theorem~\ref{Thm:ExtremalConfig}, below, says that we can construct a master template starting from any set of building blocks $\cA$ that satisfies ($\ast$), and any collection $\cB$ of pairwise sums of functions in $\cA$ that satisfies $(\ast\ast)$ and $(\dagger\dagger)$.  In Section~\ref{Sec:Generic}, we will choose such $\cA$ and $\cB$ on a case-by-case basis, according to the properties of $\Sigma$.

\subsection{The master template}
\label{sec:master}

The master template $\theta$ is a tropical linear combination of pairwise sums of building blocks.  As in the vertex avoiding case, we will construct $\theta$ so that it has nearly constant slope on every bridge, with perhaps slope 1 or 2 higher for a short distance at the beginning of the bridge, and we find it useful to specify the nearly constant slopes $s_k(\theta)$ in advance.  These slopes are most easily described in terms of partitions associated to the slope vectors $s_k(\Sigma)$, as follows.

We associate to $\Sigma$ a sequence of partitions
\[
\lambda'_0, \lambda_1, \lambda'_1 \ldots, \lambda'_g, \lambda_{g+1},
\]
each with at most $r + 1$ columns, numbered from 0 to $r$.  The $(r-i)$th column of $\lambda_k$ contains $(g-d+r)+ s_k[i]-i$ boxes.  Similarly, the $(r-i)$th column of $\lambda'_k$ contains $(g-d+r) + s'_k[i]-i$ boxes.  Note that $\lambda_k$ is a subset of $\lambda'_{k-1}$, and $\lambda'_k$ contains at most one box that is not contained in $\lambda_k$.  Moreover, $\lambda_{g+1}$ contains the $(r+1) \times (g-d+r)$ rectangle.

We then choose $z$ to be the largest integer such that $\lambda'_z$ contains exactly 6 boxes in the union of the first two rows, and $\lambda_z$ does not.  Similarly, we let $z'$ be the largest integer such that $\lambda'_{z'+2}$ contains exactly 10 boxes in the union of the second and third row, and $\lambda_{z'+2}$ does not.  Note that, since each partition in the sequence contains at most 1 box not contained in the previous partition, such $z$ and $z'$ exist.

With this new definition of $z$ and $z'$, we define the slopes of $\theta$ as in Section~\ref{sec:slopes}:

\begin{displaymath}
s_k (\theta) = \left\{ \begin{array}{ll}
4 & \textrm{if $k \leq z$} \\
3 & \textrm{if $z < k \leq z'$}\\
2 & \textrm{if $z' < k \leq g$.}
\end{array} \right.
\end{displaymath}

\begin{theorem}
\label{Thm:ExtremalConfig}
Let $\cA$ be a subset of the building blocks satisfying property $(\ast)$ and let
\[
\cB \subset \{ \varphi + \varphi' \, \vert \, \varphi, \varphi' \in \cA \}
\]
be a subset satisfying property $(\ast\ast)$ and property $(\dagger\dagger)$.  Then there is a tropical linear combination $\theta$ of the functions in $\cB$ with $s_k(\theta)$ as specified above, such that
\begin{enumerate}
\item  each function $\psi \in \cB$ is assigned to some loop $\gamma_k$ or bridge $\beta_k$ and achieves the minimum at some point $v$ on the loop or bridge to which it is assigned,
\item  any other function $\psi'$ that achieves the minimum at $v$ agrees with $\psi$ on the loop $\gamma_k$.
\end{enumerate}
\end{theorem}
\noindent Note that, in the case where $\psi$ and $\psi'$ are assigned to the bridge $\beta_k$, the second condition says that they agree on the preceding loop $\gamma_k$.

\subsection{Algorithm for constructing the master template}
\label{Sec:Alg}

Throughout this section, and for the remainder of the paper, we assume that the hypotheses of Theorem~\ref{Thm:ExtremalConfig} are satisfied.  In particular, we let $\cA$ be a subset of the building blocks satisfying property $(\ast)$, and $\cB$ a subset of pairwise sums of functions in $\cA$ satisfying properties $(\ast\ast)$ and $(\dagger\dagger)$.

We now sketch the overall procedure that we will use to build the master template
\[
\theta = \min \{ \psi + c_{\psi} \, \vert \, \psi \in \cB \}
\]
with the slopes $s_k(\theta)$ specified above.  The algorithm is in many ways similar to that presented in Section~\ref{Sec:VertexAvoiding}, and we will highlight the differences when they appear.  As in Section~\ref{Sec:VertexAvoiding}, we move from left to right across each of the three blocks where $s_k(\theta)$ is constant, assigning functions to loops, and adjusting the coefficients so that each function achieves the minimum on the loop to which it is assigned.  At the end of each block, we start the next block by choosing coefficients such that $\theta$ bends at the midpoint of the bridge between blocks.

In the special case where the divisor class is vertex avoiding, the building blocks are the functions $\varphi_i$, the set $\cB$ of all pairwise sums $\varphi_i + \varphi_j$ satisfies properties $(\ast \ast)$ and $(\dagger\dagger)$ vacuously, and the template we construct is precisely the maximally independent combination constructed in Section~\ref{Sec:VertexAvoiding}.

In the general case, our algorithm for constructing the master template is as follows.

\medskip

\noindent \textbf{Start at the First Bridge.}
Since $\rho \leq 2$, the ramification weight at $w_0$ is at most 2, so the slope vector $s'_0(\Sigma)$ is one of the following:
\[
(-3,-2, -1, 0, 1, 2, 3),
\]
\[
(-4,-2, -1, 0, 1, 2, 3),
\]
\[
(-5,-2, -1, 0, 1, 2, 3),
\]
\[
(-4,-3, -1, 0, 1, 2, 3).
\]
In particular, $s'_0 [6] = 3$ and $s'_0 [5] = 2$.  For every $\psi \in \cB$ with $s'_0 (\psi) = 2s'_0 [6] = 6$, initialize the coefficient of $\psi$ to be 0.  For every $\psi' \in \cB$ satisfying $s'_0 (\psi') = s'_0 [6] + s'_0 [5] = 5$, initialize the coefficient of $\psi'$ so that it agrees with $\psi$ at the midpoint of the first bridge $\beta_0$.  There are several new permissible functions on the first loop.  Initialize the coefficient of each new permissible function so that it equals $\theta$ halfway between the midpoint and the rightmost endpoint of the bridge $\beta_0$, and then apply the loop subroutine.  Initialize all other coefficients to $\infty$.  Proceed to the first loop.

\medskip

\noindent \textbf{Loop Subroutine.}
Each time we arrive at a loop $\gamma_k$, apply the following steps.

\medskip

\noindent \textbf{Loop Subroutine, Step 1:  No unassigned permissible functions.}
If there are no unassigned permissible functions on $\gamma_k$, skip this loop and proceed to the next loop.

\medskip

\noindent \textbf{Loop Subroutine, Step 2:  All unassigned permissible functions are new and in the same equivalence class.}
If every unassigned permissible function on $\gamma_k$ is new and in the same equivalence class, assign all the permissible functions in this equivalence class.  Note that, by Lemma~\ref{Lem:OneNewFunction}, if $\gamma_k$ is not the first loop in a block, then all new functions are automatically in the same equivalence class.  Set their coefficients so that they agree with the minimum of the other functions at $v_k$.  Note that the sum of the slopes of any new function along the non-bridge edges adjacent to $v_k$ is $s_k (\theta)-1$, smaller than the corresponding sum of slopes for a non-new function.  It follows that the new functions are the only functions to achieve the minimum on a subinterval of one of these two edges, a short distance to the right of $v_k$.  Proceed to the next loop.

\medskip

\noindent \textbf{Loop Subroutine, Step 3:  Re-initialize unassigned coefficients.}
Otherwise, there is at least one unassigned permissible function $\psi$ on $\gamma_k$ such that $c_{\psi}$ is finite.  Find the unassigned, permissible function $\psi \in \cB$ that maximizes $\psi (w_k) + c_{\psi}$, among finite values of $c_{\psi}$.   Initialize the coefficients of the new permissible functions (if any) and adjust the coefficients of the other unassigned permissible functions upward so that they all agree at $w_k$.  (The unassigned permissible functions are strictly less than all other functions, at every point in $\gamma_k$, even after this upward adjustment.)

\medskip

\noindent \textbf{Loop Subroutine, Step 4:  Assign departing functions.}
If there is a departing function, assign it to the loop.  Note that any two departing functions agree on this loop, by Lemma~\ref{Lem:OneNewFunction}. Adjust the coefficients of the non-departing unassigned permissible functions upward so that they all agree with the departing function of smallest slope at a point on the following bridge a short distance to the right of $w_k$, but far enough so that the departing functions are the only functions to achieve the minimum at any point of the loop.  This is possible because the building blocks have constant slopes along the bridges, and the bridges are much longer than the loops.  Note that the departing functions achieve the minimum on the whole loop, no other functions achieve the minimum on this loop, and any two departing functions agree on the loop.  Proceed to the next loop.

\medskip

\noindent \textbf{Loop Subroutine, Step 5: Skip skippable loops.}
In the vertex avoiding case, the loops with positive multiplicity are precisely the lingering loops, and they have the property that $s_{k}[i] = s_{k+1}[i]$ for all $i$.  These are the loops that we skipped in the algorithm, without assigning a function.

In the general case, the loops with unassigned permissible functions that we skip are characterized as follows.

\begin{definition}
\label{Def:Skippable}
We will say that the loop $\gamma_k$ is \emph{skippable} if not all unassigned permissible functions are new and agree with each other, there are no unassigned departing permissible functions on $\gamma_k$, and there is an unassigned permissible function $\psi = \varphi + \varphi' \in \cB$ satsifying one of the following:
\begin{enumerate}
\item  $2D + \ddiv (\psi)$ contains a point whose shortest distance to $w_k$ is a non-integer multiple of $m_k$, or
\item  $2D + \ddiv (\psi)$ contains $w_k$, or
\item  $D + \ddiv (\varphi')$ contains two points of $\gamma_k \smallsetminus \{ v_k \}$.
\end{enumerate}
\end{definition}
\noindent In the vertex avoiding case, conditions (ii) and (iii) of Definition~\ref{Def:Skippable} are never satisfied, and condition (i) is satisfied precisely on the lingering loops.  In general, condition (i) can be satisfied even on non-lingering loops.  If $\gamma_k$ is skippable, then either $\gamma_k$ or $\beta_k$ has positive multiplicity.  Note also that whether a loop is skippable depends on which functions have been previously assigned.  In particular, if there is an unassigned departing function, then the loop is not skippable.

If $\gamma_k$ is skippable, then do not assign any functions.  Proceed to the next loop.

\medskip

\noindent \textbf{Loop Subroutine, Step 6: Otherwise, use Proposition~\ref{Prop:NewThreeShape}.}  In the remaining cases, when there are unassigned permissible functions, but not all are new, none are departing, and the loop is not skippable, we assign an entire equivalence class of permissible functions, chosen using the following lemma and proposition, which are close analogues of Lemma~\ref{Lem:VAAtMostThree} and Proposition~\ref{prop:threeshape}.

\begin{lemma}
\label{Lem:AtMostThree}
If there is no unassigned departing permissible function on $\gamma_k$, then the number of equivalence classes of permissible functions on $\gamma_k$ is at most 3.
\end{lemma}

\begin{proof}
Consider the set of building blocks $\varphi \in \cA$ such that there exists $\varphi' \in \cA$ with $\varphi + \varphi' \in \cB$ an unassigned permissible function on $\gamma_k$.  If every such building block satisfies $s_{k+1} (\varphi) = s'_k [\tau'_k(\varphi)]$, then the proof of Lemma~\ref{Lem:VAAtMostThree} essentially goes through.  Specifically, any two functions in $\cA$ with the same $s_{k+1}$ agree on $\gamma_k$ by property $(\ast)$.  Thus, the number of equivalence classes of permissible functions on $\gamma_k$ is bounded above by the number of pairs $(i,j)$ such that $s'_k[i]+s'_k[j] = s_k(\theta)$.  This number of pairs is at most 3, exactly as in the proof of Lemma~\ref{Lem:VAAtMostThree}.

On the other hand, suppose that there is an unassigned permissible function $\varphi + \varphi' \in \cB$ such that $s_{k+1} (\varphi) < s'_k [\tau'_k(\varphi)]$.  By property $(\ast\ast)$, there is a function $\psi \in \cB$ that agrees with $\varphi + \varphi'$ on $\Gamma_{[1,k]}$, with the property that
\[
s_{k+1} (\psi) > s_{k+1} (\varphi + \varphi') \geq s'_k (\theta).
\]
Since $\psi$ is a departing function, it must have been assigned to a previous loop.  But $\varphi + \varphi'$ agrees with $\psi$ on this previous loop, so $\varphi + \varphi'$ must be assigned to this previous loop as well.  This contradicts our assumption that $\varphi + \varphi'$ was unassigned.
\end{proof}

\begin{proposition}
\label{Prop:NewThreeShape}
Consider a set of at most three equivalence classes of functions in $\cB$ on a non-skippable loop $\gamma_k$.  If all of the functions take the same value at $w_k$, then there is a point of $\gamma_k$ at which one of these equivalence classes is strictly less than the others.
\end{proposition}

\begin{proof}
The proof of Proposition~\ref{prop:threeshape} depends only on the restrictions of the functions to the loop $\gamma_k$, and the fact that, if $\psi$ is one of these functions, then $2D+\ddiv(\psi)$ does not contain $w_k$.  This latter fact is guaranteed by our assumption that $\gamma_k$ is not skippable.  The conclusion therefore continues to hold if we replace the functions with equivalence classes of functions.
\end{proof}

Combining Lemma~\ref{Lem:AtMostThree} and Proposition~\ref{Prop:NewThreeShape}, we see that there is an equivalence class of unassigned permissible functions that achieves the minimum uniquely at some point of $\gamma_k$.  We assign all of the permissible functions in this equivalence class to the loop, and adjust their coefficients upward slightly---enough so that they will never achieve the minimum on any loops to the right, but not so much that they do not achieve the minimum uniquely on this loop.  (Specifically, we increase the coefficient of these functions by $\frac{1}{3}m_k$, as in Lemma~\ref{Lem:VANotToRight}.)  Proceed to the next loop.

\medskip

\noindent \textbf{Proceeding to the Next Loop.}
If $\gamma_k$ is not the last in its block, then apply the loop subroutine on $\gamma_{k+1}$.  Otherwise, apply the following subroutine for proceeding to the next block.

\medskip

\noindent \textbf{Proceeding to the Next Block.}
After applying the loop subroutine to the last loop in a block, we will see that there is at most one equivalence class of unassigned permissible function, and these functions already achieve the minimum on the outgoing bridge, without any further adjustments of the coefficients.

If the current block is not the last one, then proceed to the first loop of the next block.  There are several new permissible functions.  Initialize the coefficient of each new permissible function so that it is equal to $\theta$ at the midpoint of the bridge between the blocks, and then apply the loop subroutine.  Otherwise, we are at the last loop $\gamma_g$, and proceed to the last bridge.

\medskip

\noindent \textbf{The Last Bridge.}
Just as we enumerated the possible ramification sequences at $w_0$, we may also enumerate the possible ramification sequences at $v_{g+1}$.  By an argument symmetric to that applied to the first bridge, we see that $s_{g+1} [1] = 1$ and $s_{g+1} [0] = 0$.  For every $\psi \in \cB$ with $s_{g+1} (\psi) = s_{g+1} [0] + s_{g+1} [1] = 1$, initialize the coefficient of $\psi$ so that it equals $\theta$ at the midpoint of the last bridge.  For every $\psi' \in \cB$ satisfying $s_{g+1} (\psi') = 2s_{g+1} [0] = 0$, initialize the coefficient of $\psi'$ so that it equals $\theta$ halfway between the midpoint and the rightmost endpoint of the last bridge.  Output the coefficients $\{ c_{\psi} \}$.

Note that all functions assigned to a given loop agree on that loop.  It is possible, however, for a function $\psi \in \cB$ to be assigned to the loop $\gamma_k$ while another function $\psi'$ that agrees with it on $\gamma_k$ is not.  This is the case, for example, if $\psi$ is departing, but $\psi'$ is not.  It is also the case if $\psi$ is permissible on $\gamma_k$, but $\psi'$ is not.

\subsection{Verifying the master template}

In this section, we prove Theorem~\ref{Thm:ExtremalConfig}, by verifying that the master template constructed via the algorithm presented in the previous section has the claimed properties.  We assume the hypotheses of the theorem; in particular, $\cA$ is a set of building blocks that satisfies property $(\ast)$ and $\cB$ is a collection of pairwise sums of functions in $\cA$ that satisfies $(\ast\ast)$ and $(\dagger\dagger)$.  We begin by checking that every function that is assigned to a loop or bridge achieves the minimum at some point of that loop or bridge.

\begin{lemma}
\label{Lem:NotToRight}
Suppose that $\psi$ is assigned to the loop $\gamma_k$ or the bridge $\beta_k$, and let $k' > k$ be the smallest value such that there is a function assigned to $\gamma_{k'}$.  Then $\psi$ does not achieve the minimum at any point to the right of $v_{k'}$.
\end{lemma}

\begin{proof}
This follows by the same argument as Lemma~\ref{Lem:VANotToRight}, using the definition of permissibilty and the fact that the building blocks have constant slopes along bridges.
\end{proof}

We now show that, on each non-skippable loop, there is a point where the function assigned to that loop achieves the minimum, and the other functions that achieve the minimum agree on the loop.  The analogous statement about functions assigned to bridges will follow by a counting argument similar to that of Section~\ref{Sec:VertexAvoiding}, which we will see in the proof of Theorem~\ref{Thm:ExtremalConfig}.

\begin{lemma}
\label{Lem:Config}
Suppose that $\psi$ is assigned to the loop $\gamma_k$.  Then there is a point $v \in \gamma_k$ where $\psi$ achieves the minimum.  Moreover, any other function $\psi'$ that achieves the minimum at $v$ agrees with $\psi$ on $\gamma_k$.
\end{lemma}

\begin{proof}
This follows by the same argument as Lemma~\ref{Lem:VAConfig}.
\end{proof}

The remainder of this section is devoted to showing that every function in $\cB$ is assigned to some loop or bridge.  Ultimately, this is a counting argument similar to that of Section~\ref{sec:counting}, but the details are more subtle in the general case.  We proceed via a sequence of lemmas and propositions.  First, let $b$ be the largest integer such that $\lambda'_b$ contains exactly 7 boxes in its first two rows, and $\lambda_b$ does not.  Similarly, let $b'$ be the largest integer such that $\lambda'_{b'}$ contains exactly 8 boxes in the union of its first and third row, and $\lambda_{b'}$ does not.  Note that, in the case where $D$ is vertex avoiding, $b$ and $b'$ are as defined in Section~\ref{Sec:VertexAvoiding}.

The following two propositions are analogues of Lemmas~\ref{lem:onenew} and~\ref{lem:nonew}, respectively, with shiny functions in place of new permissible functions.

\begin{proposition}
\label{Prop:Lingering}
If $\gamma_k$ is skippable and not the first loop in a block, then no permissible function is shiny on $\gamma_k$.
\end{proposition}

\begin{proposition}
\label{Prop:GammaA}
The loops $\gamma_z, \gamma_b, \gamma_{b'},$ and $\gamma_{z'+2}$ are all non-skippable, and no permissible function is shiny on any of them.
\end{proposition}

The proofs of these propositions rely heavily on property $(\ast\ast)$, and use the following two technical lemmas about permissible functions on skippable loops.

\begin{lemma}
\label{Lem:Skippable1}
Let $\gamma_k$ be a skippable loop and let $\psi = \varphi + \varphi' \in \cB$ be an unassigned permissible function on $\gamma_k$.  Then
\begin{enumerate}
\item  $s_{k+1} (\psi) = s'_k (\theta)$,
\item  $s_{k+1} (\varphi) = s'_k [\tau'_k (\varphi)]$, and
\item  $s_{k+1} (\varphi') = s'_k [\tau'_k (\varphi')]$.
\end{enumerate}
\end{lemma}

\begin{proof}
By definition, no unassigned permissible function is departing on a skippable loop.  Therefore,
\[
s_{k+1} (\psi) = s'_k (\theta).
\]
It remains to show that $s_{k+1}(\varphi) = s'_k[\tau'_k(\varphi)]$.  Suppose not.  Then
\[
s_{k+1} (\varphi) < s'_k [\tau'_k (\varphi)],
\]
and, by property $(\ast\ast)$, there is a function $\psi' \in \cB$ that agrees with $\psi$ on $\Gamma_{[1,k]}$, with the property that
\[
s_{k+1} (\psi') > s_{k+1} (\psi) = s'_k (\theta).
\]
We claim that this is impossible.  Indeed, if $\psi'$ is unassigned on $\gamma_k$, then it would be a departing function, contradicting the hypothesis that $\gamma_k$ is skippable.  On the other hand, if $\psi'$ is assigned to a previous loop then $\psi$ would have been assigned to that loop as well.

We conclude that $s_{k+1} (\varphi) = s'_k [\tau'_k (\varphi)]$, and, similarly,  $s_{k+1} (\varphi') = s'_k [\tau'_k (\varphi')]$, as required.
\end{proof}

\begin{lemma}
\label{Lem:Skippable2}
Let $\gamma_k$ be a skippable loop, and suppose there is a building block $\phi$ such that $s_{k+1} (\phi) > s_k [\tau_k (\phi)]$.  Then there is a permissible function $\psi = \varphi + \varphi' \in \cB$ such that
\begin{enumerate}
\item  $s_{k+1} (\varphi) = s_k (\varphi) + 1$,
\item  $s_{k+1} (\varphi') = s_k (\varphi') - 1$, and
\item $D + \ddiv(\varphi')$ contains either $w_k$, a point of $\gamma_k$ whose distance from $w_k$ is not an integer multiple of $m_k$, or two points of $\gamma_k \smallsetminus \{ v_k \}$.
\end{enumerate}
\end{lemma}

\begin{proof}
Since $\gamma_k$ is skippable, there is an unassigned permissible function $\psi = \varphi + \varphi' \in \cB$ such that $2D+\ddiv(\psi)$ contains either $w_k$, a point whose distance from $w_k$ is a non-integer multiple of $m_k$, or two points of $\gamma_k \smallsetminus \{ v_k \}$.

We first consider the case where one of the two functions $\varphi, \varphi'$ has smaller slope on $\beta_k$ than on $\beta_{k-1}$.  Suppose without loss of generality that
\[
s_{k+1} (\varphi') < s_k (\varphi').
\]
Since $\psi$ is permissible, we must have $s_{k+1} (\psi) \geq s_k (\psi)$.  It follows that
\[
s_{k+1} (\varphi) > s_k (\varphi).
\]
Since the slope of a building block can increase by at most 1 from one bridge to the next, we therefore see that
\[
s_{k+1} (\varphi) = s_k (\varphi) + 1 \mbox{ and } s_{k+1} (\varphi') = s_k (\varphi') -1.
\]
It follows that the restriction of $D + \ddiv(\varphi)$ to $\gamma_k$ is zero, and $D + \ddiv(\varphi')$ contains either $w_k$, a point of $\gamma_k$ whose distance from $w_k$ is not an integer multiple of $m_k$, or two points of $\gamma_k \smallsetminus \{ v_k \}$.

To complete the proof, we will rule out the possibility that neither function $\varphi, \varphi'$ has smaller slope on $\beta_k$ than on $\beta_{k-1}$.  Assume that
\[
s_{k+1} (\varphi) \geq s_k (\varphi) \mbox{ and } s_{k+1} (\varphi') \geq s_k (\varphi').
\]
Note that this immediately rules out the possibility that $D + \ddiv(\varphi')$ contains more than one point of $\gamma_k$.  We will reach a contradiction by showing that $\gamma_k$ is a switching loop and then applying property $(\ast\ast)$.  As a first step in this direction, we claim that $2D+ \ddiv \psi$ contains $w_k$.  Since $s_{k+1} (\phi) > s_k [\tau_k (\phi)]$, the restriction of $D+\ddiv(\phi)$ to $\gamma_k \smallsetminus \{ v_k \}$ has degree 0 (see, e.g., Remark~\ref{Rem:Shiny}).  It follows that the shortest distance from the point of $D$ on $\gamma_k$ to $w_k$ is an integer multiple of $m_k$.  Combined with our assumption that the slopes of $\varphi$ and $\varphi'$ do not decrease from $\beta_{k-1}$ to $\beta_k$, we see that the shortest distances from $w_k$ to the point of $D+\ddiv(\varphi)$ and the point of $D+\ddiv(\varphi')$ on $\gamma_k$ are integer multiples of $m_k$ as well.  Therefore, $2D+\ddiv(\psi)$ cannot contain a point whose shortest distance from $w_k$ is a non-integer multiple of $m_k$, and must therefore contain $w_k$, as claimed.

Without loss of generality, we assume that $D+\ddiv(\varphi)$ contains $w_k$.  Since the restriction of $D+\ddiv(\varphi)$ to $\gamma_k$ has degree at most 1, we see that this restriction is equal to $w_k$.  It follows that
\[
s_{k+1} (\varphi) = s_k (\varphi),
\]
and $\varphi$ agrees with $\phi$ on $\gamma_k$, but
\[
s_{k+1} (\varphi) = s_{k+1}(\phi) - 1.
\]

We now show that $\gamma_k$ switches slope $\tau_k (\phi)$.  By assumption, $s_{k+1} (\phi) > s_k [\tau_k (\phi)]$.  Combining this with the two equations above, we see that
\[
s_k (\varphi) \geq s_k [\tau_k (\phi)].
\]
This implies $\tau_k (\varphi) \geq \tau_k (\phi)$.  By Lemma~\ref{Lem:Skippable1}, however, we have
\[
s_{k+1} (\varphi) = s'_k [\tau'_k (\varphi)],
\]
so $\tau'_k (\varphi) < \tau'_k (\phi)$.  Combining these inequalities with the fact that slope index sequences are nondecreasing, we see that
\[
\tau_k (\phi) \leq \tau_k (\varphi) \leq \tau'_k (\varphi) < \tau'_k (\phi).
\]
Since $\phi$ is a building block, by Definition~\ref{def:buildingblock}(i), it follows that $\gamma_k$ switches slope $\tau_k (\phi)$.

We now apply property $(\ast\ast)$ again, in a similar way to the beginning of the proof.  Specifically, there is a function $\psi' \in \cB$ that agrees with $\psi$ on $\Gamma_{[1,k]}$, with the property that
\[
s_{k+1} (\psi') > s_{k+1} (\psi) = s'_k (\theta).
\]
If $\psi'$ has not been assigned to a previous loop, then it is an unassigned departing function on $\gamma_k$, and for this reason $\gamma_k$ is not skippable.  If $\psi'$ has been assigned to a previous loop, then since $\psi$ agrees with $\psi'$ on this previous loop, we see that $\psi$ must have been assigned to this previous loop as well.  We therefore arrive at a contradiction, which rules out the possibility that neither $\varphi$ nor $\varphi'$ has smaller slope on $\beta_k$ than on $\beta_{k-1}$ and completes the proof of the lemma.
\end{proof}

\begin{proof}[Proof of Proposition~\ref{Prop:Lingering}]
By Proposition~\ref{Prop:Shiny}, any shiny permissible function on $\gamma_k$ has a summand $\phi \in \cA$ satsifying $s_{k+1} (\phi) > s_k [\tau_k (\phi)]$.  We will assume that such a function $\phi$ exists, and consider unassigned permissible functions of the form $\phi + \phi' \in \cB$.  Since $\phi$ exists, by Lemma~\ref{Lem:Skippable2}, there is an unassigned permissible function $\varphi + \varphi' \in \cB$ such that
\[
s_{k+1} (\varphi) = s_k (\varphi) + 1, \ \ \ \   s_{k+1} (\varphi') = s_k (\varphi') -1,
\]
and $D + \ddiv(\varphi')$ contains either $w_k$, or a point of $\gamma_k$ whose distance from $w_k$ is not an integer multiple of $m_k$, or two points of $\gamma_k \smallsetminus \{ v_k \}$.

Both $D+\ddiv(\phi)$ and $D+\ddiv(\varphi)$ contain no points of $\gamma_k \smallsetminus \{ v_k \}$.  Any two building blocks with this property agree, and have the same slope along the bridge $\beta_k$.  It follows that $s_{k+1} (\varphi) = s_{k+1} (\phi)$.  By Lemma~\ref{Lem:Skippable1}, we also have
\[
s_{k+1} (\varphi) + s_{k+1} (\varphi') = s_{k+1} (\phi) + s_{k+1} (\phi') = s'_k (\theta).
\]
Subtracting, we see that
\[
s_{k+1} (\varphi') = s_{k+1} (\phi').
\]
Moreover, by Lemma~\ref{Lem:Skippable1}, we have
\[
s_{k+1} (\varphi') = s'_k [\tau'_k (\varphi')] \mbox{ and } s_{k+1} (\phi') = s'_k [\tau'_k (\phi')],
\]
so $\tau'_k (\varphi') = \tau'_k (\phi')$, which, by property $(\ast)$, implies that $\varphi'$ and $\phi'$ agree on $\gamma_k$.

Since $\varphi'$ and $\phi'$ agree on $\gamma_k$ and their slopes on $\beta_k$ are equal, the difference between $\ddiv \varphi'$ and $\ddiv \phi'$ on $\gamma_k$ must be supported at $v_k$.  Now the restriction of $D+\ddiv(\varphi')$ to $\gamma_k$ has degree 2 and, since $\phi + \phi'$ is shiny, the restriction of $2D+\ddiv(\phi+\phi')$ to $\gamma_k \smallsetminus \{ v_k \}$ has degree at most 1.  It follows that $D+\ddiv(\varphi')$ contains $v_k$.  This forces $D+\ddiv(\varphi')$ to be supported at points whose shortest distance to $w_k$ is an integer multiple of $m_k$.  Recall, however, that $\varphi'$ was chosen so that $D + \ddiv(\varphi')$ contains either $w_k$ or a point of $\gamma_k$ whose distance from $w_k$ is not an integer multiple of $m_k$.  We therefore see that $D+\ddiv(\varphi')$ contains $w_k$, and hence
\[
\left[ D+\ddiv(\varphi') \right] \vert_{\gamma_k} = v_k + w_k .
\]
Thus, $\varphi'$ agrees with $\varphi$ and $\phi$ on $\gamma_k$, but since $D+\ddiv(\varphi')$ contains $v_k$ and $D+\ddiv(\varphi)$ does not, we have
\[
s_k (\varphi) < s_k (\varphi').
\]
It follows that $\tau_k (\varphi) \leq \tau_k (\varphi')$.  Similarly, since $D+\ddiv(\varphi')$ contains $w_k$ and $D+\ddiv(\varphi)$ does not, we have
\[
s_{k+1} (\varphi) > s_{k+1} (\varphi').
\]
By Lemma~\ref{Lem:Skippable1}, however, we have
\[
s_{k+1} (\varphi) = s'_k [\tau'_k (\varphi)] \mbox{ and } s_{k+1} (\varphi') = s'_k [\tau'_k (\varphi')]
\]
Thus, $\tau'_k (\varphi) > \tau'_k (\varphi')$.  Combining these inequalities with the fact that slope index sequences are nondecreasing, we see that
\[
\tau_k (\varphi) \leq \tau_k (\varphi') \leq \tau'_k (\varphi') < \tau'_k (\varphi).
\]
Since $\varphi$ is a building block, by Definition~\ref{def:buildingblock}(i), it follows that $\gamma_k$ switches slope $\tau_k (\varphi)$.

By property $(\ast\ast)$, there is a function $\psi' \in \cB$ that agrees with $\phi+\phi'$ on $\Gamma_{[1,k]}$, with the property that
\[
s_{k+1} (\psi') > s_{k+1} (\phi + \phi') = s'_k (\theta).
\]
If $\psi'$ has not been assigned to a previous loop, then it is an unassigned departing function on $\gamma_k$, and for this reason $\gamma_k$ is not skippable.  If $\psi'$ has been assigned to a previous loop, then since $\phi + \phi'$ agrees with $\psi'$ on this previous loop, we see that $\phi+\phi'$ must have been assigned to this previous loop as well.  This contradicts our assumption that $\phi+\phi'$ was unassigned.  We conclude that there are no shiny functions on $\gamma_k$, as required.
\end{proof}

\begin{proof}[Proof of Proposition~\ref{Prop:GammaA}]
Let $k \in \{ z,b,b',z'+2 \}$, and note that the choice of these four loops $\gamma_k$ guarantees that there is an index $i$ such that $s_k [i] < s'_k [i]$.  We must show that $\gamma_k$ is not skippable, and that no permissible function is shiny on $\gamma_k$.  We begin by showing that $\gamma_k$ is not skippable.

Suppose $\gamma_k$ is skippable.  Then there is an unassigned permissible function $\psi = \varphi + \varphi'$ on $\gamma_k$ such that $2D + \ddiv(\varphi + \varphi')$ contains either $w_k$ or a point whose distance from $w_k$ is not an integer multiple of $m_k$, or $D + \ddiv(\varphi')$ contains two points of $\gamma_k \smallsetminus \{ v_k \}$.  By Lemma~\ref{Lem:Skippable1}, we have
\[
s_{k+1} (\varphi) + s_{k+1}(\varphi') = s'_k [\tau'_k (\varphi)] + s'_k [\tau'_k (\varphi')] = s'_k (\theta).
\]
Moreover, by Lemma~\ref{Lem:Skippable2}, we have
\[
s_{k+1} (\varphi) = s_k (\varphi) + 1,
\]
hence the restriction of $D+\ddiv(\varphi)$ to $\gamma_k$ must have degree 0, which forces $\tau'_k (\varphi) = i$.  As in Lemma~\ref{lem:nonew}, the choice of $z$, $b$, $b'$, and $z'$ ensures that there does not exist a value of $j$ such that $s'_k [i] + s'_k [j] = s'_k (\theta)$.  Thus, $\gamma_k$ cannot be a skippable loop.

It remains to show that there are no shiny permissible functions on $\gamma_k$.  Let $\varphi$ be a function satisfying $s_k (\varphi) = s_k [i]$ and $s_{k+1} (\varphi) = s'_k [i]$.  Any function that is shiny on $\gamma_k$ must agree with a function of the form $\psi = \varphi + \varphi'$ with $s_{k+1} (\varphi + \varphi') = s'_k (\theta)$.  Again, because there does not exist a value of $j$ such that $s'_k [i] + s'_k [j] = s'_k (\theta)$, we see that $s_{k+1}(\varphi')$ cannot equal $s'_k[j]$ for any $j$.  This means that
\[
s_{k+1} (\varphi') < s'_k [\tau'_k(\varphi')],
\]
and hence, $D+\ddiv(\varphi')$ contains $w_k$.  Since $\psi$ is shiny, the restriction of $2D+\ddiv(\psi)$ to $\gamma_k \smallsetminus \{ v_k \}$ has degree at most 1, and hence this restriction is exactly $w_k$.  It follows that $\psi$ agrees with $2\varphi$ on $\gamma_k$, and
\[
s_{k+1} (\psi) = s'_k (\theta) = s_{k+1} (2\varphi) -1.
\]
This implies that $s'_k (\theta)$ is odd, so $\gamma_k$ is contained in the middle block, and $s'_k(\theta) = 3$.  However, $b$ and $b'$ were chosen so that $s_{k+1}(\varphi)$ is at most 1 if the box contained in $\lambda'_k$ but not $\lambda_k$ is in the first row, and $s_{k+1}(\varphi)$ is at least 3 if this box is contained in the second or third row.
\end{proof}

We now analyze the output of the algorithm.  We first note the following.

\begin{lemma}
\label{Lem:OneEqLoop}
Suppose that $\varphi + \varphi'$ and $\phi + \phi'$ are assigned to the same loop $\gamma_k$.  Then, after possibly reordering the summands, $\varphi$ agrees with $\phi$ and $\varphi'$ agrees with $\phi'$ on $\gamma_k$.  Moreover,
\[
\tau_k (\varphi) = \tau_k (\phi) \mbox{ and } \tau_k (\varphi') = \tau_k (\phi'),
\]
\[
\tau'_k (\varphi) = \tau'_k (\phi) \mbox{ and } \tau'_k (\varphi') = \tau'_k (\phi').
\]
\end{lemma}

\begin{proof}
Since both functions are assigned to the loop $\gamma_k$, we see that $\varphi + \varphi'$ agrees with $\phi + \phi'$ on $\gamma_k$.  It suffices to show that $\varphi$ agrees with $\phi$ on $\gamma_k$.  Indeed, if $\varphi$ agrees with $\phi$ on $\gamma_k$, then the fact that $\varphi + \varphi'$ agrees with $\phi + \phi'$ will then imply that $\varphi'$ agrees with $\phi'$, and Remark~\ref{Rem:Equivalent} shows the equality of slope indices.

The fact that $\varphi + \varphi'$ agrees with $\phi + \phi'$ is equivalent to the statement that the restrictions of $2D + \ddiv(\varphi + \varphi')$ and $2D + \ddiv(\phi + \phi')$ to $\gamma_k \smallsetminus \{ v_k , w_k \}$ are the same.  Thus, if $D+\ddiv(\varphi)$ contains a point $v \in \gamma_k \smallsetminus \{ v_k , w_k \}$, then one of $D+\ddiv(\phi)$ or $D+\ddiv(\phi')$ must contain $v$ as well.  If $v$ is the only point of $\gamma_k \smallsetminus \{ v_k , w_k \}$ contained in both $D+\ddiv(\varphi)$ and $D+\ddiv(\phi)$, then $\varphi$ and $\phi$ agree on $\gamma_k$.  It follows that, if the restrictions of $D+\ddiv(\varphi), D+\ddiv(\varphi'), D+\ddiv(\phi)$, and $D+\ddiv(\phi')$ to $\gamma_k \smallsetminus \{ v_k , w_k \}$ all have degree at most 1, then the conclusion holds.

The other possibility is that the restriction of one of these 4 divisors to $\gamma_k \smallsetminus \{ v_k , w_k \}$ has degree 2.  However, this implies that either $\gamma_k$ is skippable, or the assigned functions are departing.  If both $\varphi + \varphi'$ and $\phi + \phi'$ are departing, then the restrictions to $\gamma_k$ of $2D+\ddiv(\varphi+\varphi')$ and $2D+\ddiv(\phi+\phi')$ have degree at most 1, contradicting our assumption that one of them has degree 2.  Since $\varphi + \varphi'$ and $\phi + \phi'$ are assigned to $\gamma_k$, the loop cannot be skippable.
\end{proof}

Recall that, in the vertex avoiding case, Corollary~\ref{cor:counting} says that the number of permissible functions on a block is exactly 1 more than the number of non-lingering loops in that block.  This was shown by counting the number of permissible functions on the first loop of the block, and then observing that there is at most one new permissible function on every non-lingering loop.  Since there are no new permissible functions on lingering loops, the same observation shows that the number of unassigned permissible functions never increases, when proceeding from one loop $\gamma_k$ to the next in a block, and that it decreases by one when $k \in \{z,b,b',z'+2 \}$.

In the general case, we may assign several functions to the same loop, so instead of counting individual unassigned permissible functions, we count collections of such functions, which we call \emph{cohorts} and define as follows.

\begin{definition}
We say that a function $\psi \in \cB$ \emph{leaves its shine} on the last loop $\gamma_k$ satisfying:
\begin{enumerate}
\item  $\psi$ is shiny or new on $\gamma_k$, and
\item  $\psi$ is not assigned to a loop $\gamma_{\ell}$ with $\ell < k$.
\end{enumerate}
Two unassigned permissible functions are in the same \emph{cohort} on $\gamma_\ell$ if they both leave their shine and agree on some loop $\gamma_k$, with $k \leq \ell$.
\end{definition}

\noindent Every function is new on the first loop where it is permissible.  Then, eventually, there is a loop where it leaves its shine and joins a cohort.  On any loop that is not the first loop in a block, all shiny or new functions agree, so at most one new cohort is created. This is one way in which new cohorts behave like the new permissible functions in Section~\ref{sec:counting}; there may be several new cohorts on the first loop of a block, and then at most one new cohort on each subsequent loop.  Furthermore, there are no shiny functions on $\gamma_k$, for $k \in \{z,b,b',z'+2 \}$, so no new cohorts are formed on these loops.  The next proposition says that, on each loop where a new cohort is created, and also on the special loops $\gamma_k$, for $k \in \{z,b,b',z'+2 \}$, an entire cohort is assigned.  This is essential for the proof of Theorem~\ref{Thm:ExtremalConfig}, where we bound the number of cohorts on each loop while moving from left to right across a block to show that every function in $\cB$ is assigned to some loop or bridge.

\begin{remark}
On a non-skippable loop where no new cohort is created, the functions that are assigned typically form a proper subset of a cohort.  In this way, cohorts may lose members as we move from left to right across a block, but the number of cohorts on each loop behaves just as predictably as the number of unassigned permissible functions in the vertex avoiding case.

In the vertex avoiding case, each cohort consists of a single unassigned permissible function, and the argument for counting cohorts in the proof of Theorem~\ref{Thm:ExtremalConfig} specializes to the argument for counting permissible functions in Section~\ref{sec:counting}.
\end{remark}

\begin{proposition}
\label{Prop:Counting}
Suppose that some function leaves its shine on $\gamma_\ell$, or $\ell \in \{ z,b,b',z'+2 \}$.  If $\psi \in \cB$ is assigned to $\gamma_{\ell}$, then any other function in the same cohort on $\gamma_{\ell}$ is also assigned to $\gamma_{\ell}$.
\end{proposition}

\noindent To prove Proposition~\ref{Prop:Counting}, we will use $(\dagger\dagger)$ along with a preliminary lemma.

\begin{lemma}  \label{lem:slopeindexchange}
Let $\psi = \varphi + \varphi'$ and $\psi' = \phi + \phi'$ be functions that leave their shine and agree on $\gamma_{k_0}$.  Suppose $k$ is the smallest integer such that  $k \geq k_0$ and the sets of slope indices $\{ \tau_{k+1}(\varphi), \tau_{k+1}(\varphi') \}$ and $\{ \tau_{k+1}(\phi), \tau_{k+1}(\phi') \}$ are different.  Suppose, furthermore, that $\psi$ and $\psi'$ are unassigned and permissible when we arrive at $\gamma_k$.  Then
\begin{enumerate}
\item either $\gamma_k$ is a switching loop or $\beta_{k}$ is a switching bridge,
\item $k \notin \{ z,b,b',z'+2 \}$, and
\item either $\psi$ or $\psi'$ is assigned to $\gamma_k$.
\end{enumerate}
Furthermore, if one of $\psi$, $\psi'$ is assigned to $\gamma_k$ and the other is not, then no function is shiny on $\gamma_k$.
\end{lemma}

\begin{proof}
By assumption, $\psi$ agrees with $\psi'$ on $\gamma_{k_0}$.  Also, by Lemma~\ref{Lem:ShinyEq}, we have that the sets of slope indices $\{ \tau_{k_0}(\varphi), \tau_{k_0}(\varphi') \}$ and $\{ \tau_{k_0}(\phi), \tau_{k_0}(\phi') \}$ are the same. Lemma~\ref{Lem:SameSlope} then says that  $\psi$ agrees with $\psi'$ on $\gamma_t$ for all $t$ in the range $k_0 \leq t < k$.

After possibly relabeling the functions, we may assume that $\tau_k(\varphi) = \tau_k(\phi)$ and $\tau_{k}(\varphi') = \tau_{k}(\phi')$, and suppose
\[
\tau_{k+1}(\varphi) > \tau_{k+1}(\phi).
\]
Since slope indices of building blocks only change due to switching (Definition~\ref{def:buildingblock}), it follows that either $\gamma_k$ switches slope $\tau_k(\phi)$ or $\beta_k$ switches slope $\tau'_k(\phi)$.  It follows that $k \notin \{ z,b,b',z'+2 \}$.
Since a switching loop or bridge can switch at most one slope, we also have
\[
\tau_{k+1}(\varphi') \geq \tau_{k+1}(\phi'),
\]
with equality if $\tau_k (\varphi) \neq \tau_k (\varphi')$.

\medskip

 \emph{Claim 1:  Either $\psi$ or $\psi'$ is assigned to $\gamma_k$.}  Suppose that neither $\psi$ nor $\psi'$ is assigned to $\gamma_k$.  Then $\psi$ and $\psi'$ are both permissible and not shiny on $\gamma_{k+1}$. Therefore,
\[
s_{k+1} (\varphi) = s_{k+1} [\tau_{k+1}(\varphi)], s_{k+1} (\varphi') = s_{k+1} [\tau_{k+1}(\varphi')],
\]
\[
s_{k+1} (\phi) = s_{k+1} [\tau_{k+1}(\phi)], s_{k+1} (\phi') = s_{k+1} [\tau_{k+1}(\phi')].
\]
Summing these, we obtain
\[
s_{k+1} (\psi) = s_{k+1} [\tau_{k+1}(\varphi)] + s_{k+1} [\tau_{k+1}(\varphi')] > s_{k+1} [\tau_{k+1} (\phi)] + s_{k+1} [\tau_{k+1}(\phi')] = s_{k+1} (\psi').
\]
It follows that $\psi$ is departing on $\gamma_k$.  This contradicts the supposition that neither $\psi$ nor $\psi'$ is assigned to $\gamma_k$ and proves the claim.

\medskip

It remains to show that if one of $\psi$, $\psi'$ is not assigned to $\gamma_k$, then no function is shiny on $\gamma_k$.

\medskip

\emph{Claim 2: If $\psi$ and $\psi'$ agree on $\gamma_k$, then no function is shiny.}  This is straightforward.  Indeed, if $\psi$ and $\psi'$ agree and one is assigned while the other is not, then the one that is assigned is departing and the other is not.  In this case, no function is shiny on $\gamma_k$ by property $(\dagger\dagger)$.

\medskip

For the remainder of the proof, we assume $\psi$ and $\psi'$ do not agree on $\gamma_k$, and show that no function is shiny.

\medskip

\emph{Claim 3: The functions $\varphi$ and $\phi$ do not agree on $\gamma_k$.}  Since $\psi$ and $\psi'$ do not agree on $\gamma_k$, either $\varphi$ and $\phi$ do not agree on $\gamma_k$, or $\varphi'$ and $\phi'$ do not agree on $\gamma_k$.  By property $(\ast)$, if $\varphi'$ and $\phi'$ do not agree, then $\tau'_k (\varphi') \neq \tau'_k (\phi')$.  This implies that $\gamma_k$ switches slope $\tau'_k (\varphi')$.  Since a switching loop can switch at most one slope, it follows that $\varphi'$ agrees with $\varphi$ on $\gamma_k$, $\tau_k (\varphi') = \tau_k (\varphi)$, and $\tau_{k+1} (\varphi') = \tau_{k+1} (\varphi)$.  In this case, we may relabel $\varphi$ and $\varphi'$ without loss of generality, and the result follows.

\medskip

\emph{Claim 4: The loop $\gamma_k$ is a decreasing loop and switches slope $\tau_k(\phi)$.}  To see this, note that neither $\psi$ nor $\psi'$ is shiny on $\gamma_k$, hence
\[
s_k (\varphi) = s_k [\tau_k(\varphi)], s_k (\varphi') = s_k [\tau_k(\varphi')],
\]
\[
s_k (\phi) = s_k [\tau_k(\phi)], s_k (\phi') = s_k [\tau_k(\phi')].
\]
It follows that
\[
s_k (\varphi) = s_k (\phi) \mbox{ and } s_k (\varphi') = s_k (\phi').
\]
Recall that, on a loop, every divisor of degree 1 is equivalent to a unique effective divisor.  Thus, since $\varphi$ and $\phi$ have the same incoming slope,
if the restrictions of $D+\ddiv(\varphi)$ and $D+\ddiv(\phi)$ to $\gamma_k \smallsetminus \{ w_k \}$ each have degree at most 1, then $\varphi$ agrees with $\phi$ on $\gamma_k$.  Because we showed, in the previous claim, that $\varphi$ and $\phi$ do not agree on $\gamma_k$, the restriction of $D+\ddiv(\phi)$ to $\gamma_k \smallsetminus \{ w_k \}$ has degree 2.  Equivalently,
\[
s_{k+1} (\phi) = s_k (\phi) - 1,
\]
so $\gamma_k$ is a decreasing loop and hence has positive multiplicity.  We already showed that either $\gamma_k$ switches slope $\tau_k(\phi)$ or $\beta_k$ is a switching bridge.  Since switching bridges have multiplicity 2 and the sum of all multiplicities of loops and bridges is at most 2, it follows that $\gamma_k$ must in fact switch slope $\tau_k (\phi)$, as claimed.

\medskip

We now complete the proof that no function is shiny on $\gamma_k$.  Suppose $\eta + \eta'$ is shiny on $\gamma_k$.  By Proposition~\ref{Prop:Shiny}, after possibly relabeling, the restriction of $D+\ddiv(\eta)$ to $\gamma_k \smallsetminus \{ v_k \}$ has degree 0.  Because $\psi'$ is permissible on $\gamma_k$, and $s_{k+1} (\phi) < s_k (\phi)$, as shown above, we must have
\[
s_{k+1} (\phi') = s_k (\phi') + 1.
\]
Thus the restriction of $D + \ddiv(\phi')$ to $\gamma_k \smallsetminus \{ v_k \}$ also has degree 0, $\eta$ agrees with $\phi'$ on $\gamma_k$, and
\[
s_{k+1} (\eta') = s_{k+1} (\phi').
\]
Since $\eta + \eta'$ is permissible on $\gamma_k$, we also have
\[
s_{k+1} (\eta) \geq s_{k+1} (\phi).
\]
At the same time, since $\eta + \eta'$ is shiny,
\[
s_k [\tau_k (\eta)] + s_k [\tau_k (\eta')] < s_k (\theta) = s_k [\tau_k (\phi)] + s_k [\tau_k (\phi')].
\]
It follows that either $\tau_k (\eta) < \tau_k (\phi)$, and $\gamma_k$ switches slope $\tau_k (\eta)$, or $\tau_k (\eta') < \tau_k (\phi')$, and $\gamma_k$ switches slope $\tau_k (\eta')$.  However, we will show that neither of this is possible.  Indeed, the first is impossible because $\gamma_k$ switches slope $\tau_k (\phi)$, and a loop can switch at most 1 slope.  The second requires
\[
\tau_k (\eta') = \tau_k (\phi')-1 = \tau_k (\phi).
\]
However, since
\[
s_k [\tau_k (\eta') < s_k [\tau_k (\phi')] < s'_k [\tau'_k (\phi')],
\]
we see that $s_k [\tau_k (\eta')] \leq s'_k [\tau_k (\eta')+1]+2$.  Since the slope of a function in $R(D)$ can increase by at most one from the left side of $\gamma_k$ to the right side, we see that there is no function $\eta''$ with $s_k (\eta'') \leq s_k [\tau_k (\eta')]$ and $s'_k (\eta'') \geq s'_k [\tau_k (\eta')+1]$.  This shows that it is impossible for $\gamma_k$ to switch slope $\tau_k (\phi') - 1$, and completes the proof of the lemma.
\end{proof}

We now prove Proposition~\ref{Prop:Counting}.

\begin{proof}[Proof of Proposition~\ref{Prop:Counting}]
Suppose that $\psi$ and $\psi'$ are in the same cohort on $\gamma_\ell$, and that the functions assigned to $\gamma_\ell$ include $\psi$ but not $\psi'$.  We must show that $\ell \not \in \{ z,b,b',z'+2 \}$ and no function leaves its shine on $\gamma_\ell$.

Let $\gamma_{k_0}$ be the loop where $\psi$ and $\psi'$ leave their shine.  By Lemma~\ref{lem:slopeindexchange}, if the set of slope indices $\{ \tau_{k}(\varphi), \tau_{k}(\varphi') \}$ and $\{ \tau_{k}(\phi), \tau_{k}(\phi') \}$ are different for some $k_0 \leq k \leq \ell$, then one of $\psi$ or $\psi'$ is assigned to $\gamma_k$, contradicting our assumption that they are in the same cohort on $\gamma_{\ell}$.  Furthermore, if $\{ \tau_{\ell+1}(\varphi), \tau_{\ell+1}(\varphi') \}$ and $\{ \tau_{\ell+1}(\phi), \tau_{\ell+1}(\phi') \}$ are different, then by Lemma~\ref{lem:slopeindexchange}, $\ell \notin \{ z,b,b',z'+2 \}$ and no function leaves its shine on $\gamma_\ell$.  We may therefore assume that the sets of slope indices $\{ \tau_{k}(\varphi), \tau_{k}(\varphi') \}$ and $\{ \tau_{k}(\phi), \tau_{k}(\phi') \}$ are the same for $k_0 \leq k \leq \ell + 1$.   By Lemma~\ref{Lem:SameSlope}, this implies that
\[
\tau'_{\ell}(\varphi) = \tau'_{\ell}(\phi) \mbox{ and } \tau'_{\ell}(\varphi') = \tau'_{\ell}(\phi')
\]
and both pairs of functions agree on $\gamma_{\ell}$ by property $(\ast)$.

Since $\psi$ and $\psi'$ agree on $\gamma_\ell$, and the functions assigned to $\gamma_\ell$ include $\psi$ but not $\psi'$, we see that $\psi$ is a departing function, but $\psi'$ is not.  In this case, since $\psi'$ leaves the shine on $\gamma_{k_0}$, we see that $\psi'$ is not shiny on $\gamma_{\ell+1}$, hence we must have
\[
s_{\ell+1} (\phi) = s_{\ell+1} [\tau_{\ell+1}(\phi)] \mbox{ and } s_{\ell+1} (\phi') = s_{\ell+1} [\tau_{\ell+1}(\phi')].
\]
We show that $\ell \notin \{ z,b,b',z'+2 \}$.  If $\gamma_{\ell}$ is a switching loop or $\beta_{\ell}$ is a switching bridge, then $\ell \notin \{ z,b,b',z'+2 \}$, hence we may assume that $\tau_{\ell} = \tau'_{\ell} = \tau_{\ell+1}$ for each of the functions $\varphi, \varphi', \phi,$ and $\phi'$.  Because $\psi$ is departing, we have either
\[
s'_{\ell} [\tau'_{\ell} (\varphi)] > s_{\ell} [\tau_{\ell} (\varphi)] \mbox{ or } s'_{\ell} [\tau'_{\ell} (\varphi)] > s_{\ell} [\tau_{\ell} (\varphi)].
\]
Assume without loss of generality that the first inequality holds.  If $s_{\ell+1} [\tau_{\ell+1}(\varphi)] < s'_{\ell} [\tau'_{\ell}(\varphi)]$, then $\lambda_{\ell+1}$ is contained in $\lambda_{\ell}$, hence $\ell \notin \{ z,b,b',z'+2 \}$.   On the other hand, suppose that $s_{\ell+1} [\tau_{\ell+1}(\varphi)] = s'_{\ell} [\tau'_{\ell}(\varphi)]$.  Since $\psi'$ is not departing, we have
\[
s_{\ell} (\theta) = s'_{\ell} (\psi') = s'_{\ell} [\tau'_{\ell+1}(\varphi)] + s'_{\ell} [\tau'_{\ell}(\varphi')].
\]
But $z,b,b',$ and $z'$ are chosen so that there is no integer $j$ such that $s_{\ell} (\theta) = s'_{\ell} [\tau'_{\ell}(\phi)] + s'_{\ell} [j]$, so again $\ell \notin \{ z,b,b',z'+2 \}$.

Because $\psi$ is departing, either $\varphi$ or $\varphi'$ must have higher slope on $\beta_{\ell}$ than on $\beta_{\ell-1}$.  Without loss of generality we may assume that $s_{\ell+1} (\varphi) > s_{\ell} (\varphi)$.  Our assumption that the slope indices of $\varphi$ and $\varphi'$ agree with those of $\phi$ and $\phi'$ implies that either
\[
s_{\ell+1} (\varphi) > s_{\ell+1} [\tau_{\ell+1}(\varphi)] \mbox{ or } s_{\ell+1} (\varphi') > s_{\ell+1} [\tau_{\ell+1}(\varphi')].
\]
In other words, either
\[
s'_{\ell} [\tau'_{\ell} (\varphi)] > s_{\ell+1} [\tau'_{\ell} (\varphi)] \mbox{ or } s'_{\ell} [\tau'_{\ell} (\varphi')] > s_{\ell+1} [\tau'_{\ell} (\varphi')].
\]

Now, assume that $\eta + \eta' \in \cB$ is a shiny function on $\gamma_{\ell}$.  In order to show that $\eta + \eta'$ does not leave the shine on $\gamma_{\ell}$, we first show that $\eta+\eta'$ cannot agree with a departing function on $\gamma_{\ell}$.  Any function that is both departing and shiny on $\gamma_{\ell}$ agrees with $2\varphi$, and such a function only exists if $s_{\ell+1} (2\varphi) = s'_{\ell} (\theta)+1$.  From this we see that both
\[
s_{\ell+1} (\varphi) > s_{\ell+1} [\tau_{\ell+1}(\varphi)] \mbox{ and } s_{\ell+1} (\varphi') > s_{\ell+1} [\tau_{\ell+1}(\varphi')].
\]
But, because
\[
s_{\ell+1} (\phi) = s_{\ell+1} [\tau_{\ell+1}(\phi)] \mbox{ and } s_{\ell+1} (\phi') = s_{\ell+1} [\tau_{\ell+1}(\phi')],
\]
we see that $s_{\ell+1} (\phi+\phi') = s_{\ell} (\theta) -1$.  This implies that $\phi+\phi'$ is not permissible on $\gamma_{\ell}$, a contradiction.

We now show that $\eta + \eta'$ does not leave the shine on $\gamma_{\ell}$.  To see this, we will prove by case analysis that one of the following holds:
\[
s_{\ell+1} (\eta) < s'_{\ell} [\tau'_{\ell}(\eta)], s_{\ell+1} (\eta') < s'_{\ell} [\tau'_{\ell}(\eta')],
\]
\[
s_{\ell+1} (\eta ) > s_{\ell+1} [\tau_{\ell+1}(\eta)], \mbox{ or } s_{\ell+1} (\eta') > s_{\ell+1} [\tau_{\ell+1}(\eta')].
\]
To see that the claim follows, note that if one of the first two inequalities holds, then $2D+\ddiv(\eta+\eta')$ contains $w_{\ell}$, hence $\eta+\eta'$ agrees with a departing function on $\gamma_{\ell}$, a contradiction.  If one of the second two inequalities holds, we see that $\eta+\eta'$ is shiny on $\gamma_{\ell+1}$, and therefore does not leave the shine on $\gamma_{\ell}$.

It therefore remains to show that one of the inequalities above holds.  By Proposition~\ref{Prop:Shiny}, we may assume that the restriction of $D+\ddiv(\eta)$ to $\gamma_{\ell} \smallsetminus \{ v_{\ell} \}$ has degree 0.  It follows that $\eta$ agrees with $\varphi$ on $\gamma_{\ell}$, and $s_{\ell+1} (\eta) = s_{\ell+1} (\varphi)$.  Since $\eta+\eta'$ is not departing on $\gamma_{\ell}$, we have $s_{\ell+1} (\eta') = s_{\ell+1} (\varphi')-1$.  By property $(\dagger\dagger)$, the bridge $\beta_{\ell}$ is not a switching bridge, so $\tau_{\ell+1} (\eta) = \tau'_{\ell} (\eta)$ and $\tau_{\ell+1} (\eta') = \tau'_{\ell} (\eta')$.

We now consider several cases.  First, suppose that $s_{\ell+1} (\varphi) > s_{\ell+1} [\tau_{\ell+1} (\varphi)]$.  If $\tau'_{\ell} (\eta) > \tau'_{\ell} (\varphi)$, then
\[
s'_{\ell} [\tau'_{\ell} (\eta)] > s'_{\ell} [\tau'_{\ell} (\varphi)] \geq s_{\ell+1} (\varphi) = s_{\ell+1} (\eta).
\]
On the other hand, if $\tau'_{\ell} (\eta) \leq \tau'_{\ell} (\varphi)$, then
\[
s_{\ell+1} (\eta) = s_{\ell+1} (\varphi) > s_{\ell+1} [\tau_{\ell+1} (\varphi)] \geq s_{\ell+1} [\tau_{\ell+1} (\eta)].
\]
Next, suppose that $s_{\ell+1} (\varphi') > s_{\ell+1} [\tau_{\ell+1} (\varphi')]$.  If $\tau'_{\ell} (\eta') \geq \tau'_{\ell} (\varphi')$, then
\[
s'_{\ell} [\tau'_{\ell} (\eta')] \geq s'_{\ell} [\tau'_{\ell} (\varphi')] \geq s_{\ell+1} (\varphi') > s_{\ell+1} (\eta').
\]
On the other hand, if $\tau'_{\ell} (\eta') < \tau'_{\ell} (\varphi')$, then
\[
s_{\ell+1} (\eta') = s_{\ell+1} (\varphi')-1 > s_{\ell+1} [\tau_{\ell+1} (\varphi')] - 1 \geq s_{\ell+1} [\tau_{\ell+1} (\eta')].
\]
\end{proof}

\begin{proof}[Proof of Theorem~\ref{Thm:ExtremalConfig}]
If $\psi \in \cB$ is not permissible on any loop, then either $s_1 (\psi) >4$ or $s_{g+1} (\psi) < 2$.  By construction, we see that $\psi$ achieves the minimum either on the first or last bridge.

We argue that every permissible function on the first block $\Gamma_{[1,z]}$ is assigned to a loop or bridge in the first block.  The other blocks follow by a similar argument.  We first consider the case where every non-skippable loop in $\Gamma_{[1,z]}$ has at least one assigned function.  For each loop $\gamma_k$ in the block, we consider the number of cohorts on $\gamma_k$ with the property that some function $\psi$ in the cohort is not assigned to $\gamma_k$.  We will show by induction that, for $k<z$, the number of such cohorts on $\gamma_k$ is at most 2.  By our enumeration of the possible slope vectors $s'_0(\Sigma)$, described in the algorithm under the first bridge, we see that there are at most three cohorts on $\gamma_1$, and at most two if $\gamma_1$ is skippable.  If $\gamma_1$ is not skippable, then there is a function $\psi \in \cB$ that is assigned to $\gamma_1$.  This function leaves the shine on $\gamma_1$, so by Proposition~\ref{Prop:Counting}, any function in the same cohort as $\psi$ is also assigned to $\gamma_1$.  It follows that the number of cohorts such that some function $\psi$ in the cohort is not assigned to $\gamma_1$ is at most 2.

As we proceed from left to right across the block, every time we reach a new loop, there are two possibilities.  One possibility is that no function leaves the shine on $\gamma_k$, in which case by definition there are no more cohorts on $\gamma_k$ than there are on $\gamma_{k-1}$.  The other possibility is that some function $\psi$ leaves the shine on $\gamma_k$.  In this case, by assumption, there are permissible functions on $\gamma_k$, and $\gamma_k$ is not skippable, so some function is assigned to $\gamma_k$.  By Proposition~\ref{Prop:Counting}, any function in the same cohort as $\psi'$ is also assigned to $\gamma_k$.  It follows that the number of cohorts on $\gamma_k$ such that some function in the cohort is not assigned to $\gamma_k$ is equal to the number of cohorts on $\gamma_{k-1}$ such that some function in the cohort is not assigned to $\gamma_{k-1}$.  Specifically, as we proceed from $\gamma_{k-1}$ to $\gamma_k$, we introduce the cohort of $\psi$, but we remove the cohort of $\psi'$.  By induction, therefore, the number of cohorts on $\gamma_k$ with the property that some function in the cohort is not assigned to $\gamma_k$ is at most 2.

By Proposition~\ref{Prop:GammaA}, no function is shiny on $\gamma_z$, and $\gamma_z$ is not skippable.  Combining this with our enumeration of cohorts in the preceding paragraph, we see that there are at most 2 cohorts on $\gamma_z$. By assumption, there is a function $\psi \in \cB$ that is assigned to $\gamma_z$, and by Proposition~\ref{Prop:Counting}, any function in the same cohort on $\gamma_z$ is also assigned to $\gamma_z$.  After assigning this cohort, there is at most one cohort left.  Also by Lemma~\ref{lem:slopeindexchange}, if $\varphi + \varphi'$ and $\phi + \phi'$ are in the remaining cohort, then the sets of slope indices $\{ \tau'_z (\varphi) , \tau'_z (\varphi') \}$ and $\{ \tau'_z (\phi) , \tau'_z (\phi') \}$ are the same, hence $\varphi + \varphi'$ and $\phi + \phi'$ agree on $\gamma_z$.  It follows that everything in the remaining cohort is assigned to the bridge $\beta_z$.

On the other hand, suppose that there is a non-skippable loop with no assigned function, and let $\gamma_k$ be the first such loop.  If $\psi$ is a function that was permissible on an earlier loop but not permissible on $\gamma_k$, then there is a $k'<k$ such that $\psi$ is a departing permissible function on $\gamma_{k'}$.  By construction, this function must be assigned to loop $\gamma_{k'}$, or an earlier loop.  It follows that all functions $\psi$ that are permissible on loops $\gamma_{k'}$ for $k'<k$ have been assigned.  Now, for each non-skippable loop $\gamma_{k'}$ with $k'>k$, there is at most one equivalence class of new permissible function on $\gamma_{k'}$, and on skippable loops, there are none.  By construction, since there is only one equivalence class of unassigned permissible functions on $\gamma_{k'}$, this equivalence class is assigned to the loop $\gamma_{k'}$.  In this way, every function that is permissible on the block is assigned to some loop.
\end{proof}

Theorem~\ref{Thm:ExtremalConfig} shows that, if two functions are assigned to the same loop $\gamma_k$ or the same bridge $\beta_k$, then they agree on $\gamma_k$.  In fact, slightly more is true.

\begin{lemma}
\label{Lem:OneEq}
Suppose that $\varphi + \varphi'$ and $\phi + \phi'$ are assigned to the same loop $\gamma_k$ or bridge $\beta_k$.  Then, after possibly reordering $\varphi$ and $\varphi'$, we have that $\varphi$ agrees with $\phi$ on $\gamma_k$ and $\varphi'$ agrees with $\phi'$ on $\gamma_k$.  Moreover, we have
\[
\tau'_k (\varphi) = \tau'_k (\phi) \mbox{ and } \tau'_k (\varphi') = \tau'_k (\phi').
\]
\end{lemma}

\begin{proof}
In the case where the two functions are assigned to the same loop, this is simply Lemma~\ref{Lem:OneEqLoop}.  It therefore suffices to consider the case where they are assigned to the same bridge.   We see from the proof of Theorem~\ref{Thm:ExtremalConfig} that, if $\varphi + \varphi'$ and $\phi + \phi'$ are assigned to the bridge $\beta_k$, then they are in the same cohort on $\gamma_k$.  By Proposition~\ref{Prop:Counting}, if the set of slope indices $\{ \tau'_k (\varphi), \tau'_k (\varphi') \}$ is different from $\{ \tau'_k (\varphi), \tau'_k (\varphi') \}$, then one of the two functions is assigned to the loop $\gamma_k$.  It follows that, if both functions are assigned to the bridge $\beta_k$, then these two sets of slope indices are the same, and the conclusion holds by property $(\ast)$.
\end{proof}

\section{Proof of Theorem~\ref{Thm:MainThm}}
\label{Sec:Generic}

Recall that the master template $\theta$, constructed in the previous section, is only an intermediary step in our argument.  Our goal is to construct a maximally independent configuration of 28 pairwise sums of functions in $\Sigma$, which may not contain the building blocks.

We proceed by considering several cases, depending on the number of switching loops and bridges.  Recall that switching loops have positive multiplicity, as do decreasing loops and bridges, and switching bridges have multiplicity at least 2.  Moreover, the sum of all multiplicities of loops and bridges is bounded above by $\rho$, which we assume to be at most 2.  In particular, the existence of a switching bridge precludes the existence of a switching loop, and in any case there can be at most two switching loops.  Therefore, $\Sigma$ falls into one of the following cases:
\begin{enumerate}
\item  There are no switching loops or bridges.
\item  There is a switching bridge.
\item  There is one switching loop.
\item  There are two switching loops.
\end{enumerate}
In the first and third cases, we consider subcases depending on the number of decreasing bridges.  The case of two switching loops is the most delicate, and we consider subcases depending on the relationship between the two switching loops.

Our basic strategy is the same throughout.  We first identify a collection $S$ of 7 or more functions in $\Sigma$.  All of these functions are either  building blocks or have a particularly simple expression as a tropical linear combination of building blocks.  There will often be helpful relations between these functions, which we highlight.  We then identify a collection $\cA$ of building blocks satisfying property $(\ast)$, and a collection $\cB$ of pairwise sums of elements of $\cA$ satisfying $(\ast\ast)$ and $(\dagger\dagger)$.  This allows us to run the algorithm from Section~\ref{Sec:Construction} to construct a master template $\theta$.  Then, we identify a collection of 28 pairwise sums of functions in $S$, form a tropical linear combination by fitting each of these functions to the master template, and verify that this combination is maximally independent.  The process of fitting a function to the master template may be thought of as a ``best approximation from above."  The following lemma captures this idea.

\begin{lemma}
\label{Lem:Replace}
Let $\theta = \min_{\psi \in \cB} \{ \psi + c_\psi \}$, and assume that each $\psi \in \cB$ achieves the minimum at some point of $\Gamma$.  Then, for any subset $\cB' \subset \cB$ and any $\varphi = \min_{\psi'  \in \cB'} \{ \psi' + a_{\psi'} \}$, there is some $b$ such that $\varphi + b \geq \theta$, with equality on the entire region where some $\psi' \in \cB'$ achieves the minimum in $\theta$.
\end{lemma}

\begin{proof}
Let $b$ be the maximum of $c_{\psi} - a_{\psi}$, and choose $\psi' \in \cB'$ such that $b = c_{\psi'} - a_{\psi'}$.  Then $\varphi + b \geq \min \{ c_{\psi} + \varphi_{\psi} \} = \theta$, with equality at points where $\psi'$ achieves the minimum in $\theta$.
\end{proof}

\subsection{Case 1:  no switching loops or bridges}

Suppose there are no switching loops or bridges.  By Corollary~\ref{Cor:GenericFns}, for each $i$ there is a function $\varphi_i \in \Sigma$ such that
\[
s_k (\varphi_i) = s_k [i] \mbox{ and } s'_k (\varphi_i) = s'_k [i], \mbox{ for all } k.
\]
We will show that the set of 28 functions $\{ \varphi_i + \varphi_j \}$ is maximally independent.

\medskip

\subsubsection{Subcase 1a:  no decreasing bridges}  This includes the case where $D$ is vertex avoiding.

\begin{lemma}
\label{Lem:BB}
In this subcase, each of the functions $\varphi_i$ is a building block.  The set $\cA = \{ \varphi_i \}$ satisfies property $(\ast)$, and the set $\cB = \{ \varphi_i + \varphi_j \}$ satisfies properties $(\ast\ast)$ and $(\dagger\dagger)$.
\end{lemma}

\begin{proof}
Because there are no decreasing bridges, the functions $\varphi_i$ have constant slope along each bridge.  Also, the slope index sequence associated to $\varphi_i$ is the constant sequence $i$.  It follows that each function $\varphi_i$ is a building block.  From this we see that, if $i \neq j$, then $\tau'_k (\varphi_i) \neq \tau'_k (\varphi_j)$ for any $k$, and thus property $(\ast)$ is satisfied.  Since there are no switching loops and
\[
s_{k+1} (\varphi_i) = s'_k [\tau'_k (\varphi_i)] = s'_k[i] \mbox{ for all } k,
\]
property $(\ast\ast)$ is satisfied vacuously.  Since there are no switching loops or bridges, property $(\dagger\dagger)$ is satisfied vacuously as well.
\end{proof}

By Lemma~\ref{Lem:BB}, the set $\cB$ satisfies the hypotheses of Theorem~\ref{Thm:ExtremalConfig}, and hence there is a tropical linear combination $\theta$ of the functions $\varphi_{ij}$ such that each function $\varphi_{ij} = \varphi_i + \varphi_j$ achieves the minimum at some point of the loop or bridge to which it is assigned in the master template algorithm, and all other functions that achieve the minimum at this point agree with $\varphi_{ij}$ on this loop or bridge.

Since each $\varphi_{ij}$ is a sum of two elements of $\Sigma$, we set $\vartheta = \theta$.

\begin{theorem}
\label{thm:easycase}
In this subcase, $\vartheta$ is a maximally independent combination.
\end{theorem}

\begin{proof}
It suffices to show that if two functions $\varphi_{ij}, \varphi_{i'j'}$ are assigned to the same loop or bridge, then $(i,j) = (i',j')$.  This follows directly from Lemma~\ref{Lem:OneEq}, since $\tau'_k (\varphi_i) = i$ for all $i$ and $k$.
\end{proof}

\medskip

\subsubsection{Subcase 1b: one decreasing bridge, of multiplicity one}  In this subcase the functions $\varphi_i$ are not all building blocks.  To illuminate our strategy, we first consider the situation where there is a bridge $\beta_{\ell}$ of multiplicity 1, and no other decreasing bridges.  In this situation, there is one index $h$ such that the slope of $\varphi_h$ decreases on $\beta_{\ell}$, and it decreases by exactly 1.

Let $v \in \beta_{\ell}$ be the point where the function $\varphi_h$ bends.  Note that $\varphi_h$ can be written as a tropical linear combination of two building blocks, both of which agree with $\varphi_h$ on $\Gamma \smallsetminus \beta_\ell$, but with different slopes on that bridge.  We label these building blocks $\varphi^0_h$ and $\varphi^\infty_h$, so that they have constant slopes $s_{\ell+1} [h]$ and $s'_{\ell} [h] = s_{\ell+1} [h] + 1$, respectively, along the bridge $\beta_{\ell}$.

\begin{lemma}
\label{Lem:BB2}
In this subcase, the set
\[
\cA = \{ \varphi_i \, \vert \, i \neq h \} \cup \{ \varphi^0_h , \varphi^{\infty}_h \}
\]
satisfies property $(\ast)$.  The set $\cB$ of pairwise sums of elements of $\cA$ satisfies properties $(\ast\ast)$ and $(\dagger\dagger)$.
\end{lemma}

\begin{proof}
Every function in $\cA$ has constant slope along each bridge.  By construction, the slope index sequence associated to $\varphi_i$ is the constant sequence $i$, and the slope index sequence associated to both $\varphi^0_h$ and $\varphi^{\infty}_h$ is the constant sequence $h$.  In particular, each function in $\cA$ is a building block.  Moreover, since $\varphi^0_h$ and $\varphi^{\infty}_h$ agree on every loop, $\cA$ satisfies property $(\ast)$.

By construction, $\varphi^0_h$ is the only function in $\cA$ with
\[
s_{k+1} (\varphi^0_h) < s'_k [\tau'_k (\varphi^0_h)] = s'_k [h],
\]
and then only when $k = \ell$.  For any $\varphi \in \cA$, the function $\varphi + \varphi^{\infty}_h$ agrees with $\varphi + \varphi^0_h$ on $\Gamma_{[1,\ell]}$, but has higher slope on $\gamma_{\ell}$.  Thus, $\cB$ satisfies property $(\ast\ast)$.  Since there are no switching loops or bridges, property $(\dagger\dagger)$ is satisfied vacuously.
\end{proof}

By Lemma~\ref{Lem:BB2}, the set $\cB$ satisfies the hypotheses of Theorem~\ref{Thm:ExtremalConfig}, so there is a tropical linear combination $\theta$ such that each function $\psi \in \cB$ achieves the minimum at some point on its assigned loop or bridge of $\Gamma$, and all other functions that achieve the minimum at this point agree with $\psi$ on this loop or bridge. Roughly speaking, we construct the required maximally independent combination $\vartheta$ from $\theta$ by replacing the pair of functions $\varphi^0_h$, $\varphi^{\infty}_h$ with the single function $\varphi_h$, wherever these appear as summands.  More precisely, by Lemma~\ref{Lem:Replace}, since $\varphi_{hj}$ is a tropical linear combination of $\varphi^0_h + \varphi_j$ and $\varphi^{\infty}_h + \varphi_j$, we may set the coefficient of $\varphi_{hj}$ so that it dominates $\theta$, while agreeing with one of $\varphi^0_h + \varphi_j$ or $\varphi^{\infty}_h + \varphi_j$ on the region where this function achieves the minimum in the master template.

\begin{theorem}
\label{Thm:BendOnBridge}
In this subcase, $\vartheta$ is a maximally independent combination.
\end{theorem}

\begin{proof}
By Lemma~\ref{Lem:OneEq}, if two functions $\psi, \psi' \in \cB$ are assigned to the same loop, then these two functions must be either
\[
\psi = \varphi^0_h + \varphi , \psi' = \varphi^{\infty}_h + \varphi \text{ for some } \varphi \in \cA, \text{ or}
\]
\[
\psi = 2\varphi^0_h , \psi' = 2\varphi^{\infty}_h.
\]
Thus, in the master template $\theta$, for each function $\varphi_{ij}$ with $i,j \neq h$, there is a point $v \in \Gamma$ where $\varphi_{ij}$ is the only function to achieve the minimum.  By Lemma~\ref{Lem:Replace}, the region where $\varphi_{hj}$ achieves the minimum in $\vartheta$ contains the region where one of $\varphi^0_h + \varphi_j, \varphi^{\infty}_h + \varphi_j$ achieves the minimum in the master template $\theta$.  Similarly, the region where $\varphi_{hh}$ achieves the minimum in $\vartheta$ contains the region where one of $2\varphi^0_h, \varphi^0_h + \varphi^{\infty}_h, 2\varphi^{\infty}_h$ achieves the minimum in the master template $\theta$.

Thus, if $\varphi_{ij}$ with $i,j \neq h$ is the only function to achieve the minimum at some point $v$ in the master template $\theta$, it continues to be the only function to achieve the minimum at $v$ in the best approximation $\vartheta$.  Similarly, there is a point $v \in \Gamma$ where $\varphi_{hj}$ is the only function to achieve the minimum.  This point $v$ is contained either in the loop or bridge to which $\varphi^0_h + \varphi_j$ is assigned in the master template, or in the loop or bridge to which $\varphi^{\infty}_h + \varphi_j$ is assigned in the master template.
\end{proof}

\begin{remark}
Note that the master template constructed in Theorem~\ref{Thm:ExtremalConfig} is independent of the point $v \in \beta_{\ell}$ where $\varphi_h$ bends.  The dependence on $v$ appears in Theorem~\ref{Thm:BendOnBridge} when we use the master template to obtain a maximally independent combination of the functions $\varphi_{ij}$.  Consider the case where there is a point $w \in \beta_{\ell}$ where $\varphi^0_h + \varphi_j$ and $\varphi^{\infty}_h + \varphi_j$ cross.  In this case, as we move $v$ from one side of $w$ to the other, our maximally independent combination transitions between two combinatorial types.  When $v$ is to the left of $w$, $\varphi_{hj}$ achieves the minimum where $\varphi^0_h + \varphi_j$ does, and dominates $\varphi^{\infty}_h + \varphi_j$.  Similarly, when $v$ is to the right of $w$, $\varphi_{hj}$ achieves the minimum where $\varphi^{\infty}_h + \varphi_j$ does, and dominates $\varphi^0_h + \varphi_j$. This situation is illustrated in Figure~\ref{Fig:PhaseTransition}.
\end{remark}

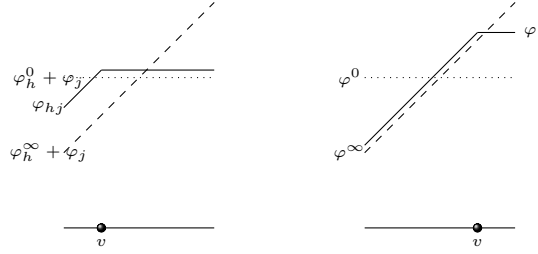
\begin{figure}[H]
\begin{tikzpicture}

\draw (0,0)--(2,0);
\draw [dotted](0,2)--(2,2);
\draw [dashed](0,1)--(2,3);
\draw (0,1.6)--(0.5,2.1);
\draw (0.5,2.1)--(2,2.1);
\draw [ball color=black] (0.5,0) circle (0.55mm);
\draw (0.5,-0.2) node {{\tiny $v$}};
\draw (-0.2,2) node {{\tiny $\varphi^0_h + \varphi_j$}};
\draw (-0.2,1.6) node {{\tiny $\varphi_{hj}$}};
\draw (-0.2,1) node {{\tiny $\varphi^{\infty}_h + \varphi_j$}};

\draw (4,0)--(6,0);
\draw [dotted](4,2)--(6,2);
\draw [dashed](4,1)--(6,3);
\draw (5.5,2.6)--(6,2.6);
\draw (4,1.1)--(5.5,2.6);
\draw [ball color=black] (5.5,0) circle (0.55mm);
\draw (5.5,-0.2) node {{\tiny $v$}};
\draw (3.8,2) node {{\tiny $\varphi^0$}};
\draw (6.2,2.6) node {{\tiny $\varphi$}};
\draw (3.8,1) node {{\tiny $\varphi^{\infty}$}};

\end{tikzpicture}
\caption{As we vary the point $v \in \beta_{\ell}$, the maximally independent combination transitions from one combinatorial type to another.  Specifically, when $v$ is to the left of $w$, $\varphi_{hj}$ achieves the minimum on the same region as $\varphi^0_h + \varphi_j$, whereas if $v$ is to the right of $w$, $\varphi_{hj}$ achieves the minimum on the same region as $\varphi^{\infty}_h + \varphi_j$.}
\label{Fig:PhaseTransition}
\end{figure}

\medskip

\subsubsection{Case 1c:  remaining cases without switching loops or bridges}

The remaining possibilities are that there may be two decreasing bridges of multiplicity 1, or one decreasing bridge of multiplicity 2.  We choose the set $\cA$ in a similar way to the previous case.  Specifically, we define a function $\varphi \in R(D)$ to be in $\cA$ if there is an $i$ such that:
\begin{enumerate}
\item  on every loop $\gamma_k$, $\varphi$ agrees with $\varphi_i$, and
\item  on every bridge $\beta_k$, there is a tangent vector $\zeta$ such that $\varphi$ has constant slope $s_{\zeta} (\varphi_i)$ on $\beta_k$.
\end{enumerate}

\begin{lemma}
\label{Lem:BB3}
In this subcase, the set $\cA$ satisfies property $(\ast)$.  The set $\cB$ of pairwise sums of elements of $\cA$ satisfies properties $(\ast\ast)$ and $(\dagger\dagger)$.
\end{lemma}

\begin{proof}
The proof is very similar to that of Lemma~\ref{Lem:BB2}.  Note that, as in the previous cases, since there are no switching loops or bridges, the only valid slope index sequences are the constant sequences.  By definition, if two functions in $\cA$ have the same slope index $i$, then both agree with $\varphi_i$ on every loop.  Thus, $\cA$ satisfies property $(\ast)$.

If $\varphi \in \cA$ has constant slope index sequence $i$ and $s_{k+1} (\varphi) < s'_k [i]$,
then there is some $\varphi' \in \cA$ that agrees with $\varphi$ on each component of $\Gamma \smallsetminus \beta_k$, with
$
s_{k+1} (\varphi') = s'_k [i].
$
Thus, $\cB$ satisfies property $(\ast\ast)$.  Since there are no switching loops or bridges, property $(\dagger\dagger)$ is satisfied vacuously.
\end{proof}

By Lemma~\ref{Lem:BB3}, the set $\cB$ satisfies the hypotheses of Theorem~\ref{Thm:ExtremalConfig}, so there is a tropical linear combination $\theta$ of the functions $\psi \in \cB$ such that each function $\psi$ achieves the minimum at some point of its assigned loop or bridge, and all other functions that achieve the minimum at this point agree with $\psi$ on this loop or bridge.  Now, by Lemma~\ref{Lem:Replace}, there is a tropical linear combination $\vartheta$ of the functions $\varphi_{ij}$ that dominates the master template, such that each function $\varphi_{ij}$ equals the master template at some point.

In any of these cases, the functions $\varphi_i$ are tropical linear combinations of the building blocks.  Moreover, for $i \neq j$, the set of building blocks needed to construct $\varphi_i$ as a tropical linear combination is disjoint from the set of building blocks needed to construct $\varphi_j$.  For this reason, the argument of Theorem~\ref{Thm:BendOnBridge} carries through essentially without change.

\begin{theorem}
\label{Thm:BendOnBridge2}
In this subcase, $\vartheta$ is a maximally independent combination.
\end{theorem}

\begin{proof}
The proof in this case is very similar to that of Theorem~\ref{Thm:BendOnBridge}.  Recall that, because there are no switching loops or bridges, the only valid slope index sequences are constant sequences.  For each $i$, the function $\varphi_i$ is a tropical linear combination of the building blocks with slope index sequence $i$.  Thus, in the tropical linear combination $\vartheta$, the function $\varphi_{ij}$ achieves the minimum on the same region as a function $\varphi + \varphi'$, where $\varphi$ has constant slope index sequence $i$ and $\varphi'$ has constant slope index sequence $j$.  Second, by Lemma~\ref{Lem:OneEq}, if another function $\phi + \phi' \in \cB$ is assigned to the same loop or bridge as $\varphi + \varphi'$, then $\phi$ must have constant slope index sequence $i$, and $\phi'$ must have constant slope index sequence $j$.  It follows that there is a point $v$ in this loop or bridge where $\varphi_{ij}$ achieves the minimum uniquely.
\end{proof}

\subsection{Case 2:  a switching bridge}  \label{Sec:SwitchBridgeCase}

In this section, we consider the case where there is a switching bridge $\beta_{\ell}$.  By Lemma~\ref{Lem:BridgeSwitch}, there is a unique index $h$ such that
\[
s'_{\ell} [h+1] = s'_{\ell} [h] +1 = s_{\ell+1} [h+1] + 1 = s_{\ell+1} [h] + 2 .
\]
Moreover, there is a point $x \in \beta_{\ell}$ where both slopes decrease.  For each $i$, the incoming slope at this point $s_x[i]$ is equal to the slope $s'_\ell[i]$ at the beginning of the bridge.  Similarly, the outgoing slope $s'_x[i]$ is equal to the slope $s_{\ell + 1}[i]$ at the end of the brige.  Note that, since switching bridges have multiplicity 2, no other bridge or loop has positive multiplicity.

By Lemma~\ref{Lem:GenericFns}, for all $j \notin \{h,h+1\}$, there is a function $\varphi_j \in \Sigma$ with
\[
s_k (\varphi_j) = s_k [j] \mbox{ and } s'_k (\varphi_j) = s'_k [j] \mbox{ for all } k.
\]
Because $\beta_{\ell}$ is the only decreasing bridge, these functions $\varphi_j$ are building blocks.

Our first goal is to identify a suitable collection of functions in $\Sigma$.  This collection will consist of the functions $\varphi_j$ for $j \notin \{h,h+1\}$, plus three more functions that we define in Proposition~\ref{Prop:BridgeFns}.  These three functions are very similar to those constructed in Example~\ref{Ex:Interval}. We find it helpful to illustrate this with a picture, which provides a ``zoomed out" view in which the chain of loops looks like an interval.

Figure~\ref{Fig:BridgeDependence} depicts the essential properties of the three functions appearing in Proposition~\ref{Prop:BridgeFns}.  We present one copy of the chain of loops, depicted as an interval, with markings at the switching loops and bridges.  These markings break the interval into regions, and the regions are labeled with the relevant slope indices.  For example, if a region is labeled with the value $h$ in the copy of the interval corresponding to a function $\varphi$, this indicates that $\varphi$ has slope $s_k (\varphi) = s_k [h]$ for all $k$ in the given region.  We include similar schematic illustrations in all subsequent cases.

\begin{figure}[H]
\begin{tikzpicture}

\draw (-0.2,1) node {{\tiny $\varphi_C$}};
\draw (0,1)--(6,1);
\draw [ball color=white] (3,1) circle (0.55mm);
\draw (1.5,1.2) node {{\tiny $h+1$}};
\draw (4.5,1.2) node {{\tiny $h$}};

\draw (-0.2,3) node {{\tiny $\varphi_A$}};
\draw (0,3)--(6,3);
\draw [ball color=white] (3,3) circle (0.55mm);
\draw [ball color=black] (4.5,3) circle (0.55mm);
\draw (1.5,3.2) node {{\tiny $h$}};
\draw (3.75,3.2) node {{\tiny $h+1$}};
\draw (5.25,3.2) node {{\tiny $h$}};

\draw (-0.2,2) node {{\tiny $\varphi_B$}};
\draw (0,2)--(6,2);
\draw [ball color=white] (3,2) circle (0.55mm);
\draw [ball color=black] (1.5,2) circle (0.55mm);
\draw (0.75,2.2) node {{\tiny $h+1$}};
\draw (2.25,2.2) node {{\tiny $h$}};
\draw (4.5,2.2) node {{\tiny $h+1$}};

\draw (0,-.5)--(6,-.5);
\draw [ball color=white] (3,-.5) circle (0.55mm);
\draw [ball color=black] (1.5,-.5) circle (0.55mm);
\draw [ball color=black] (4.5,-.5) circle (0.55mm);
\draw (0.75,-0.3) node {{\tiny $BC$}};
\draw (3,-0.3) node {{\tiny $AB$}};
\draw (5.25,-0.3) node {{\tiny $AC$}};

\end{tikzpicture}
\caption{A schematic depiction of the three functions $\varphi_A$, $\varphi_B$, $\varphi_C$, and the dependence between them.  The white dot represents the point $x$, and the black dots are located at undetermined points on either side of $x$.}
\label{Fig:BridgeDependence}
\end{figure}
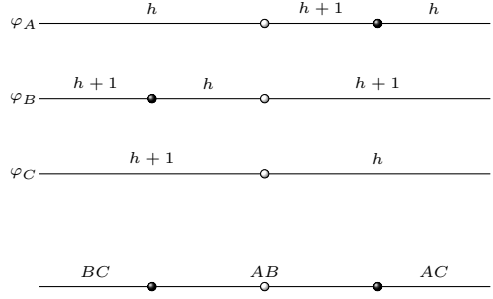

\begin{proposition}
\label{Prop:BridgeFns}
There are functions $\varphi_A , \varphi_B,$ and $\varphi_C$ in $\Sigma$ with the following properties:
\begin{enumerate}
\item  $s_x (\varphi_A) = s_x [h]$, and $s'_k (\varphi_A) = s'_k [h]$ for all $k \leq \ell$.
\item  $s'_x (\varphi_B) = s'_x [h+1]$, and $s_k (\varphi_B) = s_k[h+1]$ for all $k > \ell$.
\item  $s'_k (\varphi_C) = s'_k [h+1]$ for all $k \leq \ell$, and $s_k (\varphi_C) = s_k [h]$ for all $k > \ell$.
\end{enumerate}
Moreover, these functions can be chosen such that $s_k (\varphi) \in \{ s_k [h], s_k [h+1]\}$ for all $k$.
\end{proposition}

\begin{proof}
By Lemma~\ref{Lemma:Existence}, there is a function $\varphi_A \in \Sigma$ such that $s'_0 (\varphi_A) \leq s'_0 [h]$, and $s_{g+1} (\varphi_A) \geq s_{g+1} [h]$.  Since $\beta_{\ell}$ is the only switching bridge, and there are no switching loops, we see that $s'_k (\varphi_A) \leq s'_k [h]$ for all $k\leq \ell$, and $s_k (\varphi_A) \geq s_k [h]$ for all $k > \ell$.  In particular, $s_{\ell+1} (\varphi_A) \geq s_{\ell+1} [h]$, which forces $s_x (\varphi_A) \geq s_x [h]$.  Again, since $\beta_{\ell}$ is the only switching bridge, this forces $s'_k (\varphi_A) \geq s'_k [h]$ for all $k \leq \ell$, and $s_k (\varphi_A) \leq s_k [h+1]$ for all $k > \ell$.

Similarly, by Lemma~\ref{Lemma:Existence}, there is a function $\varphi_B \in \Sigma$ such that $s'_0 (\varphi_B) \leq s'_0 [h+1]$, and $s_{g+1} (\varphi_B) \geq s_{g+1} [h+1]$.  The facts about the slopes of $\varphi_B$ on the various bridges follow from an analysis analogous to the one above.

Suppose that $f_A , f_B \in V$ are functions tropicalizing to $\varphi_A , \varphi_B$, respectively.  If $f$ is a function in the pencil spanned by $f_A$ and $f_B$, then we see that $s'_0 (\trop (f)) \leq s'_0 [h+1]$ and $s_{g+1} (\trop (f)) \geq s_{g+1} [h]$.  It follows from the analysis above that $s_k (\trop (f))$ is equal to either $s_k [h]$ or $s_k [h+1]$ for any given $k$.  Moreover, for any $k$ there is a function $f$ in this pencil such that $s_k (\trop (f)) = s_k [h]$, and there is a function $f$ in this pencil such that $s_k (\trop (f)) = s_k [h+1]$.  Let $\varphi$ be the tropicalization of a function in this pencil such that $s_x (\varphi) = s_x [h+1]$.  Notice that this forces $s'_k (\varphi) = s'_k [h+1]$ for all $k \leq \ell$.  Similarly, let $\varphi'$ be the tropicalization of a function in this pencil such that $s'_x (\varphi') = s'_x [h]$.  Notice that this forces $s_k (\varphi') = s_k [h]$ for all $k > \ell$.  Finally, by adding a scalar to $\varphi'$, we may assume that $\varphi$ and $\varphi'$ agree at $x$, and let $\varphi_C = \min \{ \varphi, \varphi' \}$.
\end{proof}

We now enumerate the building blocks that we will use.  Our set $\cA$ will consist of the functions $\varphi_j$ for $j \notin \{h,h+1\}$, plus three more functions, defined as follows.

\begin{definition}
There are building blocks $\varphi^0_h, \varphi^0_{h+1},$ and $\varphi^{\infty}_h$ with the following slopes:
\begin{enumerate}
\item  $s_k (\varphi^0_h) = s_k [h]$ for all $k$;
\item  $s_k (\varphi^0_{h+1}) = s'_{k-1} [h+1]$ for all $k$;
\item  $s_k (\varphi^{\infty}_h) = s_k[h]$ for all $k < \ell$, and $s_k (\varphi^{\infty}_h) = s_k [h+1]$ for all $k \geq \ell$.
\end{enumerate}
\end{definition}

\noindent These building blocks are unique up to additive constants.  Note that the function $\varphi^0_h$ cannot be an element of $\Sigma$, because $s'_{\ell} (\varphi^0_h) \notin s'_{\ell} (\Sigma)$.  Similarly, the function $\varphi^0_{h+1}$ cannot be an element of $\Sigma$ either.  However, the functions $\varphi_A$, $\varphi_B$, and $\varphi_C$ can be written as tropical linear combinations of these building blocks.  Specifically, the function $\varphi_A \in \Sigma$ is a tropical linear combination of the functions $\varphi^0_h$ and $\varphi^{\infty}_h$, where the two functions simultaneously achieve the minimum at a point of distance $t$, measured along the bridges and bottom edges, to the right of $x$.  Similarly, the function $\varphi_B$ is a tropical linear combination of the functions $\varphi^0_{h+1}$ and $\varphi^{\infty}_h$, where the two functions simultaneously achieve the minimum at a point of distance $t'$ to the left of $x$.  The function $\varphi_C$ is a tropical linear combination of the functions $\varphi^0_h$ and $\varphi^0_{h+1}$, where the two functions simultaneously achieve the minimum at $x$.

In particular, the function $\varphi_A$ is determined by the parameter $t$, and the function $\varphi_B$ is determined by the parameter $t'$.  As in Example~\ref{Ex:Interval}, each of these parameters controls the other.

\begin{proposition}
\label{Prop:BridgeDependence}
The distance $t'$ is determined by the distance $t$.
\end{proposition}

\begin{proof}
The three functions $\varphi_A , \varphi_B, \varphi_C$ are tropicalizations of functions in a pencil, and are therefore tropically dependent.  If we consider the point of distance $t$ to the right of $x$, measured along the bridges and bottom edges, at which the function $\varphi_A$ equals $\varphi^{\infty}_h$ to the left and with $\varphi^0_h$ to the right, we see that locally in a neighborhood of this point, $\varphi_B$ agrees with $\varphi^{\infty}_h$ and $\varphi_C$ agrees with $\varphi^0_h$.  Thus, in the tropical dependence between these three functions, all three must achieve the minimum at this point.  This determines the other point, to the left of $x$, where all three achieve the minimum, which by the same reasoning is of distance $t'$ from $x$.
\end{proof}

Thus, these functions depend on the single parameter $t$.  In the proof of Theorem~\ref{Thm:SwitchingBridge}, we will use the following estimate.

\begin{corollary}
\label{Cor:BoundOnT}
If the subinterval of $\beta_{\ell}$ where $\varphi_A$ has slope $s'_{\ell}[h]$ has length less than $m_{\ell}$, then $t'=t$.
\end{corollary}

\begin{proof}
If the point of distance $t$ to the right of $x$ is not contained in the bridge $\beta_{\ell}$, then $\varphi_A$ has slope $s'_{\ell}[h]$ on the entire bridge $\beta_{\ell}$.  The assumption therefore implies that this point is contained in the bridge $\beta_{\ell}$.  The point of distance $t'$ to the left of $x$ is contained either in the bridge $\beta_{\ell}$ or the bottom edge of the loop $\gamma_{\ell}$.  We consider the case where this point is contained in the bridge first.  Examining the tropical dependence constructed in Proposition~\ref{Prop:BridgeDependence}, we see that
\[
(s'_{\ell}[h+1]- s'_{\ell} [h])t' = (s_{\ell+1} [h+1]- s_{\ell+1} [h])t.
\]
But
\[
s'_{\ell}[h+1]- s'_{\ell} [h] = s_{\ell+1} [h+1]- s_{\ell+1} [h] = 1,
\]
and the result follows.

In the case where the point of distance $t'$ to the left of $x$ is contained in the bottom edge of the loop $\gamma_{\ell}$, we note that, because $\mu_{\ell} = 0$, $\varphi^0_{h+1}$ has slope one greater than $\varphi^0_h$ along this edge.  The result then follows by the same argument as the previous case.
\end{proof}

We now begin to verify the hypotheses of Theorem~\ref{Thm:ExtremalConfig}.

\begin{lemma}
\label{Lem:BridgeAst}
The set
\[
\cA = \{ \varphi_i \, \vert \, i \neq h,h+1 \} \cup \{ \varphi^0_h , \varphi^0_{h+1} , \varphi^{\infty}_h \}
\]
satisfies property $(\ast)$.
\end{lemma}

\begin{proof}
Note that the slope index sequence of $\varphi^0_h$ is the constant sequence $h$ and the slope index sequence of $\varphi^0_{h+1}$ is the constant sequence $h+1$.  We then have $\tau_k(\varphi^{\infty}_h) = \tau'_k(\varphi^\infty_h) = h$ for $k \leq \ell$, and $\tau_k(\varphi^{\infty}_h) = \tau'_k(\varphi^\infty_h) = h+1$ for $k > \ell$.  Thus, if two different functions $\varphi, \varphi' \in \cA$ satisfy $\tau'_k (\varphi) = \tau'_k (\varphi')$, then $\varphi = \varphi^{\infty}_h$.  Morevoer, either $\varphi' = \varphi^0_h$ and $k \leq \ell$, or $\varphi' = \varphi^0_{h+1}$ and $k > \ell$.  In either case, we see that $\varphi$ agrees with $\varphi'$ on $\gamma_k$, so $\cA$ satisfies property $(\ast)$.
\end{proof}

Unlike the previous cases, here we sometimes choose $\cB$ to be a proper subset of the set of all pairwise sums of functions in $\cA$.  We do this in order to satisfy property $(\dagger\dagger)$.  If
\[
s_{\ell}[h] < s'_{\ell}[h],
\]
and if there are two functions $\varphi, \varphi' \in \cA$ such that
\[
s'_{\ell}(\theta) = s_{\ell+1} [h] + s_{\ell+1} (\varphi) = s_{\ell+1} [h] + s_{\ell+1} (\varphi') +1,
\]
then one of the two functions $\varphi^0_h + \varphi$ or $\varphi^{\infty}_h + \varphi'$ will be excluded from the set $\cB$.  Our choice of which function to exclude depends on the length of the subinterval of $\beta_{\ell}$ on which $\varphi_A$ has slope $s'_{\ell}[h]$.  Notice that, if this subinterval is strictly contained in $\beta_{\ell}$, then its length is $t$ plus the distance from $w_{\ell}$ to $x$.  If this length is less than $m_{\ell}$, we choose the set $\cB$ to be the set of all pairwise sums of functions in $\cA$ other than $\varphi^{\infty}_h + \varphi'$.  If this length is greater than or equal to $m_{\ell}$, we choose the set $\cB$ to be the set of all pairwise sums of functions in $\cA$ other than $\varphi^0_h + \varphi$.  Note that, because the functions in $\cA$ have distinct slopes along $\beta_{\ell}$, there is at most one sum of two functions in $\cA$ that is not contained in $\cB$.

\begin{lemma}
\label{Lem:BBOnBridge}
In this subcase, the set $\cB$ satisfies properties $(\ast\ast)$ and $(\dagger\dagger)$.
\end{lemma}

\begin{proof}
By construction, if two building blocks agree on $\gamma_k$ and one has higher slope on $\beta_k$, then $k=\ell$, and the two building blocks are $\varphi^0_h$ and $\varphi^{\infty}_h$.  But $\varphi^{\infty}_h$ agrees with $\varphi^0_h$ on all of $\Gamma_{[1,\ell]}$, so $\cB$ satisfies property $(\ast\ast)$.

Now, suppose that $\varphi^{\infty}_h + \varphi$ is departing on $\gamma_{\ell}$ for some $\varphi \in \cA$.  Then either
\[
s_{\ell+1} (\varphi^{\infty}_h + \varphi) = s_{\ell+1}[h] + s_{\ell+1}(\varphi) + 1 = s'_{\ell} (\theta)+1,
\]
or $\varphi = \varphi^{\infty}_h$, and
\[
s_{\ell+1} (\varphi^{\infty}_h + \varphi) = s_{\ell+1}[h] + s_{\ell+1}(\varphi) + 1 = s'_{\ell} (\theta)+2.
\]
By replacing $\varphi$ with $\varphi^0_h$ in the latter case, we see that, in either case, there exists $\varphi \in \cA$ with
\[
s'_{\ell}(\theta) = s_{\ell+1} [h] + s_{\ell+1} (\varphi).
\]

We now show that, if
\[
s_{\ell+1} (\varphi^{\infty}_h) \leq s_{\ell} (\varphi^{\infty}_h),
\]
then no function is shiny on $\gamma_{\ell}$.  This is because any shiny function on $\gamma_{\ell}$ must be of the form $\varphi + \varphi'$, where $s_{\ell+1} (\varphi') = s_{\ell+1} (\varphi^0_h)$.  Since $\beta_{\ell}$ only switches slope $h$, by property $(\ast)$ this implies that $\varphi'$ agrees with $\varphi^0_h$ on $\gamma_{\ell}$.  But, since $\mu (\gamma_{\ell}) = 0$ and
\[
s_{\ell+1} (\varphi^0_h) < s_{\ell} (\varphi^0_h),
\]
we see that the restriction of $D+\ddiv(\varphi^0_h)$ to $\gamma_{\ell} \smallsetminus \{ v_{\ell} \}$ has degree 2.  Thus, any function of this form cannot be shiny.

We may therefore consider the case where
\[
s_{\ell+1} (\varphi^{\infty}_h) = s_{\ell} (\varphi^{\infty}_h)+1.
\]
This implies that any shiny function on $\gamma_{\ell}$ must be of the form $\varphi^{\infty}_h + \varphi'$.  By construction, either this function or $\varphi^0_h + \varphi$ is omitted from $\cB$.  Thus, $\cB$ satisfies property $(\dagger\dagger)$.
\end{proof}

By Lemma~\ref{Lem:BBOnBridge}, the set $\cB$ satisfies the hypotheses of Theorem~\ref{Thm:ExtremalConfig}, so there is a tropical linear combination $\theta$ such that each function $\psi \in \cB$ achieves the minimum at some point of the loop or bridge to which it is assigned, and any other function that achieves the minimum at this point agrees with $\psi$ on this loop or bridge.  We now describe how to construct a maximally independent combination $\vartheta$ of 28 pairwise sums of functions in $\Sigma$ from the master template $\theta$.  We start by using the functions in $\cB$ that are already pairwise sums of functions in $\Sigma$.  Then, for each $j$, we fit a subset of the functions $\{ \varphi_A + \varphi_j, \varphi_B + \varphi_j, \varphi_C + \varphi_j \}$ to the master template, using Lemma~\ref{Lem:Replace}.  Note that these are tropicalizations of functions in a pencil, so not all three can appear among the 28 in our maximally independent combination, and in fact we will use exactly two.  Our choice of which one to exclude depends on the combinatorial type of the master template $\theta$, and hence on the parameter $t$.

\medskip

\textbf{Step 1:  if $i,j \notin \{ h,h+1\}$.}  The function $\varphi_{ij}$ is in $\cB$, and we set its coefficient in the tropical linear combination $\vartheta$ to be equal to its coefficient in the master template $\theta$.

\medskip

We now describe how to replace the functions $\varphi^0_h + \varphi_j, \varphi^0_{h+1,j}$, and $\varphi^{\infty}_h + \varphi_j$ when $j \notin \{ h,h+1\}$.

\medskip

\textbf{Step 2:  if both $\varphi^0_h + \varphi_j$ and $\varphi^{\infty}_h + \varphi_j$ are contained in $\cB$.}  By Lemma~\ref{Lem:Replace}, we may set the coefficient of $\varphi_C + \varphi_j$ so that it dominates the master template $\theta$, and equals $\theta$ where either $\varphi^0_h + \varphi_j$ or $\varphi^0_{h+1,j}$ achieves the minimum.  We note that it is possible that $\varphi_C + \varphi_j$ equals \emph{both}, but in our construction of the master template $\theta$, if we perturb the coefficients of all functions in $\cB$ that are assigned to the same loop or bridge by some small value $\epsilon$, this does not change the conclusion of Theorem~\ref{Thm:ExtremalConfig}.  We may therefore assume that, if $\varphi_C + \varphi_j$ achieves the minimum on the loop or bridge where $\varphi^0_h + \varphi_j$ is assigned, then it does not achieve the minimum on the loop or bridge where $\varphi^0_{h+1,j}$ is assigned, and vice-versa.  If $\varphi_C + \varphi_j$ achieves the minimum where $\varphi^0_h + \varphi_j$ is assigned, then using Lemma~\ref{Lem:Replace}, we set the coefficient of $\varphi_B + \varphi_j$ so that it dominates the master template $\theta$, and equals $\theta$ where either $\varphi^0_{h+1,j}$ or $\varphi^{\infty}_h + \varphi_j$ achieves the minimum.  In this case, we exclude $\varphi_A + \varphi_j$ from the tropical linear combination $\vartheta$.  Similarly, if $\varphi_C + \varphi_j$ achieves the minimum where $\varphi^0_{h+1,j}$ is assigned, then using Lemma~\ref{Lem:Replace}, we set the coefficient of $\varphi_A + \varphi_j$ so that it dominates the master template $\theta$ and equals $\theta$ where either $\varphi^0_h + \varphi_j$ or $\varphi^{\infty}_h + \varphi_j$ achieves the minimum.  And in this case, we exclude $\varphi_B + \varphi_j$

\medskip

We now turn to the cases where there is a sum of two functions in $\cA$ that is not contained in $\cB$.  The basic idea of these cases is as follows.  If the subinterval of $\beta_{\ell}$ on which $\varphi_A$ has slope $s'_{\ell}[h]$ has length less than $m_{\ell}$, then we are in some sense ``very close'' to the case where there is no switching bridge, and we set the coefficients as though there is no switching bridge.  If the length of this subinterval is greater than or equal to $m_{\ell}$, then we set the coefficients of the relevant functions so that they achieve the minimum to the left of $\beta_{\ell}$, and use the fact that they have larger slope than $\theta$ on a large subinterval of $\beta_{\ell}$ to show that they cannot achieve the minimum to the right of $\beta_{\ell}$.

\medskip

\textbf{Step 3:  if $j \notin \{ h, h+1\}$ and $\varphi^{\infty}_h + \varphi_j \notin \cB$.}  We set the coefficient of $\varphi_C + \varphi_j$ as in Step 2.  We can do this because both $\varphi^0_h + \varphi_j$ and $\varphi^0_{h+1,j}$ are contained in $\cB$.  If $\varphi_C + \varphi_j$ achieves the minimum on the loop or bridge to which $\varphi^0_{h+1,j}$ was assigned, we then set the coefficient of $\varphi_A + \varphi_j$ so that it equals $\varphi^0_h + \varphi_j$ on the loop or bridge where it is assigned.  Our assumption that $\varphi_A$ has slope $s'_{\ell}[h]$ on a small subinterval of $\beta_{\ell}$ implies that the two functions do in fact agree on this loop or bridge.  Indeed, this loop or bridge is to the right of $\gamma_{\ell}$, because
\[
s_{\ell+1} (\varphi^0_h + \varphi_j) < s'_{\ell} (\theta),
\]
and because $\rho \leq 2$,
\[
s_k (\varphi^0_h + \varphi_j) \leq s_{\ell} (\varphi^0_h + \varphi_j) + 1 \mbox{ for all } k \leq \ell.
\]

Similarly, if $\varphi_C + \varphi_j$ achieves the minimum on the loop or bridge to which $\varphi^0_h + \varphi_j$ was assigned, we then set the coefficient of $\varphi_B + \varphi_j$ so that it equals $\varphi^0_{h+1,j}$ on the loop or bridge where this function is assigned.  The slope of this function implies that this bridge or loop is to the left of $\gamma_{\ell}$, and the fact that $t=t'$ implies that $\varphi_B$ equals $\varphi^0_{h+1}$ on this loop or bridge.

\medskip

\textbf{Step 4:  if $\varphi^0_h + \varphi_j \notin \cB$.}  In this case, because
\[
s_{\ell} (\varphi^0_{h+1,j}) > s_{\ell} (\theta),
\]
we see that $\varphi^0_{h+1,j}$ is assigned to a loop $\gamma_k$ with $k<\ell$.  We set the coefficient of $\varphi_C + \varphi_j$ so that it equals $\varphi^0_{h+1,j}$ on the loop where this function is assigned.  Note that, because
\[
s_{\ell+1} (\varphi^{\infty}_h + \varphi_j) > s'_{\ell} (\theta),
\]
we see that $\varphi^{\infty}_h + \varphi_j$ is assigned to a loop $\gamma_{k'}$ with $k < k' \leq \ell$.  We then set the coefficient of $\varphi_A + \varphi_j$ so that it equals $\varphi^{\infty}_h + \varphi_j$ on the loop where it is assigned.

\medskip

\textbf{Step 5:  if $j = h$ or $h+1$.}  A similar construction can be applied to the case where $j=h$ or $h+1$, and we provide only a brief sketch.  We first use Lemma~\ref{Lem:Replace} to set the coefficient of $\varphi_C + \varphi_C$.  Depending on where this achieves the minimum, we then set the coefficient of either $\varphi_A + \varphi_C$ or $\varphi_B + \varphi_C$.  Finally, depending on where this second function achieves the minimum, we then choose one of $\varphi_A + \varphi_A$, $\varphi_A + \varphi_B$, or $\varphi_B + \varphi_B$, and use Lemma~\ref{Lem:Replace} to set its coefficient.

\begin{theorem}
\label{Thm:SwitchingBridge}
In this case, $\vartheta$ is a maximally independent combination of 28 functions.
\end{theorem}

\begin{proof}
To see that there are exactly 28 functions in the tropical linear combination $\vartheta$, we note that $\vartheta$ involves exactly as many functions as there are pairs $(i,j)$, $0 \leq i \leq j \leq 6$.  Specifically, for $i,j \notin \{ h,h+1\}$, we have the function $\varphi_{ij}$.  For $j \notin \{ h,h+1\}$, the two pairs $(h,j), (h+1,j)$ correspond to two functions, one of which is $\varphi_C + \varphi_j$, and the other of which is either $\varphi_A + \varphi_j$ or $\varphi_B + \varphi_j$.  Similarly, the three pairs $(h,h), (h,h+1), (h+1,h+1)$ correspond to three functions.

We first consider the case where the set $\cB$ consists of all pairwise sums of elements of $\cA$.  By construction, each of these 28 functions achieves the minimum on a region where one of the functions in $\cB$ achieves the minimum in the master template.  By Lemma~\ref{Lem:OneEq}, if two functions $\psi, \psi' \in \cB$ are assigned to the same loop $\gamma_k$ or bridge $\beta_k$, then $\psi = \varphi^{\infty}_h + \varphi$ for some $\varphi \in \cA$, and either
\[
\psi' = \varphi^0_h + \varphi, k \leq \ell, \mbox{ or } \psi' = \varphi^0_{h+1} + \varphi, k > \ell .
\]
Assume for simplicity that $\varphi = \varphi_j$ for some $j$.  The other cases are similar.  The function $\varphi_C + \varphi_j$ achieves the minimum at a point $v$ on the bridge or loop to which either $\varphi^0_h + \varphi_j$ or $\varphi^0_{h+1,j}$ is assigned, but not both.  Note that $\varphi_C$ does not agree with $\varphi^{\infty}_h$ on any loop, so $\varphi_C + \varphi_j$ does not achieve the minimum on the loop or bridge to which $\varphi^{\infty}_h + \varphi_j$ is assigned.  Note that $\varphi_A + \varphi_j$ can achieve the minimum at $v$ only if $v$ is contained in the bridge or loop to which $\varphi^0_h + \varphi_j$ is assigned.  But it is precisely in this case that we use the function $\varphi_B + \varphi_j$ in the tropical linear combination $\vartheta$, rather than $\varphi_A + \varphi_j$.  Thus, $\varphi_C + \varphi_j$ is the only function to achieve the minimum at $v$.  If $v$ is contained in the bridge or loop to which $\varphi^0_h + \varphi_j$ is assigned, then $\varphi_C + \varphi_j$ does not achieve the minimum on the loop or bridge to which either $\varphi^0_{h+1,j}$ or $\varphi^{\infty}_h + \varphi_j$ is assigned.  It follows that there is a point where $\varphi_B + \varphi_j$ is the only function to achieve the minimum.  Otherwise, there is a point where $\varphi_A + \varphi_j$ is the only function to achieve the minimum.

We now turn to the cases where some function is omitted from the set $\cB$.  First, suppose that the length of the subinterval of $\beta_{\ell}$ on which $\varphi_A$ has slope $s'_{\ell}[h]$ is less than $m_{\ell}$.  We will consider the case where the tropical linear combination $\vartheta$ involves the function $\varphi_A + \varphi'$.  The case where it involves $\varphi_B + \varphi'$ is similar.  For any point $v$ to the left of $\beta_{\ell}$, we have
\[
\varphi_A (v) \leq \varphi^0_h (v),
\]
so we must check that $\varphi_A + \varphi'$ does not achieve the minimum at any such point.  By assumption, however, we have
\[
\varphi_A (v) \geq \varphi^0_h (v) - m_{\ell}.
\]
Since
\[
s_{\ell+1} (\varphi^0_h + \varphi') < s'_{\ell} (\theta),
\]
we see that
\[
\varphi^0_h (v) + \varphi' (v) \gg \theta(v),
\]
so $\varphi_A + \varphi'$ cannot achieve the minimum at $v$.

Now, suppose that the length of the subinterval of $\beta_{\ell}$ on which $\varphi_A$ has slope $s'_{\ell}[h]$ is greater than or equal to $m_{\ell}$.  We must show that the functions $\varphi_A + \varphi$ and $\varphi_C + \varphi$ dominate the master template.  In particular, there are points $v$ to the right of $\beta_{\ell}$ where $\varphi_A (v) < \varphi^{\infty}_h (v)$, and similarly where $\varphi_C (v) < \varphi^0_{h+1} (v)$.  Because $\varphi_A + \varphi$ has slope greater than $s'_{\ell} (\theta)$ on a subinterval of $\beta_{\ell}$ of length greater than $m_{\ell}$, however, we see that it cannot achieve the minimum to the right of $\gamma_{\ell}$.  Similarly, because $\varphi_C + \varphi$ has slope greater than $s'_{\ell-1}(\theta)$ on $\beta_{\ell-1}$, we see that it cannot achieve the minimum to the right of $\gamma_{\ell-1}$.
\end{proof}

\subsection{Case 3: one switching loop}
\label{Sec:Switch}
In this section, we consider the case where there is only one switching loop $\gamma_{\ell}$, which switches slope $h$, leaving the case of two switching loops for the final section.

As in the previous section, our first goal is to identify a suitable collection of functions in $\Sigma$.  By Lemma~\ref{Lem:GenericFns}, for all $j \notin \{ h,h+1\}$, there is a function $\varphi_j \in \Sigma$ with
\[
s_k (\varphi_j) = s_k [j] \mbox{ and } s'_k (\varphi_j) = s'_k [j] \mbox{ for all } k.
\]
Unlike in the previous section, in this case there may also be a decreasing bridge, so the functions $\varphi_j$ are not necessarily building blocks.  Since there is at most 1 decreasing bridge, and at most 1 slope decreases on it, at most 1 of the functions $\varphi_j$ is not a building block.

The functions we wish to consider in $\Sigma$ will be the functions $\varphi_j$ for $j \notin \{ h,h+1\}$, plus three more functions that we define in Proposition~\ref{Prop:BridgeFns}.  These three functions are very similar to those constructed in the previous section, and the proof of Proposition~\ref{Prop:LoopFns} is virtually identical to that of Proposition~\ref{Prop:BridgeFns}.  We again recommend the reader to compare these functions to those constructed in Example~\ref{Ex:Interval}, keeping in mind that the chain of loops looks like an interval when viewed from far away.  These functions are illustrated in Figure~\ref{Fig:Dependence}.

\begin{figure}[H]
\begin{tikzpicture}

\draw (-0.2,1) node {{\tiny $\varphi_C$}};
\draw (0,1)--(6,1);
\draw [ball color=white] (3,1) circle (0.55mm);
\draw (1.5,1.2) node {{\tiny $h+1$}};
\draw (4.5,1.2) node {{\tiny $h$}};

\draw (-0.2,3) node {{\tiny $\varphi_A$}};
\draw (0,3)--(6,3);
\draw [ball color=white] (3,3) circle (0.55mm);
\draw [ball color=black] (4.5,3) circle (0.55mm);
\draw (1.5,3.2) node {{\tiny $h$}};
\draw (3.75,3.2) node {{\tiny $h+1$}};
\draw (5.25,3.2) node {{\tiny $h$}};

\draw (-0.2,2) node {{\tiny $\varphi_B$}};
\draw (0,2)--(6,2);
\draw [ball color=white] (3,2) circle (0.55mm);
\draw [ball color=black] (1.5,2) circle (0.55mm);
\draw (0.75,2.2) node {{\tiny $h+1$}};
\draw (2.25,2.2) node {{\tiny $h$}};
\draw (4.5,2.2) node {{\tiny $h+1$}};

\draw (0,-.5)--(6,-.5);
\draw [ball color=white] (3,-.5) circle (0.55mm);
\draw [ball color=black] (1.5,-.5) circle (0.55mm);
\draw [ball color=black] (4.5,-.5) circle (0.55mm);
\draw (0.75,-0.3) node {{\tiny $BC$}};
\draw (3,-0.3) node {{\tiny $AB$}};
\draw (5.25,-0.3) node {{\tiny $AC$}};

\end{tikzpicture}
\caption{A schematic depiction of the three functions $\varphi_A$, $\varphi_B$, $\varphi_C$, and the dependence between them.  The white dot represents the switching loop $\gamma_{\ell}$, and the black dots are located at undetermined points on either side of $\gamma_{\ell}$.}
\label{Fig:Dependence}
\end{figure}
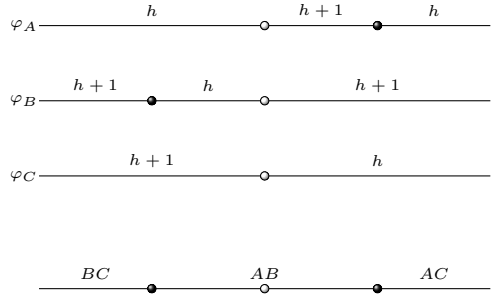

\begin{proposition}
\label{Prop:LoopFns}
There exist functions $\varphi_A , \varphi_B,$ and $\varphi_C \in \Sigma$ with the following properties:
\begin{enumerate}
\item  $s_k (\varphi_A) = s_k [h]$ for all $k\leq\ell$.
\item  $s_k (\varphi_B) = s_k [h+1]$ for all $k > \ell$.
\item  $s_k (\varphi_C) = s_k [h+1]$ for all $k\leq\ell$, and $s_k (\varphi_C) = s_k [h]$ for all $k > \ell$.
\end{enumerate}
Moreover, these functions can be chosen such that $s_k (\varphi) \in \{ s_k [h], s_k [h+1]\}$, for all $k$.
\end{proposition}

\begin{proof}
The argument is identical to the proof of Proposition~\ref{Prop:BridgeFns}.
\end{proof}

Also, as in the previous section, the functions $\varphi_A, \varphi_B$, and $\varphi_C$ can be written as tropical linear combinations of simpler functions in $R(D)$, characterized as follows.

\begin{definition}
We define functions $\varphi^0_h, \varphi^0_{h+1},$ and $\varphi^{\infty}_h$ with the following slopes:
\begin{enumerate}
\item  $s_k (\varphi^0_h) = s_k [h]$ and $s'_k (\varphi^0_h) = s'_k[h]$ for all $k$;
\item  $s_k (\varphi^0_{h+1}) = s_k [h+1]$ and $s'_k (\varphi^0_{h+1}) = s'_k [h+1]$ for all $k$;
\item  $s_k (\varphi^{\infty}_h) = s_k [h]$, $s'_{k-1} (\varphi^{\infty}_h) = s'_{k-1} [h]$ for all $k\leq\ell$, and $s_k (\varphi^{\infty}_h) = s_k [h+1]$, $s'_{k-1} (\varphi^{\infty}_h) = s'_{k-1} [h+1]$ for all $k > \ell$.
\end{enumerate}
\end{definition}

\noindent Unlike the previous section, in this case there may be a decreasing bridge, so the functions $\varphi^0_h, \varphi^0_{h+1}$, and $\varphi^{\infty}_h$ are not necessarily building blocks.  Since there is at most 1 decreasing bridge, and at most 1 slope decreases on it, at most 2 of these 3 functions are not building blocks.

Just as in the previous section, the functions $\varphi_A, \varphi_B,$ and $\varphi_C$ are tropical linear combinations of the functions $\varphi^0_h, \varphi^0_{h+1},$ and $\varphi^{\infty}_h$.  Specifically, the function $\varphi_A \in \Sigma$ is a tropical linear combination of the functions $\varphi^0_h$ and $\varphi^{\infty}_h$, where the two functions simultaneously achieve the minimum at a point of distance $t$ to the right of $\gamma_{\ell}$.  Similarly, the function $\varphi_B$ is a tropical linear combination of the functions $\varphi^0_{h+1}$ and $\varphi^{\infty}_h$, where the two functions simultaneously achieve the minimum at a point of distance $t'$ to the left of $\gamma_{\ell}$.  The function $\varphi_C$ is a tropical linear combination of the functions $\varphi^0_h$ and $\varphi^0_{h+1}$, where the two functions simultaneously achieve the minimum on the loop $\gamma_{\ell}$ where they agree.

The function $\varphi_A$ is determined by the parameter $t$, and the function $\varphi_B$ is determined by the parameter $t'$.  As in the previous section, it may therefore appear that the possibilities for $\Sigma$ form a 2-parameter family.  Just as in Proposition~\ref{Prop:BridgeDependence}, however, the distance $t'$ is determined by the distance $t$.  In the proof of Theorem~\ref{Thm:Switching}, we will use the following estimate.

\begin{corollary}
\label{Cor:SecondBoundOnT}
If $t<m_{\ell}$, then $t'<2m_{\ell}$.
\end{corollary}

\begin{proof}
If $t<m_{\ell}$, then the point of distance $t$ to the right of $\gamma_{\ell}$ is contained in the bridge $\beta_{\ell}$, and the point of distance $t'$ to the left of $\gamma_{\ell}$ is contained in the bridge $\beta_{\ell-1}$.  Examining the tropical dependence constructed in Proposition~\ref{Prop:BridgeDependence}, it follows that
\[
(s_{\ell}[h+1]- s_{\ell} [h])t' = (s_{\ell+1} [h+1]- s_{\ell+1} [h])t'.
\]
Both $1 \leq s_{\ell} [h+1] - s_{\ell} [h]) \leq 2$ and $1 \leq s_{\ell+1} [h+1] - s_{\ell+1} [h]) \leq 2$.  The result follows.
\end{proof}

\medskip

\subsubsection{Subcase 3a:  no decreasing bridges}

For simplicity, we first consider the case where there are no decreasing bridges.  This case is very similar to case 2.  We begin by verifying the hypotheses of Theorem~\ref{Thm:ExtremalConfig}.

\begin{lemma}
\label{Lem:LoopAst}
In this subcase, the set
\[
\cA = \{ \varphi_i \, \vert \, i \neq h,h+1 \} \cup \{ \varphi^0_h , \varphi^0_{h+1} , \varphi^{\infty}_h \}
\]
consists of building blocks, and satisfies property $(\ast)$.
\end{lemma}

\begin{proof}
Because there are no decreasing bridges, we see that each function in $\cA$ has constant slope along the bridges.  It follows that the functions in $\cA$ are building blocks.  The proof that $\cA$ satisfies property $(\ast)$ is identical to the proof of Lemma~\ref{Lem:BridgeAst}.
\end{proof}

We choose the set $\cB$ in a similar way to the previous case, in order to satisfy property $(\dagger\dagger)$.  If there exist two values $j,j'$ such that
\[
s'_{\ell}(\theta) = s'_{\ell} [h] + s'_{\ell} [j] = s'_{\ell} [h] + s'_{\ell} [j'] +1,
\]
then one of the two functions $\varphi^0_h + \varphi_j$ or $\varphi^{\infty}_h + \varphi_{j'}$ will be excluded from the set $\cB$, depending on $t$.  If $t<m_{\ell}$, we exclude $\varphi^{\infty}_{h} + \varphi_{j'}$, and if $t \geq m_{\ell}$, we exclude $\varphi^0_h + \varphi_j$.

\begin{lemma}
\label{Lem:BBOnLoop}
In this subcase, the set $\cB$ satisfies properties $(\ast\ast)$ and $(\dagger\dagger)$.
\end{lemma}

\begin{proof}
The argument that $\cB$ satisfies property $(\ast\ast)$ is identical to the proof of Lemma~\ref{Lem:BBOnBridge}.  Also, as in the proof of Lemma~\ref{Lem:BBOnBridge}, if two functions in $\cB$ agree on $\gamma_k$ and only one of the two is departing, then $k=\ell$, these two functions must be $\varphi^{\infty}_h + \varphi_j$ and $\varphi^0_h + \varphi_j$ for some $j$, and any shiny function on $\gamma_{\ell}$ must be $\varphi^{\infty}_{h} + \varphi_{j'}$, where
\[
s'_{\ell}(\theta) = s'_{\ell} [h] + s'_{\ell} [j] = s'_{\ell} [h] + s'_{\ell} [j'] +1.
\]
By construction, either $\varphi^0_h + \varphi_j$ or $\varphi^{\infty}_{h} + \varphi_{j'}$ is omitted from $\cB$.  Thus, $\cB$ satisfies property $(\dagger\dagger)$.
\end{proof}

Our construction of the ``best approximation'' $\vartheta$ is now identical to the previous case, as is the proof that it is a maximally independent combination of 28 functions in $\Sigma$.

\begin{theorem}
\label{Thm:Switching}
In this subcase, $\vartheta$ is a maximally independent combination of 28 functions.
\end{theorem}

\begin{proof}
The proof is identical to that of Theorem~\ref{Thm:SwitchingBridge}.
\end{proof}

\medskip

\subsubsection{Subcase 3b: a decreasing bridge}

We now consider the case where there is both a switching loop $\gamma_{\ell}$ and also a decreasing bridge $\beta_{\ell'}$.  In this case, since the bridge $\beta_{\ell'}$ has multiplicity 1, there is a unique value $h'$ such that
\[
s'_{\ell'} [h'] = s_{\ell'+1} [h']+1.
\]
The basic idea here is that we combine the constructions from subcases 1b and 3a to ``fix" the decreasing bridge and the switching loop, respectively.

As in subcase 3a, we define
\[
\cA = \{ \varphi_i \, \vert \, i \neq h,h+1 \} \cup \{ \varphi^0_h , \varphi^0_{h+1} , \varphi^{\infty}_h \}.
\]
We then define $\cA'$ to be the set of functions $\varphi \in R(D)$ with the following properties:
\begin{enumerate}
\item  there is a function $\varphi' \in \cA$ such that $\varphi$ agrees with $\varphi'$ on each connected component of $\Gamma \smallsetminus \beta_{\ell'}$, and
\item  $\varphi$ has constant slope on $\beta_{\ell'}$, equal to either $s'_{\ell'} [\tau'_{\ell'} (\varphi')]$ or $s_{\ell'+1} [\tau_{\ell'+1} (\varphi')]$.
\end{enumerate}
The functions in $\cA'$ are building blocks, and by combining the proofs of Lemmas~\ref{Lem:BB2} and~\ref{Lem:LoopAst}, we see that $\cA'$ satisfies property $(\ast)$.
Note that each function in $\cA$ is a tropical linear combination of functions in $\cA'$.  The functions $\varphi_A , \varphi_B$, and $\varphi_C$ are themselves tropical linear combinations of the functions in $\cA$.

We define $\cB$, a set of pairwise sums of functions in $\cA$, in the same way as the previous section.  We define $\cB'$ to be the set of pairwise sums of functions in $\cA'$ with the property that the associated pairwise sum of functions in $\cA$ is an element of $\cB$.  Combining Lemmas~\ref{Lem:BB2} and~\ref{Lem:BBOnLoop}, we see that $\cB'$ satisfies properties $(\ast\ast)$ and $(\dagger\dagger)$.

We can now construct the tropical linear combination $\vartheta$.  First, we use Theorem~\ref{Thm:ExtremalConfig} to construct a master template $\theta$ out of the functions in $\cB'$.  We then follow the construction from subcase 1b to build an intermediary tropical linear combination $\vartheta'$ of the functions in $\cB'$, and the construction from subcase 3a to build the tropical linear combination $\vartheta$, replacing the functions $\varphi^0_h, \varphi^0_{h+1}$, and $\varphi^{\infty}_h$ in $\vartheta'$ with the functions $\varphi_A, \varphi_B$, and $\varphi_C$ in $\vartheta$.

\begin{theorem}
In this subcase, $\vartheta$ is a maximally independent combination.
\end{theorem}

\begin{proof}
By Theorem~\ref{Thm:BendOnBridge}, in $\vartheta'$, every function in $\cB$ achieves the minimum on the loop or bridge to which one of its associated building blocks in $\cB'$ was assigned in the master template $\theta$.  By mild abuse of language, we may therefore refer to the loop or bridge to which a function in $\cB$ is assigned in $\vartheta'$.  Also by Theorem~\ref{Thm:BendOnBridge}, if two functions in $\cB$ are assigned to the same loop $\gamma_k$ or bridge $\beta_k$ in $\vartheta'$, then one of the functions is $\varphi^{\infty}_h + \varphi$ for some $\varphi \in \cA$, and the other is either $\varphi^0_h + \varphi$ or $\varphi^0_{h+1} + \varphi$.  This is the setup required for the proof of Theorem~\ref{Thm:Switching}, which then shows that $\vartheta$ is a maximally independent combination.
\end{proof}

\subsection{Case 4: two switching loops}
In this final section, we consider the case where there are two switching loops, $\gamma_{\ell}$ and $\gamma_{\ell'}$, with $\ell < \ell'$.  We write $h$ for the slope that is switched by $\gamma_{\ell}$, and $h'$ for the slope that is switched by $\gamma_{\ell'}$.  Note that both loops must have multiplicity 1.  By our classification of switching loops, we have
\[
s'_{\ell} [i] = s_{\ell} [i] \mbox{ and } s'_{\ell'} [i] = s_{\ell'} [i] \text{ for all } i.
\]
Moreover,
\[
s_{\ell} [h+1] = s_{\ell} [h]+1 \mbox{ and } s_{\ell'} [h'+1] = s_{\ell'} [h']+1.
\]
Note also that since $\rho=2$ and we have two loops with positive multiplicity, by Theorem~\ref{Thm:BNThm} there are no decreasing loops or bridges, and the ramification weights at $w_0$ and $v_{g+1}$ are both zero.  In other words, $s'_0 [j] = s_{g+1} [j] = j$ for all $j$.  Note that this implies that there are only finitely many building blocks.  In this respect, the final case is somewhat simpler than some of the previous cases.

We break our analysis into several subcases, depending on the relationship between $h$ and $h'$.  By Lemma~\ref{Lem:GenericFns}, for all $j \notin \{ h,h+1,h',h'+1\}$, there is a function $\varphi_j \in \Sigma$ with
\[
s_k (\varphi_j) = s_k [j] \mbox{ and } s'_k (\varphi_j) = s'_k [j] \mbox{ for all } k.
\]
Because there are no decreasing bridges, the functions $\varphi_j$ are building blocks.

\medskip

\subsubsection{Subcase 4a:  $h' \notin \{ h-1, h, h+1\}$}

This is the simplest subcase because, roughly speaking, the two switching loops do not interact with one another.  More precisely, there are functions $\varphi_A, \varphi_B$, and $\varphi_C$ in $\Sigma$ with slopes as defined in Proposition~\ref{Prop:LoopFns}, and similarly, replacing $\ell$ with $\ell'$ and $h$ with $h'$, there are functions $\varphi'_A, \varphi'_B$, and $\varphi'_C$ in $\Sigma$ with such slopes.  We may then define building blocks $\varphi^0_{h'}, \varphi^0_{h'+1}$, and $\varphi^{\infty}_{h'}$ as in subcase 3a, and set
\[
\cA = \{ \varphi_i \, \vert \, i \neq h,h+1,h',h'+1 \} \cup \{ \varphi^0_h , \varphi^0_{h+1} , \varphi^{\infty}_h , \varphi^0_{h'} , \varphi^0_{h'+1} , \varphi^{\infty}_{h'} \}.
\]
Our construction of the set $\cB$ and the maximally independent combination $\vartheta$ now follow the exact same steps as in subcase 3a, treating each switching loop separately.

\begin{theorem}
In this subcase, $\vartheta$ is a maximally independent combination.
\end{theorem}

\begin{proof}
The argument is identical to the proof of Theorem~\ref{Thm:Switching}.
\end{proof}

\medskip

\subsubsection{Subcase 4b:  $h' = h$}

We first identify a suitable collection of functions in $\Sigma$.  These are the functions $\varphi_i$ for $i \not \in \{ h,h+1 \}$, together with the functions illustrated in Figure~\ref{Fig:Case2}.

\begin{figure}[H]
\begin{tikzpicture}

\draw (-0.2,4) node {{\tiny $A$}};
\draw (0,4)--(6,4);
\draw [ball color=white] (2,4) circle (0.55mm);
\draw [ball color=white] (4,4) circle (0.55mm);
\draw [ball color=black] (3,4) circle (0.55mm);
\draw [ball color=black] (5,4) circle (0.55mm);
\draw (1,4.2) node {{\tiny $h$}};
\draw (2.5,4.2) node {{\tiny $h+1$}};
\draw (3.5,4.2) node {{\tiny $h$}};
\draw (4.5,4.2) node {{\tiny $h+1$}};
\draw (5.5,4.2) node {{\tiny $h$}};

\draw (-0.2,3) node {{\tiny $B$}};
\draw (0,3)--(6,3);
\draw [ball color=white] (2,3) circle (0.55mm);
\draw [ball color=white] (4,3) circle (0.55mm);
\draw [ball color=black] (1,3) circle (0.55mm);
\draw [ball color=black] (3,3) circle (0.55mm);
\draw (0.5,3.2) node {{\tiny $h+1$}};
\draw (1.5,3.2) node {{\tiny $h$}};
\draw (2.5,3.2) node {{\tiny $h+1$}};
\draw (3.5,3.2) node {{\tiny $h$}};
\draw (5,3.2) node {{\tiny $h+1$}};

\draw (-0.2,2) node {{\tiny $C$}};
\draw (0,2)--(6,2);
\draw [ball color=white] (2,2) circle (0.55mm);
\draw [ball color=white] (4,2) circle (0.55mm);
\draw [ball color=black] (5,2) circle (0.55mm);
\draw (1,2.2) node {{\tiny $h+1$}};
\draw (3,2.2) node {{\tiny $h$}};
\draw (4.5,2.2) node {{\tiny $h+1$}};
\draw (5.5,2.2) node {{\tiny $h$}};

\draw (-0.2,1) node {{\tiny $D$}};
\draw (0,1)--(6,1);
\draw [ball color=white] (2,1) circle (0.55mm);
\draw [ball color=white] (4,1) circle (0.55mm);
\draw [ball color=black] (1,1) circle (0.55mm);
\draw (0.5,1.2) node {{\tiny $h+1$}};
\draw (1.5,1.2) node {{\tiny $h$}};
\draw (3,1.2) node {{\tiny $h+1$}};
\draw (5,1.2) node {{\tiny $h$}};

\draw (-0.2,0) node {{\tiny $E$}};
\draw (0,0)--(6,0);
\draw [ball color=white] (2,0) circle (0.55mm);
\draw [ball color=white] (4,0) circle (0.55mm);
\draw [ball color=black] (3,0) circle (0.55mm);
\draw (1,0.2) node {{\tiny $h+1$}};
\draw (2.5,0.2) node {{\tiny $h+1$}};
\draw (3.5,0.2) node {{\tiny $h$}};
\draw (5,0.2) node {{\tiny $h$}};

\end{tikzpicture}
\caption{The five functions of Proposition~\ref{Prop:Case2Fns}.  The white points represent the two switching loops $\gamma_{\ell}$ and $\gamma_{\ell'}$.  The black points are at indeterminate locations in the regions where they appear.}
\label{Fig:Case2}
\end{figure}
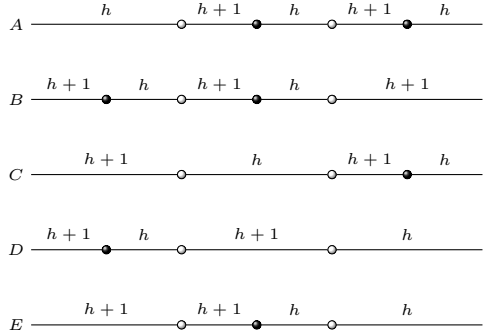

\begin{proposition}
\label{Prop:Case2Fns}
There exist functions $\varphi_A , \varphi_B , \varphi_C , \varphi_D , \varphi_E \in \Sigma$ with the following properties:
\begin{enumerate}
\item  $s_k (\varphi_A) = s_k [h]$ for all $k\leq\ell$.
\item  $s_k (\varphi_B) = s_k [h+1]$ for all $k > \ell'$.
\item  $s_k (\varphi_C) = s_k [h+1]$ for all $k\leq\ell$ and $s_k (\varphi_C) = s_k [h]$ for all $\ell < k \leq \ell'$.
\item  $s_k (\varphi_D) = s_k [h+1]$ for all $\ell < k \leq \ell'$, and $s_k (\varphi_D) = s_k [h]$ for all all $k > \ell'$
\item  $s_k (\varphi_E) = s_k [h+1]$ for all $k\leq\ell$ and $s_k (\varphi_E) = s_k [h]$ for all $k > \ell'$.
\end{enumerate}
Moreover, these functions can be chosen such that $s_k (\varphi) \in \{ s_k [h], s_k [h+1]\}$ for all $k$.
\end{proposition}

\begin{proof}
Let $p_0$ be a point on $X$ specializing to $w_0$ and $p_{g+1}$ a point on $X$ specializing to $v_{g+1}$, and consider a pencil $W$ of functions in $\cL (D_X)$ that vanish to order at least $h$ at $p_{g+1}$ and order at least $r-(h+1)$ at $p_0$.  (Note that this space is at least 2 dimensional, so it contains a pencil.)  We will choose each of the 5 functions to be the tropicalization of a function in this pencil.  Note that, if $\varphi$ is such a function, then $s'_0 (\varphi) \leq s'_0 [h+1]$ and $s_{g+1} (\varphi) \geq s_{g+1} [h]$.  Because $\gamma_{\ell}$ and $\gamma_{\ell'}$ are the only switching loops, this implies that $s_k [h] \leq s_k (\varphi) \leq s_k [h+1]$ for all $k$.  Moreover, since $\varphi$ is the tropicalization of a function in $\cL (D_X)$, $s_k (\varphi)$ must equal $s_k [j]$ for some $j$.  Therefore, for each $k$, $s_k (\varphi)$ is equal to either $s_k [h]$ or $s_k [h+1]$.

Now, since $W$ is a pencil, for each $k$ there are two possible slopes of functions $s_k (\varphi)$ for $\varphi \in \trop (W)$.  It follows that, for each $k$, there is $\varphi \in \trop (W)$ such that $s_k (\varphi) = s_k [h]$, and there is $\varphi' \in \trop (W)$ such that $s_k (\varphi') = s_k [h+1]$.  We let $\varphi_A$ be a function in $\trop (W)$ such that $s'_0 (\varphi_A) = s'_0 (h)$.  Similarly, we let $\varphi_B$ be a function in $\trop (W)$ such that $s_{g+1} (\varphi_B) = s_{g+1} [h+1]$.

To obtain $\varphi_C$, we consider two functions.  There is a function $\varphi \in \trop (W)$ such that $s_{\ell} (\varphi) = s_{\ell} [h+1]$, and there is a function $\varphi' \in \trop (W)$ such that $s'_{\ell} (\varphi') = s'_{\ell} [h]$.  By taking a suitable tropical linear combination of these two functions, we obtain $\varphi_C$.  Similarly, $\varphi_D$ is obtained by taking a suitable combination of a function $\varphi \in \trop (W)$ with $s_{\ell'} (\varphi) = s_{\ell'} [h+1]$ and a function $\varphi' \in \trop (W)$ with $s'_{\ell'} (\varphi') = s'_{\ell'} [h]$.  And $\varphi_E$ is obtained by taking a suitable combination of a function $\varphi \in \trop (W)$ with $s_{\ell} (\varphi) = s_{\ell} [h+1]$ and a function $\varphi' \in \trop (W)$ with $s'_{\ell'} (\varphi') = s'_{\ell'} [h]$.
\end{proof}

We now explore some relations between the functions described in Proposition~\ref{Prop:Case2Fns}.  These relations imply that one of the two functions depicted in Figure~\ref{Fig:TwoOptions} is contained in $\Sigma$.

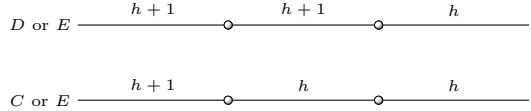
\begin{figure}[H]
\begin{tikzpicture}

\draw (0,4)--(6,4);
\draw (-0.5,4) node {{\tiny $D$ or $E$}};
\draw [ball color=white] (2,4) circle (0.55mm);
\draw [ball color=white] (4,4) circle (0.55mm);
\draw (1,4.2) node {{\tiny $h+1$}};
\draw (3,4.2) node {{\tiny $h+1$}};
\draw (5,4.2) node {{\tiny $h$}};

\draw (0,3)--(6,3);
\draw (-0.5,3) node {{\tiny $C$ or $E$}};
\draw [ball color=white] (2,3) circle (0.55mm);
\draw [ball color=white] (4,3) circle (0.55mm);
\draw (1,3.2) node {{\tiny $h+1$}};
\draw (3,3.2) node {{\tiny $h$}};
\draw (5,3.2) node {{\tiny $h$}};

\end{tikzpicture}
\caption{One of these two functions is in $\trop (W)$.}
\label{Fig:TwoOptions}
\end{figure}

\begin{lemma}
\label{Lem:Case2Dependence}
Both of the following hold:
\begin{enumerate}
\item  Either $s_k (\varphi_D) = s_k [h+1]$ for all $k \leq \ell$, or $s_k (\varphi_E) = s_k [h]$ for all $\ell < k \leq \ell'$.
\item  Either $s_k (\varphi_C) = s_k [h]$ for all $k > \ell'$, or $s_k (\varphi_E) = s_k [h+1]$ for all $\ell < k \leq \ell'$.
\end{enumerate}
\end{lemma}

\begin{proof}
Because $W$ is a pencil, the functions $\varphi_C, \varphi_D, $ and $\varphi_E$ from Proposition~\ref{Prop:Case2Fns} are tropically dependent.  If $s_{\ell} (\varphi_D) \neq s_{\ell} [h+1]$, then in this dependence the functions $\varphi_C$ and $\varphi_E$ must achieve the minimum at $v_{\ell}$.  All three functions agree on the loop $\gamma_{\ell}$, and since $s_{\ell+1} (\varphi_C) = s_{\ell+1} [h]$, it follows that one of the other two functions must also have slope $s_{\ell+1} [h]$ along the bridge $\beta_{\ell}$.  By definition, this function cannot be $\varphi_D$, so we must have $s_{\ell+1} (\varphi_E) = s_{\ell+1} [h]$.

The second statement follows by an analogous argument, considering the functions that achieve the minimum at $w_{\ell'}$.
\end{proof}

\begin{lemma}
\label{Lem:ReduceToOne}
One of the following holds.  Either
\begin{enumerate}
\item  there is no function $\varphi \in \trop (W)$ with $s_{\ell} (\varphi) = s_{\ell} [h]$ and $s'_{\ell} (\varphi) = s'_{\ell} [h+1]$, or
\item  there is no function $\varphi \in \trop (W)$ with $s_{\ell'} (\varphi) = s_{\ell'} [h]$ and $s'_{\ell'} (\varphi) = s'_{\ell'} [h+1]$.
\end{enumerate}
\end{lemma}

\begin{proof}
Consider the case where the first function depicted in Figure~\ref{Fig:TwoOptions} is contained in $\Sigma$.  Without loss of generality, assume that this function is $\varphi_E$.  Let $\varphi \in \trop (W)$ be a function with $s_{\ell} (\varphi) = s_{\ell} [h]$.  Because $W$ is a pencil, the functions $\varphi, \varphi_C,$ and $\varphi_E$ are tropically dependent.  Because $s_{\ell} (\varphi) = s_{\ell} [h]$, we see that in this dependence $\varphi_C$ and $\varphi_E$ must achieve the minimum at $w_{\ell}$.  Since $s'_{\ell} (\varphi_C) = s'_{\ell} [h]$, it follows that one of the other two functions must also have slope $s'_{\ell} [h]$ along the bridge $\beta_{\ell}$.  By assumption, this function cannot be $\varphi_E$, so it must be $\varphi$.

The other case, where the second function depicted in Figure~\ref{Fig:TwoOptions} is contained in $\Sigma$, follows by an analogous argument.
\end{proof}

\begin{remark}  \label{Rem:SwitchingPencil}
Note that Lemma~\ref{Lem:ReduceToOne} does not imply that $\gamma_{\ell}$ or $\gamma_{\ell'}$ is not a switching loop for $\Sigma$.  Rather, it implies that one of these loops is not a switching loop for the smaller tropical linear subseries $\trop (W)$.  This is sufficient for our purposes.
\end{remark}

If there is no function $\varphi \in \trop (W)$ with $s_{\ell'} (\varphi) = s_{\ell'} [h]$ and $s'_{\ell'} (\varphi) = s'_{\ell'} [h+1]$, then we construct our maximally independent combination $\vartheta$ as though $\gamma_{\ell'}$ is not a switching loop.  We define the sets $\cA$ and $\cB$ exactly as in case 3a.  The functions $\varphi_A$ and $\varphi_B$ satisfy the same conditions on slopes as in case 3a, and the function described in Lemma~\ref{Lem:Case2Dependence} satisfies the same conditions on slopes as $\varphi_C$ in case 3a.  We may therefore replace the functions $\varphi^0_h$, $\varphi^0_{h+1}$, and $\varphi^{\infty}_h$ in $\cA$ with these three functions, exactly as in case 3a. Similarly, if there is no function $\varphi \in \trop (W)$ with $s_{\ell} (\varphi) = s_{\ell} [h]$ and $s'_{\ell} (\varphi) = s'_{\ell} [h+1]$, then we construct our maximally independent combination $\vartheta$ as though $\gamma_{\ell}$ is not a switching loop.

\begin{theorem}
In this subcase, $\vartheta$ is a maximally independent combination.
\end{theorem}

\begin{proof}
By Lemma~\ref{Lem:ReduceToOne}, one of the two loops $\gamma_{\ell}, \gamma_{\ell'}$ is not a switching loop for $\trop (W)$.  Thus, either none of the functions involved in the tropical linear combination $\vartheta$ switch slope at $\gamma_{\ell}$, or none of them switch slope at $\gamma_{\ell'}$.  This case therefore reduces to Theorem~\ref{Thm:Switching}.
\end{proof}

\medskip

\subsubsection{Subcase 4c:  $h' =h+1$}

We first identify a suitable collection of functions in $\Sigma$.  These are the functions $\varphi_i$, for $i \notin \{ h, h+1, h + 2 \}$, together with those illustrated in Figure~\ref{Fig:Case3}.

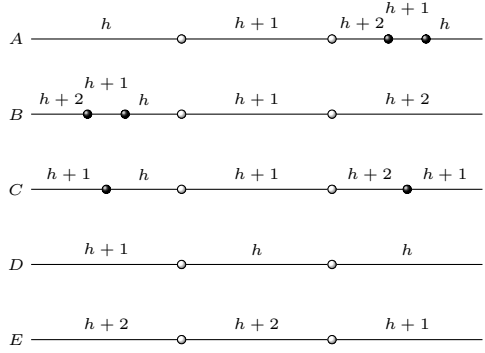
\begin{figure}[H]
\begin{tikzpicture}

\draw (-0.2,4) node {{\tiny $A$}};
\draw (0,4)--(6,4);
\draw [ball color=white] (2,4) circle (0.55mm);
\draw [ball color=white] (4,4) circle (0.55mm);
\draw [ball color=black] (4.75,4) circle (0.55mm);
\draw [ball color=black] (5.25,4) circle (0.55mm);
\draw (1,4.2) node {{\tiny $h$}};
\draw (3,4.2) node {{\tiny $h+1$}};
\draw (4.4,4.2) node {{\tiny $h+2$}};
\draw (5,4.4) node {{\tiny $h+1$}};
\draw (5.5,4.2) node {{\tiny $h$}};

\draw (-0.2,3) node {{\tiny $B$}};
\draw (0,3)--(6,3);
\draw [ball color=white] (2,3) circle (0.55mm);
\draw [ball color=white] (4,3) circle (0.55mm);
\draw [ball color=black] (0.75,3) circle (0.55mm);
\draw [ball color=black] (1.25,3) circle (0.55mm);
\draw (0.4,3.2) node {{\tiny $h+2$}};
\draw (1,3.4) node {{\tiny $h+1$}};
\draw (1.5,3.2) node {{\tiny $h$}};
\draw (3,3.2) node {{\tiny $h+1$}};
\draw (5,3.2) node {{\tiny $h+2$}};

\draw (-0.2,2) node {{\tiny $C$}};
\draw (0,2)--(6,2);
\draw [ball color=white] (2,2) circle (0.55mm);
\draw [ball color=white] (4,2) circle (0.55mm);
\draw [ball color=black] (1,2) circle (0.55mm);
\draw [ball color=black] (5,2) circle (0.55mm);
\draw (0.5,2.2) node {{\tiny $h+1$}};
\draw (1.5,2.2) node {{\tiny $h$}};
\draw (3,2.2) node {{\tiny $h+1$}};
\draw (4.5,2.2) node {{\tiny $h+2$}};
\draw (5.5,2.2) node {{\tiny $h+1$}};

\draw (-0.2,1) node {{\tiny $D$}};
\draw (0,1)--(6,1);
\draw [ball color=white] (2,1) circle (0.55mm);
\draw [ball color=white] (4,1) circle (0.55mm);
\draw (1,1.2) node {{\tiny $h+1$}};
\draw (3,1.2) node {{\tiny $h$}};
\draw (5,1.2) node {{\tiny $h$}};

\draw (-0.2,0) node {{\tiny $E$}};
\draw (0,0)--(6,0);
\draw [ball color=white] (2,0) circle (0.55mm);
\draw [ball color=white] (4,0) circle (0.55mm);
\draw (1,0.2) node {{\tiny $h+2$}};
\draw (3,0.2) node {{\tiny $h+2$}};
\draw (5,0.2) node {{\tiny $h+1$}};

\end{tikzpicture}
\caption{The five functions of Proposition~\ref{Prop:Case3Fns}.  The white points represent the switching loops $\gamma_{\ell}$ and $\gamma_{\ell'}$.  The black points are at indeterminate locations in the regions where they appear.}
\label{Fig:Case3}
\end{figure}

\begin{proposition}
\label{Prop:Case3Fns}
There exist functions $\varphi_A , \varphi_B , \varphi_C , \varphi_D , \varphi_E \in \Sigma$ with the following properties:
\begin{enumerate}
\item  $s_k (\varphi_A) = s_k [h]$ for all $k\leq\ell$.
\item  $s_k (\varphi_B) = s_k [h+2]$ for all $k > \ell'$.
\item  $s_k (\varphi_C) = s_k [h+1]$ for all $\ell < k \leq \ell'$.
\item  $s_k (\varphi_D) = s_k [h+1]$ for all $k \leq \ell$ and $s_k (\varphi_D) = s_k [h]$ for all $k > \ell$.
\item  $s_k (\varphi_E) = s_k [h+2]$ for all $k \leq \ell'$, and $s_k (\varphi_E) = s_k [h+1]$ for all all $k > \ell'$.
\end{enumerate}
Moreover, these functions can be chosen such that $s_k (\varphi) \in \{ s_k [h], s_k [h+1], s_k [h+2]\}$ for all $k$.
\end{proposition}

\begin{proof}
As in the previous case, let $p_0$ be a point on $X$ specializing to $w_0$ and $p_{g+1}$ a point on $X$ specializing to $v_{g+1}$, and let $W$ be a 3-dimensional subspace of $\cL (D_X)$ consisting of function that vanish to order at least $h$ at $p_{g+1}$ and order at least $r-(h+2)$ at $p_0$.  Let $W_1 \subset W$ be a pencil of functions in $\cL (D_X)$ that vanish to order at least $h$ at $p_{g+1}$ and order at least $r-(h+1)$ at $p_0$.  Similarly, let $W_2 \subset W$ be a pencil of functions in $\cL (D_X)$ that vanish to order at least $h+1$ at $p_{g+1}$ and order at least $r-(h+2)$ at $p_{g+1}$.  Each of the 5 functions will be the tropicalization of a function in one of these two pencils.  Note that, if $\varphi \in \trop (W_1)$, then $s_1 (\varphi) \leq s_1 [h+1]$ and $s_{g+1} (\varphi) \geq s_{g+1} [h]$.  Because $\gamma_{\ell}$ and $\gamma_{\ell'}$ are the only switching loops, this implies that $s_k [h] \leq s_k (\varphi) \leq s_k [h+2]$ for all $k$, and $s_k (\varphi) \leq s_k [h+1]$ for all $k \leq \ell'$.  Similarly, if $\varphi \in \trop (W_2)$, then $s_k [h] \leq s_k (\varphi) \leq s_k [h+2]$ for all $k$, and $s_k (\varphi) \geq s_k [h+1]$ for all $k > \ell$.

Let $\varphi_A$ be a function in $\trop (W_1)$ such that $s_1 (\varphi_A) = s_1 [h]$.  Similarly, let $\varphi_B$ be a function in $\trop (W_2)$ such that $s_{g+1} (\varphi_B) = s_{g+1} [h+2]$.  Then, we let $\varphi_C$ be a function in $\trop (W_1 \cap W_2)$; this intersection is nontrivial, because $W_1$ and $W_2$ are 2-dimensional subspaces of the 3-dimensional space $W$.  By the computations above, the choice of $\varphi_C$ forces $s_k (\varphi_C) = s_k [h+1]$ for all $\ell < k \leq \ell'$.

To obtain $\varphi_D$, we consider two functions.  There is some $\varphi \in \trop (W_1)$ such that $s_{\ell} (\varphi) = s_{\ell} [h+1]$, and there is a function $\varphi' \in \trop (W_1)$ such that $s_{\ell+1} (\varphi') = s_{\ell+1} [h]$.  By taking the minimum of these two functions, we obtain $\varphi_D$.  Similarly, $\varphi_E$ is obtained by taking the minimum of a function $\varphi \in \trop (W_2)$ with $s_{\ell'} (\varphi) = s_{\ell'} [h+2]$ and a function $\varphi' \in \trop (W_2)$ with $s_{\ell'+1} (\varphi') = s_{\ell'+1} [h+1]$.
\end{proof}

Note that the three functions $\varphi_A, \varphi_C,$ and $\varphi_D$ in $\trop (W_1)$ are tropically dependent.  They are determined, up to additive constants, by a single parameter $t_1$, just as in case 2 (see Proposition~\ref{Prop:BridgeDependence}).  Similarly, the three functions $\varphi_B, \varphi_C,$ and $\varphi_E$, in $\trop (W_2)$ are tropically dependent and are determined up to additive constants by a single parameter $t_2$.

\begin{lemma}
\label{Lem:Case3Inequivalent}
For any $k$, the functions $\varphi_D$ and $\varphi_E$ do not agree on $\gamma_k$.  Moreover, for any pair $k' \leq k$, with $k \neq \ell$ and $k' \neq \ell'$, one of the three functions $\varphi_A, \varphi_B, \varphi_C$ does not agree with $\varphi_E$ on $\gamma_{k'}$, nor with $\varphi_D$ on $\gamma_k$.
\end{lemma}

\begin{proof}
This can be seen from Figure~\ref{Fig:Case3}.  First, identify the regions of the graph containing $k'$ and $k$.  For each of the 6 possibilities, one of the three functions $\varphi_A, \varphi_B, \varphi_C$ disagrees with $\varphi_E$ on the first region and with $\varphi_D$ on the second region.  For example, if $k' \leq k \leq \ell$, then $s'_{k'} (\varphi_A) \neq s'_{k'} (\varphi_E)$, so $\varphi_A$ does not agree with $\varphi_E$ on $\gamma_{k'}$, and $s'_k (\varphi_A) \neq s'_k (\varphi_D)$, so $\varphi_A$ does not agree with $\varphi_D$ on $\gamma_k$.  The other 5 possibilities follow by a similar argument.
\end{proof}

In this case, we let $\cA$ be the set of all building blocks, which is finite because in the two switching loops case there are no decreasing loops.  Note that the restriction of a building block to either $\Gamma_{[1,\ell'-1]}$ or $\Gamma_{[\ell+1,g]}$ agrees with the restriction of some building block from subcase 3a.  These subgraphs cover $\Gamma$, so we see that $\cA$ satisfies property $(\ast)$.

We now describe how to choose the set $\cB$, depending on the parameters $t_1$ and $t_2$, in such way that it satisfies property $(\dagger\dagger)$.  Suppose that there are two indices $j$ and $j'$ such that
\[
s'_{\ell}(\theta) = s'_{\ell} [h] + s'_{\ell} [j] = s'_{\ell} [h] + s'_{\ell} [j'] +1.
\]
If $t_1<m_{\ell}$, we exclude from the set $\cB$ any function of the form $\varphi + \varphi_{j'}$, where $\varphi$ satisfies
\[
s_{\ell+1} (\varphi) = s_{\ell} (\varphi) + 1 = s_{\ell+1} [h+1].
\]
If $t_1 \geq m_{\ell}$, we exclude from the set $\cB$ any function of the form $\varphi + \varphi_j$, where $\varphi$ satisfies
\[
s_{\ell+1} (\varphi) = s_{\ell} (\varphi) = s_{\ell+1} [h].
\]
We similarly exclude functions depending on $t_2$.  The set $\cB$ satisfies properties $(\ast\ast)$ and $(\dagger\dagger)$ by the same argument as the proof of Lemma~\ref{Lem:BBOnLoop}.

Because the set $\cB$ satisfies the hypotheses of Theorem~\ref{Thm:ExtremalConfig}, we may construct the master template $\theta$.  Our strategy for constructing the maximally independent combination $\vartheta$ from the master template is very similar to cases 2 and 3a.  We go through the steps in a similar way to case 2.

\medskip

\textbf{Step 1:  if $i,j \notin \{ h,h+1,h+2\}$.}  The function $\varphi_{ij}$ is in $\cB$, and we set its coefficient in the tropical linear combination $\vartheta$ equal to its coefficient in the master template $\theta$.

\medskip

\textbf{Step 2:  if $\varphi + \varphi_j$ is contained in $\cB$ for all $\varphi \in \cA$.} By Lemma~\ref{Lem:Replace}, we set the coefficient of $\varphi_D + \varphi_j$ so that it dominates the master template $\theta$, and equals $\theta$ on some region.  Similarly, we set the coefficient of $\varphi_E + \varphi_j$ so that it dominates the master template $\theta$, and equals $\theta$ on some region.  By Lemma~\ref{Lem:Case3Inequivalent}, $\varphi_D$ and $\varphi_E$ do not agree on any loop, so there is a loop or bridge where $\varphi_D + \varphi_j$ achieves the minimum and $\varphi_E + \varphi_j$ does not, and vice-versa.  Furthermore, by Lemma~\ref{Lem:Case3Inequivalent}, one of the three functions $\varphi_A + \varphi_j$, $\varphi_B + \varphi_j$, or $\varphi_C + \varphi_j$ does not agree with either of these on these two particular loops or bridges.  We then set the coefficient of this third function so that it dominates the master template, and equals it on some region.

Now, the situations where we have omitted some functions from $\cB$ are handled in a way similar to the corresponding situations in cases 2 and 3a.  Each function that we use is a tropical linear combination of sums of 2 building blocks, only some of which are contained in the set $\cB$.  In each case, we will set the coefficients of the functions so that they agree with the summand in $\cB$ with largest coefficient, in the same way as Step 3 of case 2.  We will specifically address the different possible situations according to the value of $t_1$; the different possible situations according to the value of $t_2$ are handled in a precisely analogous manner.

\medskip

\textbf{Step 3:  if $t_1 < m_{\ell}$.}  We set the coefficients of $\varphi_D + \varphi_j$ and $\varphi_E + \varphi_j$ as in Step 2.  We can do this because they are each tropical linear combinations of functions in $\cB$.  We choose one of the three functions $\varphi_A + \varphi_j$, $\varphi_B + \varphi_j$, or $\varphi_C + \varphi_j$ exactly as in step 2.  This function is a tropical linear combination of sums of 2 building blocks, only some of which are contained in $\cB$.  We set the coefficient of this function so that it equals its summand in $\cB$ with largest coefficient, in the same way as Step 3 of case 2.

\medskip

\textbf{Step 4:  if $t_1 \geq m_{\ell}$.}  We set the coefficient of $\varphi_E + \varphi_j$ as in step 2, and set the coefficients of $\varphi_A + \varphi_j$ and $\varphi_D + \varphi_j$ to agree with its summand in $\cB$ with largest coefficient, as in step 3.

\medskip

\textbf{Step 5:  if $j = h,h+1$ or $h+2$.}  We apply a similar construction.  Specifically, for each sum of two functions in $\cA$, we first use Lemma~\ref{Lem:Replace} to replace the first summand, as in steps 2-4, and then repeat the procedure to replace the second summand.

\begin{theorem}
In this subcase, $\vartheta$ is a maximally independent combination.
\end{theorem}

\begin{proof}
The proof in this subcase is very similar to that of Theorem~\ref{Thm:SwitchingBridge}.  To see that there are exactly 28 functions in the tropical linear combination $\vartheta$, we observe that $\vartheta$ is constructed to involve exactly as many functions as there are pairs $(i,j)$, $0 \leq i \leq j \leq 6$.  Specifically, for $i,j \notin \{ h,h+1,h+2\}$, we have the function $\varphi_{ij}$.  For $j \notin \{ h,h+1,h+2\}$, the three pairs $(h,j), (h+1,j), (h+2,j)$ correspond to three functions, two of which are $\varphi_D + \varphi_j$ and $\varphi_E + \varphi_j$, and the third of which is one of $\varphi_A + \varphi_j, \varphi_B + \varphi_j,$ or $\varphi_C + \varphi_j$.  Similarly, by construction, the six pairs $(i,j)$ with both $i,j \in \{ h,h+1,h+2\}$ correspond to six functions appearing in $\vartheta$.  It remains to verify that $\vartheta$ is maximally independent.

We first discuss the case where we have omitted no building blocks from $\cB$.  By construction, for each $j \notin \{ h,h+1,h+2\}$, the functions $\varphi_D + \varphi_j$ and $\varphi_E + \varphi_j$ are involved in the tropical linear combination $\vartheta$, as is one of $\varphi_A + \varphi_j, \varphi_B + \varphi_j$, or $\varphi_C + \varphi_j$.  Each of these functions achieves the minimum on the same region as one of its summands.  Specifically, there is a function $\psi_D \in \cB$ such that $\varphi_D + \varphi_j$ achieves the minimum on the same region as $\psi_D$, and there is a function $\psi_E \in \cB$ such that $\varphi_E + \varphi_j$ achieves the minimum on the same region as $\psi_E$.  There is a also a function $\psi_A \in \cB$ such that the third function achieves the minimum on the same region as $\psi_A$.  Although the third function shares summands in common with $\varphi_D + \varphi_j$ and $\varphi_E + \varphi_j$, it is chosen so that it does not agree with either $\varphi_D + \varphi_j$ or $\varphi_E + \varphi_j$ on specific loops or bridges where these two functions achieve the minimum.  Thus, $\psi_A$ does not agree with $\psi_D$ on the loop or bridge where $\psi_D$ is assigned, and does not agree with $\psi_E$ on the loop or bridge where $\psi_E$ is assigned.  It follows that the loop or bridge to which $\psi_A$ is assigned is not equal to the loop or bridge to which either $\psi_D$ or $\psi_E$ is assigned, so the third function is the only function to achieve the minimum at a point of the loop or bridge where $\psi_A$ is assigned.

Now, the cases where we have omitted some functions from $\cB$ follow in the same way as Theorem~\ref{Thm:Switching}.  This is because the arguments in that case rely only on the slopes along the bridges $\beta_{\ell}$ and $\beta_{\ell-1}$.  This ``local'' information remains unchanged in our present case, so the result follows.
\end{proof}

\medskip

\subsubsection{Subcase 4d:  $h'=h-1$}  \label{sec:4d}

In the previous three cases, our analysis reduced to the study of the tropicalizations of certain pencils.  This last case is a little different, as it does not appear to reduce to the case of pencils.  Nevertheless, the arguments are of a similar flavor, with a just a few more combinatorial possibilities.  As in the previous cases, we begin by describing several functions that are contained in $\Sigma$ and will be used in our construction of the maximally independent combination $\vartheta$.  These functions are $\varphi_i$, for $i \notin \{h-1, h, h+1 \}$ together with those illustrated in Figure~\ref{Fig:Case4}.

\begin{figure}[H]
\begin{tikzpicture}

\draw (-0.2,4) node {{\tiny $A$}};
\draw (0,4)--(6,4);
\draw [ball color=white] (2,4) circle (0.55mm);
\draw [ball color=white] (4,4) circle (0.55mm);
\draw [ball color=black] (5,4) circle (0.55mm);
\draw (1,4.2) node {{\tiny $h-1$}};
\draw (3,4.2) node {{\tiny $h-1$}};
\draw (4.5,4.2) node {{\tiny $h$}};
\draw (5.5,4.2) node {{\tiny $h-1$}};

\draw (-0.2,3) node {{\tiny $B$}};
\draw (0,3)--(6,3);
\draw [ball color=white] (2,3) circle (0.55mm);
\draw [ball color=white] (4,3) circle (0.55mm);
\draw [ball color=black] (1,3) circle (0.55mm);
\draw (0.5,3.2) node {{\tiny $h+1$}};
\draw (1.5,3.2) node {{\tiny $h$}};
\draw (3,3.2) node {{\tiny $h+1$}};
\draw (5,3.2) node {{\tiny $h+1$}};

\draw (-0.2,2) node {{\tiny $C$}};
\draw (0,2)--(6,2);
\draw [ball color=white] (2,2) circle (0.55mm);
\draw [ball color=white] (4,2) circle (0.55mm);
\draw [ball color=black] (5,2) circle (0.55mm);
\draw (1,2.2) node {{\tiny $h+1$}};
\draw (3,2.2) node {{\tiny $h-1$}};
\draw (4.5,2.2) node {{\tiny $h$}};
\draw (5.5,2.2) node {{\tiny $h-1$}};

\draw (-0.2,1) node {{\tiny $D$}};
\draw (0,1)--(6,1);
\draw [ball color=white] (2,1) circle (0.55mm);
\draw [ball color=white] (4,1) circle (0.55mm);
\draw [ball color=black] (1,1) circle (0.55mm);
\draw (0.5,1.2) node {{\tiny $h+1$}};
\draw (1.5,1.2) node {{\tiny $h$}};
\draw (3,1.2) node {{\tiny $h+1$}};
\draw (5,1.2) node {{\tiny $h-1$}};

\draw (-0.2,0) node {{\tiny $E$}};
\draw (0,0)--(6,0);
\draw [ball color=white] (2,0) circle (0.55mm);
\draw [ball color=white] (4,0) circle (0.55mm);
\draw [ball color=black] (1,0) circle (0.55mm);
\draw [ball color=black] (5,0) circle (0.55mm);
\draw (0.5,0.2) node {{\tiny $h+1$}};
\draw (1.5,0.2) node {{\tiny $h$}};
\draw (3,0.2) node {{\tiny $h$}};
\draw (4.5,0.2) node {{\tiny $h$}};
\draw (5.5,0.2) node {{\tiny $h-1$}};

\end{tikzpicture}
\caption{The five functions of Proposition~\ref{Prop:Case4Fns}.  The black points are at indeterminate locations in the regions where they appear.}
\label{Fig:Case4}
\end{figure}
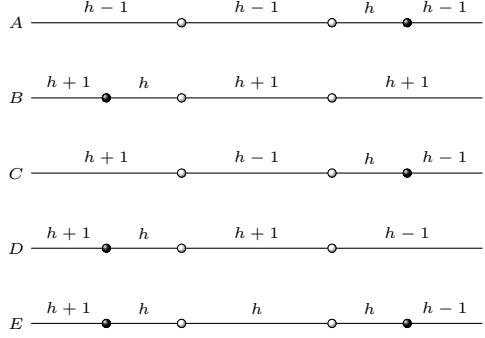

\begin{proposition}
\label{Prop:Case4Fns}
There exist functions $\varphi_A , \varphi_B , \varphi_C , \varphi_D , \varphi_E \in \Sigma$ with the following properties:
\begin{enumerate}
\item  $s_k (\varphi_A) = s_k [h-1]]$ for all $k\leq\ell'$.
\item  $s_k (\varphi_B) = s_k [h+1]$ for all $k > \ell$.
\item  $s_k (\varphi_C) = s_k [h+1]$ for all $k \leq \ell$ and $s_k (\varphi_C) = s_k [h-1]$ for all $\ell < k \leq \ell'$.
\item  $s_k (\varphi_D) = s_k [h+1]$ for all $\ell < k \leq \ell'$, and $s_k (\varphi_D) = s_k [h-1]$ for all all $k > \ell'$.
\item  $s_k (\varphi_E) = s_k [h]$ for all $\ell < k \leq \ell'$.
\end{enumerate}
Moreover, these functions can be chosen such that $s_k (\varphi) \in \{ s_k [h-1], s_k [h], s_k [h+1]\}$, for all $k$.
\end{proposition}

\begin{proof}
As before, let $p_0$ be a point on $X$ specializing to $w_0$ and $p_{g+1}$ a point on $X$ specializing to $v_{g+1}$, and let $W$ be a 3-dimensional subspace of $cL (D_X)$ consisting of functions that vanish to order at least $h-1$ at $p_{g+1}$ and order at least $r-(h+1)$ at $p_0$.  Each of the 6 functions will be the tropicalization of a function in $W$.

Now, we let $\varphi_A$ be a function in $\trop (W)$ such that $s_1 (\varphi_A) = s_1 [h-1]$.  Similarly, we let $\varphi_B$ be a function in $\trop (W)$ such that $s_{g+1} (\varphi_B) = s_{g+1} [h+1]$.

To obtain $\varphi_C$, we take a suitable tropical linear combination of two functions $\varphi$ and $\varphi'$ in $\trop(W)$, chosen such that $s_{\ell} (\varphi) = s_{\ell} [h+1]$ and  $s_{\ell+1} (\varphi') = s_{\ell+1} [h-1]$.  By taking the minimum of these two functions, we obtain $\varphi_C$.  Similarly, $\varphi_D$ is obtained by taking a suitable combination of $\phi$ and $\phi'$, chosen from $\trop(W)$ such that $s_{\ell'} (\phi) = s_{\ell'} [h+1]$ and $s_{\ell'+1} (\phi') = s_{\ell'+1} [h-1]$.  Finally, let $\varphi_E$ be a function in $\trop (W)$ with $s_{\ell+1} (\varphi_E) \leq s_{\ell+1} [h]$ and $s_{\ell'} (\varphi_E) \geq s_{\ell'} [h]$.
\end{proof}

\begin{lemma}
\label{Lem:Case4Dependence}
Both of the following hold:
\begin{enumerate}
\item  Either $s_k (\varphi_A) = s_k [h-1]$ for all $k > \ell'$, or $s_k (\varphi_E) = s_k [h-1]$ for all $k > \ell'$.
\item  Either $s_k (\varphi_B) = s_k [h+1]$ for all $k \leq \ell$, or $s_k (\varphi_E) = s_k [h+1]$ for all $k \leq \ell$.
\end{enumerate}
\end{lemma}

\begin{proof}
Because $W$ is 3-dimensional, the functions $\varphi_A, \varphi_B, \varphi_D, $ and $\varphi_E$ from Proposition~\ref{Prop:Case4Fns} are tropically dependent.  Of these 4, only $\varphi_B$ and $\varphi_D$ have the same slope along $\beta_{\ell'-1}$, thus these two achieve the minimum at $v_{\ell'}$.  Considering $\varphi_D$, it follows that a second function must have slope $s_{\ell'+1} [h-1]$ along $\beta_{\ell'}$.  By construction, this function must be either $\varphi_A$ or $\varphi_E$.

The second statement follows by an analogous argument, using the functions $\varphi_A, \varphi_B, \varphi_C,$ and $\varphi_E$.
\end{proof}

In three of the four situations produced by Lemma~\ref{Lem:Case4Dependence}, $\trop(W)$ only switches one loop.

\begin{lemma}
\label{Lem:Case4ReduceToOne}
If $s_{\ell'+1} (\varphi_A) = s_{\ell'+1} [h-1]$, then there is no function $\varphi \in \trop (W)$ with $s_{\ell'} (\varphi) = s_{\ell'} [h-1]$ and $s_{\ell'+1} (\varphi) = s_{\ell'+1} [h]$.  Similarly, if $s_{\ell} (\varphi_B) = s_{\ell} [h+1]$, then there is no function $\varphi \in \trop (W)$ with $s_{\ell} (\varphi) = s_{\ell} [h]$ and $s_{\ell+1} (\varphi) = s_{\ell+1} [h+1]$.
\end{lemma}

\begin{proof}
Consider the case where $s_{\ell'+1} (\varphi_A) = s_{\ell'+1} [h-1]$.  Let $\varphi \in \trop (W)$ be a function with $s_{\ell'+1} (\varphi) = s_{\ell'+1} [h]$.  Because $W$ is 3-dimensional, the functions $\varphi, \varphi_A, \varphi_B$ and $\varphi_E$ are tropically dependent.  Because only $\varphi$ and $\varphi_E$ have the same slope on $\beta_{\ell'}$, we see that they must achieve the minimum at $w_{\ell'}$.  Considering $\varphi_E$, we see that the minimum has slope at least $s_{\ell'} [h]$ along $\beta_{\ell'-1}$.  Because this slope must be obtained twice, and the three functions $\varphi_A , \varphi_B,$ and $\varphi_E$ have distinct slopes there, we see that $s_{\ell'} (\varphi) \geq s_{\ell'} [h]$.  The other case follows by an analogous argument.
\end{proof}

\noindent

\begin{remark}
Lemma~\ref{Lem:Case4ReduceToOne} shows that, in three of the four situations produced by Lemma~\ref{Lem:Case4Dependence}, one of the two switching loops for $\Sigma$ is not a switching loop for $W$.  This is similar to the situation highlighted in Remark~\ref{Rem:SwitchingPencil}.
\end{remark}

If $\trop(W)$ only switches one loop, then we can construct the maximally independent combination $\vartheta$ by arguments identical to those in case 3a.

\medskip

We now focus on the remaining situation, where both loops are switching loops for $\trop(W)$.  In this situation, the slopes of $\varphi_E$ are as pictured in Figure~\ref{Fig:LastCase}.

\begin{figure}[H]
\begin{tikzpicture}

\draw (0,4)--(6,4);
\draw [ball color=white] (2,4) circle (0.55mm);
\draw [ball color=white] (4,4) circle (0.55mm);
\draw (1,4.2) node {{\tiny $h+1$}};
\draw (3,4.2) node {{\tiny $h$}};
\draw (5,4.2) node {{\tiny $h-1$}};

\end{tikzpicture}
\caption{Shape of the function $\varphi_E$ when $\trop(W)$ switches both loops.}
\label{Fig:LastCase}
\end{figure}
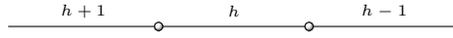

 We define $t_1$ to be the distance from $\gamma_{\ell}$ to the point to the left of $\gamma_{\ell}$ where $\varphi_B$ bends.  Similarly, we define $t_2$ to be the distance from $\gamma_{\ell'}$ to the point to the right of $\gamma_{\ell'}$ where $\varphi_A$ bends.  The functions $\varphi_A$ and $\varphi_B$ are determined by $t_1$ and $t_2$, up to additive constants.  We now describe additional functions in $\trop(W)$, which are similarly determined by $t_1$ and $t_2$.

\begin{proposition}
\label{Prop:MoreFns}
Suppose that $s_{\ell'+1} (\varphi_A) = s_{\ell'+1} [h]$ and $s_{\ell} (\varphi_B) = s_{\ell} [h]$.  Then there exist functions $\varphi_F , \varphi_G , \varphi_H \in \Sigma$ with slopes as depicted in Figure~\ref{Fig:Case4More}.  The locations of the black dots in the figure are determined by $t_1$ and $t_2$ from the dependences illustrated in Figure~\ref{Fig:LastCaseDependence}.
\end{proposition}

\begin{figure}[H]
\begin{tikzpicture}

\draw (-0.4,4) node {{\footnotesize $F$}};
\draw (0,4)--(6,4);
\draw [ball color=white] (2,4) circle (0.55mm);
\draw [ball color=white] (4,4) circle (0.55mm);
\draw [ball color=black] (2.66,4) circle (0.55mm);
\draw [ball color=black] (3.33,4) circle (0.55mm);
\draw (1,4.2) node {{\tiny $h$}};
\draw (2.5,4.2) node {{\tiny $h+1$}};
\draw (3,4.2) node {{\tiny $h$}};
\draw (3.5,4.2) node {{\tiny $h-1$}};
\draw (5,4.2) node {{\tiny $h$}};

\draw (-0.2,3) node {{\tiny or}};
\draw (0,3)--(6,3);
\draw [ball color=white] (2,3) circle (0.55mm);
\draw [ball color=white] (4,3) circle (0.55mm);
\draw [ball color=black] (5,3) circle (0.55mm);
\draw (1,3.2) node {{\tiny $h$}};
\draw (3,3.2) node {{\tiny $h+1$}};
\draw (4.5,3.2) node {{\tiny $h+1$}};
\draw (5.5,3.2) node {{\tiny $h$}};

\draw (-0.2,2) node {{\tiny or}};
\draw (0,2)--(6,2);
\draw [ball color=white] (2,2) circle (0.55mm);
\draw [ball color=white] (4,2) circle (0.55mm);
\draw [ball color=black] (1,2) circle (0.55mm);
\draw (0.5,2.2) node {{\tiny $h$}};
\draw (1.5,2.2) node {{\tiny $h-1$}};
\draw (3,2.2) node {{\tiny $h-1$}};
\draw (5,2.2) node {{\tiny $h$}};

\draw (-0.4,.7) node {{\footnotesize $G$}};
\draw (0,.7)--(6,.7);
\draw [ball color=white] (2,.7) circle (0.55mm);
\draw [ball color=white] (4,.7) circle (0.55mm);
\draw [ball color=black] (3,.7) circle (0.55mm);
\draw (1,.9) node {{\tiny $h$}};
\draw (2.5,.9) node {{\tiny $h+1$}};
\draw (3.5,.9) node {{\tiny $h$}};
\draw (5,.9) node {{\tiny $h-1$}};

\draw (-0.2,-.3) node {{\tiny or}};
\draw (0,-.3)--(6,-.3);
\draw [ball color=white] (2,-.3) circle (0.55mm);
\draw [ball color=white] (4,-.3) circle (0.55mm);
\draw [ball color=black] (5,-.3) circle (0.55mm);
\draw (1,-.1) node {{\tiny $h$}};
\draw (3,-.1) node {{\tiny $h+1$}};
\draw (4.5,-.1) node {{\tiny $h+1$}};
\draw (5.5,-.1) node {{\tiny $h-1$}};

\draw (-0.4,-1.6) node {{\footnotesize $H$}};
\draw (0,-1.6)--(6,-1.6);
\draw [ball color=white] (2,-1.6) circle (0.55mm);
\draw [ball color=white] (4,-1.6) circle (0.55mm);
\draw [ball color=black] (3,-1.6) circle (0.55mm);
\draw (1,-1.4) node {{\tiny $h$}};
\draw (2.5,-1.4) node {{\tiny $h$}};
\draw (3.5,-1.4) node {{\tiny $h-1$}};
\draw (5,-1.4) node {{\tiny $h$}};

\draw (-0.2,-2.6) node {{\tiny or}};
\draw (0,-2.6)--(6,-2.6);
\draw [ball color=white] (2,-2.6) circle (0.55mm);
\draw [ball color=white] (4,-2.6) circle (0.55mm);
\draw [ball color=black] (1,-2.6) circle (0.55mm);
\draw (0.5,-2.4) node {{\tiny $h+1$}};
\draw (1.5,-2.4) node {{\tiny $h-1$}};
\draw (3,-2.4) node {{\tiny $h-1$}};
\draw (5,-2.4) node {{\tiny $h$}};

\end{tikzpicture}
\caption{The five functions of Proposition~\ref{Prop:MoreFns}.  The black points are determined by $t_1$ and $t_2$ from the dependences in Figure~\ref{Fig:LastCaseDependence}.}
\label{Fig:Case4More}
\end{figure}
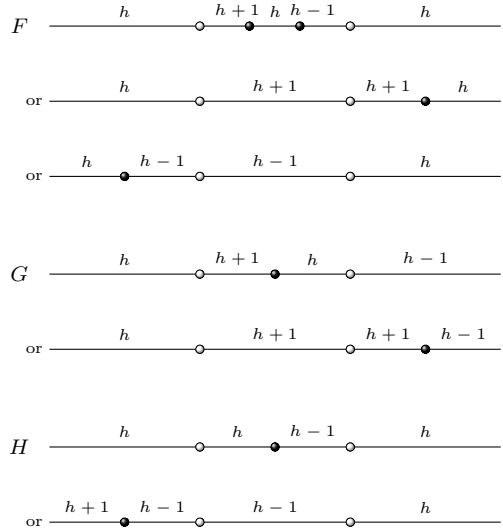

\begin{figure}[H]
\begin{tikzpicture}

\draw (0,4)--(6,4);
\draw [ball color=white] (2,4) circle (0.55mm);
\draw [ball color=white] (4,4) circle (0.55mm);
\draw [ball color=black] (1,4) circle (0.55mm);
\draw [ball color=black] (2.75,4) circle (0.55mm);
\draw [ball color=black] (3.25,4) circle (0.55mm);
\draw [ball color=black] (5,4) circle (0.55mm);
\draw (0.5,4.2) node {{\tiny $BE$}};
\draw (2,4.2) node {{\tiny $BF$}};
\draw (3,4.2) node {{\tiny $EF$}};
\draw (4,4.2) node {{\tiny $AF$}};
\draw (5.5,4.2) node {{\tiny $AE$}};

\draw (-0.2,3) node {{\tiny or}};
\draw (0,3)--(6,3);
\draw [ball color=white] (2,3) circle (0.55mm);
\draw [ball color=white] (4,3) circle (0.55mm);
\draw [ball color=black] (0.75,3) circle (0.55mm);
\draw [ball color=black] (1.25,3) circle (0.55mm);
\draw [ball color=black] (5,3) circle (0.55mm);
\draw (0.35,3.2) node {{\tiny $BE$}};
\draw (1,3.2) node {{\tiny $BF$}};
\draw (3,3.2) node {{\tiny $AF$}};
\draw (5.5,3.2) node {{\tiny $AE$}};

\draw (-0.2,2) node {{\tiny or}};
\draw (0,2)--(6,2);
\draw [ball color=white] (2,2) circle (0.55mm);
\draw [ball color=white] (4,2) circle (0.55mm);
\draw [ball color=black] (1,2) circle (0.55mm);
\draw [ball color=black] (4.75,2) circle (0.55mm);
\draw [ball color=black] (5.25,2) circle (0.55mm);
\draw (0.5,2.2) node {{\tiny $BE$}};
\draw (3,2.2) node {{\tiny $BF$}};
\draw (5,2.2) node {{\tiny $AF$}};
\draw (5.5,2.2) node {{\tiny $AE$}};

\draw (0,.7)--(6,.7);
\draw [ball color=white] (2,.7) circle (0.55mm);
\draw [ball color=white] (4,.7) circle (0.55mm);
\draw [ball color=black] (1,.7) circle (0.55mm);
\draw [ball color=black] (3,.7) circle (0.55mm);
\draw (0.5,.9) node {{\tiny $BE$}};
\draw (2,.9) node {{\tiny $BG$}};
\draw (4,.9) node {{\tiny $EG$}};

\draw (0,-.3)--(6,-.3);
\draw [ball color=white] (2,-.3) circle (0.55mm);
\draw [ball color=white] (4,-.3) circle (0.55mm);
\draw [ball color=black] (3,-.3) circle (0.55mm);
\draw [ball color=black] (5,-.3) circle (0.55mm);
\draw (2,-.1) node {{\tiny $EH$}};
\draw (4,-.1) node {{\tiny $AH$}};
\draw (5.5,-.1) node {{\tiny $AE$}};

\end{tikzpicture}
\caption{Dependences between the functions $\varphi_A , \varphi_B , \varphi_E, \varphi_F, \varphi_G$, and $\varphi_H$.}
\label{Fig:LastCaseDependence}
\end{figure}
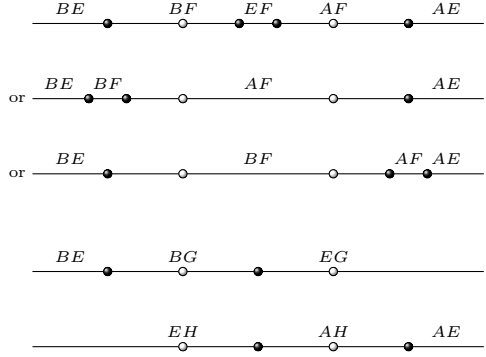

\begin{proof}
We let $\varphi_F$ be a function in $\trop (W)$ such that $s'_0 (\varphi_F) \leq h$ and $s_{g+1} (\varphi_F) \geq h$.  The slopes of $\varphi_F$ fall into one of the three possibilities depicted in Figure~\ref{Fig:Case4More}.  We now explain how to construct $\varphi_G$ and $\varphi_H$.  Let $f_A, f_B, f_E \in W$ be functions tropicalizing to $\varphi_A$, $\varphi_B$, and $\varphi_E$, respectively.  We let $\varphi_G$ be the tropicalization of a function in the pencil spanned by $f_B$ and $f_E$ with the property that $s'_0 (\varphi_G) \neq s'_0 [h+1]$.  Similarly, we let $\varphi_H$ be the tropicalization of a function in the pencil spanned by $f_A$ and $f_E$ with the property that $s_{g+1} (\varphi_H) \neq s_{g+1} [h-1]$.

By construction, the functions $\varphi_B$, $\varphi_E$, and $\varphi_G$ are tropically dependent.  The only possible dependence between them is illustrated in the fourth row of Figure~\ref{Fig:LastCaseDependence}.  Note that the point where $\varphi_G$ bends is determined by the distance $t_1$.  Similarly, the functions $\varphi_A$, $\varphi_E$, and $\varphi_H$ are tropically dependent, and the point where $\varphi_H$ bends is determined by the distance $t_2$, by the dependence shown in the last row of Figure~\ref{Fig:LastCaseDependence}.  Finally, the functions $\varphi_A$, $\varphi_B$, $\varphi_E$, and $\varphi_F$ are tropically dependent, and this dependence comes in one of three combinatorial types, depending on the three possibilities for the slopes of $\varphi_F$, as depicted in Figures~\ref{Fig:Case4More} and~\ref{Fig:LastCaseDependence}.  Again, the locations of the bends are determined by the distances $t_1$ and $t_2$.
\end{proof}

The following lemma will be useful in the proof of this subcase.

\begin{lemma}
\label{Lem:Case4Inequivalent}
The functions $\varphi_A$ and $\varphi_B$ do not agree on $\gamma_k$ for any $k$.  Moreover, for any pair $k' \leq k$ with $k' \neq \ell'$ and $k \neq \ell$, one of the four functions $\varphi_E, \varphi_F, \varphi_G, \varphi_H$ does not agree with $\varphi_B$ on $\gamma_{k'}$, nor with $\varphi_A$ on $\gamma_k$.
\end{lemma}

\begin{proof}
This can be seen from Figures~\ref{Fig:Case4} and \ref{Fig:Case4More}.  First, identify the regions of the graph containing $k$ and $k'$.  For each possibility, one of the four functions $\varphi_E, \varphi_F, \varphi_G, \varphi_H$ disagrees with $\varphi_B$ on the first region and with $\varphi_A$ on the second region.  For example, if $k'$ is to the left of the point where $\varphi_B$ bends and $k > \ell'$ is to the left of the point where $\varphi_A$ bends, then $\varphi_G$ does not agree with $\varphi_B$ on $\gamma_{k'}$, nor with $\varphi_A$ on $\gamma_k$.  The other possibilities follow by a similar argument.
\end{proof}

We construct the sets $\cA$ and $\cB$ as in the previous case.  Specifically, we let $\cA$ be the set of all building blocks, which, as in the previous subcase, is finite and satisfies property $(\ast)$.  Suppose that there exist two indices $j$ and $j'$ such that
\[
s'_{\ell}(\theta) = s'_{\ell} [h] + s'_{\ell} [j] = s'_{\ell} [h] + s'_{\ell} [j'] +1.
\]
If $t_1<m_{\ell}$, we exclude from the set $\cB$ any function of the form $\varphi + \varphi_{j'}$, where
\[
s_{\ell+1} (\varphi) = s_{\ell} (\varphi) + 1 = s_{\ell} [h+1].
\]
If $t_1 \geq m_{\ell}$, we exclude from the set $\cB$ any function of the form $\varphi + \varphi_j$, where
\[
s_{\ell+!} (\varphi) = s_{\ell} (\varphi) = s_{\ell} [h].
\]
We similarly exclude functions depending on $t_2$.  The set $\cB$ consists of all pairwise sums of elements of building blocks, aside from those excluded, and satisfies properties $(\ast\ast)$ and $(\dagger\dagger)$ by the same argument as the proof of Lemma~\ref{Lem:BBOnLoop}.

The set $\cB$ therefore satisfies the hypotheses of Theorem~\ref{Thm:ExtremalConfig}, and so we construct the master template $\theta$.  Our strategy for constructing the maximally independent combination $\vartheta$ from the master template follows steps similar to those in Case 2.

\medskip

\textbf{Step 1:  if $i,j \notin \{ h-1,h,h+1\}$.}  The function $\varphi_{ij}$ is in $\cB$, and we set its coefficient in the tropical linear combination $\vartheta$ to be equal to its coefficient in the master template $\theta$.

\medskip

\textbf{Step 2:  if $\varphi + \varphi_j$ is contained in $\cB$ for all $\varphi \in \cA$.} By Lemma~\ref{Lem:Replace}, we set the coefficients of $\varphi_A + \varphi_j$ and $\varphi_B + \varphi_j$ so that they each dominate the master template $\theta$, and equal $\theta$ on some region.  By Lemma~\ref{Lem:Case4Inequivalent}, $\varphi_A$ and $\varphi_B$  do not agree on any loop, so there is a loop or bridge where $\varphi_A + \varphi_j$ achieves the minimum and $\varphi_B + \varphi_j$ does not, and vice-versa.  Furthermore, by Lemma~\ref{Lem:Case4Inequivalent}, one of the four functions $\varphi_E + \varphi_j$, $\varphi_F + \varphi_j$, $\varphi_G + \varphi_j$, or $\varphi_H + \varphi_j$  does not agree with either $\varphi_A + \varphi_j$ or $\varphi_B + \varphi_j$ on these two particular loops or bridges.  We then set the coefficient of this third function so that it dominates the master template, and equals it on some region.

\medskip

\textbf{Step 3:  if $t_1 < m_{\ell}$ (or if $t_2 < m_{\ell'}$).}  We set the coefficient of $\varphi_A + \varphi_j$ as in Step 2.  We can do this because this function is a tropical linear combinations of functions in $\cB$.  We set the coefficient of $\varphi_B + \varphi_j$ so that it equals its summand in $\cB$ with largest coefficient.  As in step 2, one of the four functions $\varphi_E + \varphi_j$, $\varphi_F + \varphi_j$, $\varphi_G + \varphi_j$, or $\varphi_H + \varphi_j$  does not agree with either $\varphi_A + \varphi_j$ or $\varphi_B + \varphi_j$ on loops or bridges where these functions achieve the minimum.  We set the coefficient of this function using Lemma~\ref{Lem:Replace}, and  when $t_2 < m_{\ell'}$, we apply an analogous construction.

\medskip

\textbf{Step 4:  if $t_1 \geq m_{\ell}$.}  Regardless of the value of $t_2$, we set the coefficients of $\varphi_A + \varphi_j$ $\varphi_E + \varphi_j$, and $\varphi_G + \varphi_j$ using Lemma~\ref{Lem:Replace}.

\medskip

\textbf{Step 5:  if $t_2 \geq m_{\ell'}$.}  We set the coefficients of $\varphi_A + \varphi_j$, $\varphi_B + \varphi_j$ using Lemma~\ref{Lem:Replace}.  One of $\varphi_E + \varphi_j$ or $\varphi_G + \varphi_j$ does not agree with $\varphi_B + \varphi_j$ on a loop or bridge where it achieves the minimum.  We then set the coefficient of this third function using Lemma~\ref{Lem:Replace}.

\medskip

\textbf{Step 6:  if $j = h-1$, $h,$ or $h+1$.}  A similar construction can be applied to the case where $j=h-1, h$ or $h+1$, and we provide only a brief sketch.  We first use Lemma~\ref{Lem:Replace} to set the coefficients of $\varphi_A + \varphi_A$, $\varphi_A + \varphi_B$, and $\varphi_B + \varphi_B$.  Depending on where $\varphi_A + \varphi_A$ and $\varphi_A + \varphi_B$ achieve the minimum, we then set the coefficients of one of $\varphi_A + \varphi_E$, $\varphi_A + \varphi_F$, $\varphi_A + \varphi_G$, or $\varphi_A + \varphi_H$.  Similarly, depending on where $\varphi_A + \varphi_B$ and $\varphi_B + \varphi_B$ achieve the minimum, we set the coefficients of one of $\varphi_B + \varphi_E$, $\varphi_B+ \varphi_F$, $\varphi_B + \varphi_G$, or $\varphi_B + \varphi_H$.   Finally, depending on where all these functions achieve the minimum, we then choose a sum of two functions, each chosen from $\varphi_E$, $\varphi_F$, $\varphi_G$, and $\varphi_H$, and use Lemma~\ref{Lem:Replace} to set its coefficient.

\begin{theorem}
In this subcase, $\vartheta$ is a maximally independent combination.
\end{theorem}

\begin{proof}
We first discuss the situation where we have omitted no functions from $\cB$.  In this situation, for each $j \notin \{ h-1,h,h+1\}$, the functions $\varphi_A + \varphi_j$ and $\varphi_B + \varphi_j$ are involved in the tropical linear combination $\vartheta$, as is one of $\varphi_E + \varphi_j, \varphi_F + \varphi_j, \varphi_G + \varphi_j$, or $\varphi_H + \varphi_j$.  Each of these functions achieves the minimum on the same region as one of its summands.  Specifically, there is a function $\psi_A \in \cB$ such that $\varphi_A + \varphi_j$ achieves the minimum on the same region as $\psi_A$, and there is a function $\psi_B \in \cB$ such that $\varphi_B + \varphi_j$ achieves the minimum on the same region as $\psi_B$.  There is a also a function $\psi_E \in \cB$ such that the third function achieves the minimum on the same region as $\psi_E$.  Although the third function shares summands in common with $\varphi_A + \varphi_j$ and $\varphi_B + \varphi_j$, it is chosen so that it does not agree with either $\varphi_A + \varphi_j$ or $\varphi_B + \varphi_j$ on specific loops or bridges where these two functions achieve the minimum.  Thus, $\psi_E$ does not agree with $\psi_A$ on the loop or bridge where $\psi_A$ is assigned, and does not agree with $\psi_B$ on the loop or bridge where $\psi_B$ is assigned.  It follows that the loop or bridge to which $\psi_E$ is assigned is not equal to the loop or bridge to which either $\psi_A$ or $\psi_B$ is assigned.  The third function is therefore the only function to achieve the minimum at a point of the loop or bridge where $\psi_E$ is assigned.

The cases where we have omitted some functions from $\cB$ follow in the same way as Theorem~\ref{Thm:Switching}.  This is because the arguments in that case rely only on the slopes of the various functions involved along the bridges $\beta_{\ell}$ and $\beta_{\ell-1}$.  This ``local'' information remains unchanged in our present case, so the result follows.
\end{proof}

\bibliography{math}

\end{document}